\documentclass{amsart}

\usepackage{amsmath,amssymb,amsthm,amscd,mathrsfs}
\usepackage{mathtools}
\usepackage[all]{xy}  \SelectTips{eu}{} \xyoption{line}
\usepackage[usenames,dvipsnames,svgnames,table]{xcolor}
\usepackage[linkbordercolor=BrickRed, citebordercolor=OliveGreen]{hyperref}
\usepackage{amsrefs}
\usepackage{marginnote}
\usepackage[british]{babel}

\usepackage{ulem}

\usepackage{tikz}
\usepackage{tikz-cd}

\theoremstyle{plain}
\newtheorem{theorem}{Theorem}[section]

\newtheorem{lemma}[theorem]{Lemma}

\newtheorem{proposition}[theorem]{Proposition}
\newtheorem{corollary}[theorem]{Corollary}

\theoremstyle{definition}
\newtheorem{definition}[theorem]{Definition}

\newtheorem{example}[theorem]{Example}

\newtheorem{ipg}[theorem]{}

\theoremstyle{remark}
\newtheorem{remark}[theorem]{Remark}

\usepackage{eucal}

\newcommand{\A}{\mathcal{A}}
\newcommand{\B}{\mathcal{B}}
\newcommand{\C}{\mathcal{C}}
\newcommand{\D}{\mathcal{D}}

\newcommand{\F}{\mathcal{F}}

\newcommand{\I}{\mathcal{I}}
\newcommand{\K}{\mathcal{K}}

\newcommand{\M}{\mathcal{M}}

\newcommand{\R}{\mathcal{R}}

\newcommand{\X}{\mathcal{X}}
\newcommand{\Y}{\mathcal{Y}}
\newcommand{\W}{\mathcal{W}}

\newcommand{\Obj}{\textnormal{Obj}\mspace{1mu}}

\DeclareMathOperator{\Ext}{Ext}
\DeclareMathOperator{\ext}{ext} 
\DeclareMathOperator{\Hom}{Hom}
\DeclareMathOperator{\Tor}{Tor}
\DeclareMathOperator{\tor}{tor} 
\DeclareMathOperator{\colim}{colim}

\DeclareMathOperator{\im}{Im}
\DeclareMathOperator{\coker}{coker}

\DeclareMathOperator{\Filt}{Filt}

\DeclareMathOperator{\End}{End}

\DeclareMathOperator{\rad}{rad}

\DeclareMathOperator{\sk}{sk}
\DeclareMathOperator{\cosk}{cosk}
\DeclareMathOperator{\res}{res}

\newcommand{\tens}[1]{\mathbin{\mathop{\otimes}\limits_{#1}}}

\newcommand{\rightperp}[1]{#1^{\perp}}
\newcommand{\leftperp}[1]{{}^\perp #1}
\newcommand{\class}{\mathcal}
\newcommand{\op}{^\mathrm{o}}
\newcommand{\id}{\mathrm{id}}
\newcommand{\Id}{\mathrm{Id}}

\begin{document}
\emergencystretch 3em

\title[Linear Reedy categories, qh algebras and model structures]%
{Linear Reedy categories, quasi-hereditary algebras and model structures}

\author{Georgios Dalezios}

\address{Department of Mathematics, Aristotle University of Thessaloniki, Thessaloniki 54124, Greece}
\email{gdalezios@math.uoa.gr}

\author{Jan \v{S}{\fontencoding{T1}\selectfont \v{t}}ov\'i\v{c}ek}

\address{Jan {\v S}{\v{t}}ov{\'{\i}}{\v{c}}ek, Charles University,
Faculty of Mathematics and Physics, Department of Algebra,
Sokolovsk\'a 83, 186 75 Praha, Czech Republic}

\email{stovicek@karlin.mff.cuni.cz}

\thanks{For the first-named author, the research project is implemented in the framework of H.F.R.I call ``Basic research Financing (Horizontal support of all Sciences)'' under the National Recovery and Resilience Plan “Greece 2.0” funded by the European Union – NextGenerationEU (H.F.R.I. Project Number: 16785).
The second-named author was supported by the grant GA~\v{C}R 23-05148S from the Czech Science Foundation.
}

\begin{abstract}
We study linear versions of Reedy categories in relation with finite dimensional algebras and abelian model structures. We prove that, for a linear Reedy category $\C$ over a field, the category of left $\C$--modules admits a highest weight structure, which in case $\C$ is finite corresponds to a quasi-hereditary algebra with an exact Borel subalgebra. We also lift complete cotorsion pairs and abelian model structures to certain categories of additive functors indexed by linear Reedy categories, generalizing analogous results from the hereditary case.
\end{abstract}


\maketitle

\tableofcontents

\section{Introduction}
\label{sec:intro}

Reedy categories form a generalization of the category $\mathbf{\Delta}$, with objects all finite ordinals and morphisms the weakly monotone functions between them. Roughly speaking, a Reedy category is a small category $\mathcal{C}$ equipped with a degree function from its objects to natural numbers (or to some greater ordinal), such that every morphism in $\mathcal{C}$ factors uniquely as a morphism lowering the degree followed by a morphism raising the degree. One of the major aspects of Reedy categories, is a theorem of Kan on the existence of a Quillen model category structure on a category of diagrams (functors) with source a Reedy category and target a given model category \cite[Ch.~15]{Hir}. 

In this work, we consider an analogous class of categories which are linear over a field, that we call \textit{linear} Reedy categories.
For instance, if $k$ is a field then the $k$--linearization of the category $\mathbf{\Delta}$ is a $k$--linear Reedy category in our sense, but the class of $k$--linear Reedy categories is much richer than the class of $k$--linearizations of ordinary Reedy categories.  
The case of small preadditive categories enriched over a field is of interest to us also since it is closely related to finite dimensional algebras. Indeed, it is well known that to any finite dimensional algebra $A$ over a field $k$, we can associate a $k$--linear category $\C$ such that the category of left $\C$--modules, which are $k$--linear functors from $\C$ to $k$--vector spaces, is equivalent to the category of left $A$--modules. In this way, the module theory of the finite dimensional algebra $A$ is encoded in the $k$--linear category $\C$. 

Here we are interested, on the one hand, in understanding the structure and the representation theory of linear Reedy categories, especially those which are associated $k$--linear categories of finite dimensional algebras.
On the other hand, we want to obtain an analogue of the aforementioned result of Kan about lifting \textit{abelian} model structures to functor categories indexed by linear Reedy categories.

To describe the representation theoretic results of this paper in more detail, fix a $k$--linear Reedy category $\C$ (Definition~\ref{dfn:Reedy_cat}) and let $(\C,\K)$ be the category of $k$--linear functors from $\C$ to the category $\K$ of $k$-vector spaces (such functors will be called \textit{left $\C$--modules}). We exhibit a collection $\{\Delta_c\}_{c\in\Obj(\C)}$ of objects in $(\C,\K)$ that we call \textit{standard} left $\C$--modules (Definition~\ref{def:standards}) and we prove the following: 
\begin{itemize}
\item[-] Theorem~\ref{thm:Reedy_simples}: There is a bijection between the objects of $\C$ and the isomorphism classes of simple left $\C$--modules.
\item[-] Theorem~\ref{thm:Reedy_exceptional}: The family of standard left $\C$--modules satisfies certain $\Hom/\Ext^{1}$--vanishing properties which are akin to exceptional sequences when interpreted in the opposite order.
\item[-] Theorem~\ref{thm:proj_by_standards} and Corollary~\ref{cor:std_filtration_summands}: For any projective left $\C$--module $\C(c,-)$ there exists a short exact sequence $0\rightarrow \C_{<\alpha}(c,-)\rightarrow \C(c,-)\rightarrow \Delta_{c}\rightarrow 0$  in $(\C,\K)$ where $\C_{<\alpha}(c,-)$ admits a (continuous) filtration whose subquotients are isomorphic to standard left $\C$--modules of the form $\Delta_d$ for $d$ of degree smaller than that of $c$.
\end{itemize}

The above results show that the category of $k$--linear functors $(\C,\K)$ satisfies properties dual to those of highest weight categories, which were introduced by Cline, Parshall and Scott in \cite[\S3]{CPS}  and have been used extensively in representation theory and Lie theory. 

In the case of a \textit{finite} $k$--linear Reedy category $\C$, that is, $\C$ is hom-finite with finitely many objects, the above results together with \cite[Thm.~3.6]{CPS} show that $(\C,\K)$ is in fact equivalent to the module category of a \textit{quasi-hereditary algebra}. We prefer to give a detailed account of this fact by introducing the concept of \textit{Reedy algebras} (Definition~ \ref{def:Reedy_algebra}) which seems to be of independent interest. Reedy algebras are proved to be quasi-hereditary having an \textit{exact Borel subalgebra} in Theorem \ref{thm:Reedy_is_qh}. Exact Borel subagebras of quasi-hereditary algebras were introduced by Koenig~\cite{Koe} as certain directed subalgebras that better control the filtration of the regular module by the standard modules.
In a subsequent work~\cite{CDSnew}, Reedy algebras are even characterized as those quasi-hereditary algebras having a triangular decomposition (or Cartan decomposition) in the sense of Koenig~\cite{Koe2}. This among others provides us, in view of the results from~\cite{Koe,Koe2}, with classes of examples of finite $k$--linear Reedy categories which are not linearizations of classical Reedy categories.

The second part of the paper is concerned with \textit{abelian model structures}, as  introduced by Hovey \cite{hovey}. These are model structures on abelian categories where the cofibrations (resp., fibrations) are monomorphisms with cofibrant cokernels (resp., epimorphisms with fibrant kernels). 
Hovey's insight was that abelian model structures correspond bijectively to certain \textit{complete cotorsion pairs} in the underlying abelian category. A cotorsion pair in an abelian category is a pair of subcategories that are $\Ext^{1}(-,-)$--orthogonal to each other, while completeness of a cotorsion pair amounts to the existence of certain special approximations (short exact sequences) in the abelian category, in a similar fashion as one considers (co)fibrant replacements in a model category. From the point of view of (relative) homological algebra and representation theory, cotorsion pairs are much more widely used than model structures. Nevertheless, Hovey's results provide an interesting connection between the two seemingly different subjects. 
From this perspective, the question of lifting an abelian model structure to 
a category of functors
is closely related to the problem of lifting complete cotorsion pairs.

We prove that complete cotorsion pairs and abelian model structures lift canonically to functor categories indexed by $k$--linear Reedy categories in Theorem~\ref{thm:lifting_cot} and Theorem \ref{thm:Reedy_model} respectively.
Our results are a generalization of the main results of \cites{hj2016,odabasi} and this also explains the connection to the first part of this paper. Namely, it is proved in \textit{loc. cit.} that under mild assumptions a complete cotorsion pair can be lifted from an abelian category to functors indexed by left or right rooted quivers. In the finite case, such quivers model hereditary algebras and from our point of view, they form special cases of Reedy categories, see~Example~\ref{ex:rooted_quivers}. Since a finite quiver with relations whose associated $k$--linear category is Reedy forms a quasi-hereditary path algebra by Theorem~\ref{thm:Reedy_is_qh}, our Theorems  \ref{thm:lifting_cot} and \ref{thm:Reedy_model} actually extend those of \cites{hj2016,odabasi} to functors indexed by a large class of diagrams of quasi-hereditary shape, see Corollary~\ref{cor:generalize_HJ} and Example~\ref{ex:generalize_HJ}.

We also note that analogous results were recently obtained by Holm and J\o rgensen \cite{hj2019} for functor categories indexed by  ``self-injective quivers with relations'',
but these are somewhat perpendicular to our results, 
since a $k$--linear Reedy category with finitely many objects has finite global dimension.

\smallskip

The layout of the paper is as follows. In Section \ref{sec:preliminaries} we recall preliminaries on functor categories and abelian model structures. In Section \ref{sec:Reedy_cats} we define linear versions of direct, inverse and Reedy categories and we give some examples. In Section \ref{sec:QH}, for a linear Reedy category $\C$, we define the class of standard $\C$--modules and prove the results displayed above, culminating in a proof that Reedy algebras are quasi-hereditary (Theorem~\ref{thm:Reedy_is_qh}). Sections \ref{sec:Fun} and \ref{sec:cot_pairs} form the technical heart for lifting complete cotorsion pairs in the linear Reedy setting (Theorem \ref{thm:lifting_cot}), and Section \ref{sec:Reedy_model} contains the Reedy abelian model structure.

\smallskip

\textbf{Relations to the existing literature.}
We point out that Reedy categories have been studied in various forms, most notably in \cites{Angelt, Moerd, riehl, Shul}. In particular, our definition of a linear Reedy category is essentially the same as in \cite[Def.~9.11]{Shul}. Some of the results in Section \ref{sec:Fun} follow the footsteps of the sources listed above, especially of \cite{riehl}. In this paper we do not aim for great generality, we rather focus on the connection to highest weight structures and Ext-orthogonal classes.

\section{Preliminaries}
\label{sec:preliminaries}
In this section we assume that the following data is given:

\begin{itemize}
\item[•] $k$ is a field and $\mathcal{K}$ is the category of $k$--vector spaces.
\item[•] $\M$ is a bicomplete $k$--linear abelian category, i.e., $\M$ is a bicomplete abelian category, its hom-sets carry a structure of $k$-vector spaces such that the composition of morphisms in $\M$ is $k$--bilinear. For example, $\M$ can be the category of left $A$--modules over a $k$--algebra $A$.
\item[•] $\C$ is a small $k$-linear category.
\item[•] $\M^{\C}$ is the category of $k$--linear functors from $\C$ to $\M$. Sometimes the category $\M^{\C}$ will be denoted by $(\C,\M)$.
\end{itemize}

\begin{ipg}\textbf{(Enriched categories)}
\label{subsec:enriched}
The category $\mathcal{K}$ is a closed symmetric monoidal category with monoidal product the tensor product of $k$--vector spaces, which we denote by $\otimes$, and unit object the regular module $k$. In the language of enriched categories \cite{Kelly}, the categories $\mathcal{C}$ and $\M$, as well as $\M^{\C}$,  are enriched over $(\mathcal{K},\otimes,k)$ (also called $\mathcal{K}$--categories). This roughly means that all concepts involved here, such as functors, natural transformations and so on, are required to respect the underlying structure imposed by $(\mathcal{K},\otimes,k)$ whenever this exists.  

Recall also that since $\M$ is bicomplete, there is a closed action of $\mathcal{K}$ on $\M$. That is, there are functors
\[
\otimes\colon\K\times\M\to\M
\qquad\textrm{and}\qquad
\hom\colon\K\op\times\M\to\M
\]
and isomorphisms natural in all of $V\in\K$ and $M,N\in\M$,
\[
\K(V,\M(M,N))\cong\M(V\otimes M,N)\cong\M(M,\hom(V,N)).
\]
More specifically, if $V=k^{(I)}$ is a $k$--vector space with basis $I$, then $V\otimes M\cong M^{(I)}$ and $\Hom(V,M)\cong M^I$ as objects of $\M$, and the functors are defined in the obvious way on morphisms.
In the language of enriched category theory, $\otimes$ is called the \textit{tensor product} and $\Hom$ is called the \textit{cotensor product}, \cite[\S3.7]{Kelly}.

Note also that for any $M,N\in\M$, the counit $\M(M,N)\otimes M\to N$ of the adjunction $(-\otimes M,\M(M,-))\colon\K\rightleftarrows\M$ is an abstraction of the evaluation morphism $\phi\otimes m\mapsto\phi(m)$ from the case where $\M$ is the category of left $A$--modules over a $k$--algebra $A$.
Similarly, the unit $M\to\hom(\M(M,N),N)$ of the adjunction $(\M(-,N),\hom(-,N))\colon\K\op\rightleftarrows\M$ abstracts the evaluation morphism $m\mapsto(f\mapsto f(m))$.
\end{ipg}

\begin{ipg}\textbf{(Rings with several objects)}
\label{subsec:ringoids}
Now we recall some facts from Mitchell \cite{Mitchell} on the category of $k$--linear functors $\M^{\C}$, which can be thought of as a category of left $\class C$--modules with values in $\class M$.

In the special case where $\M=\K$, we can define a tensor product bifunctor,
\[-\tens{\class C}- \colon (\class C\op,\class K)\times (\class C,\class K)\rightarrow \class K.\]
The bifunctor is defined on objects as follows: If $G\in (\class C^{\mathrm{o}},\class K)$ and $F\in(\class C,\class K)$, then
\[G\tens{\class C}F:=\bigoplus_{c\in\mathrm{Obj}(\class C)} \left(G(c)\otimes_{k}F(c)\right) /S,\]
where $S$ is the subspace generated by elements of the form $G(r)(x)\otimes y-x\otimes F(r)(y)$, for all $r\colon a\rightarrow b, x\in G(b)$, and $y\in F(a)$. If $\phi\colon G\rightarrow G'$ is morphism in $(\class C^{\mathrm{o}},\class K)$ and $\psi\colon F\rightarrow F'$ is a morphism in $(\class C,\class K)$, then we define a morphism
\[\phi\tens{\class C}\psi\colon G\tens{\class C}F\rightarrow G'\tens{\class C} F',\]
as the one induced on the quotient by the diagonal morphism
 \[\bigoplus\limits_{c\in\mathrm{Obj}(\class C)}\phi_{c}\otimes_{k}\psi_{c}\colon \bigoplus\limits_{c\in\mathrm{Obj}(\class C)}G(c)\otimes_{k}F(c)\rightarrow \bigoplus\limits_{c\in\mathrm{Obj}(\class C)}G'(c)\otimes_{k}F'(c).\]
The tensor product functor commutes with colimits in both variables and for all objects $c\in\class C$, $F\in(\class C,\class K)$, and $G\in(\class C^{\mathrm{o}},\class K)$, it satisfies the formula ${\class C}(-,c)\otimes_{\class C}F\cong F(c)$, see \cite{O-R}. 

To define an analogous tensor product bifunctor
\[-\tens{\class C}- \colon (\class C\op,\class K)\times (\class C,\class M)\rightarrow \class M,\]
the language of enriched categories will be useful, since we can employ weighted limits and colimits. 
Given functors $X\in\M^{\class C}$, $U\in\class K^{\C}$ and $W\in\class K^{\C^{\mathrm{o}}}$, we define the \textit{colimit of $X$ weighted by $W$}, and the \textit{limit of $X$ weighted by $U$}, respectively, by the following (enriched over $\mathcal{K}$) coend and end formulas: 
\begin{equation}
W\tens{\class C} X:=\int^{c\in\C} W(c)\otimes X(c) \,\,\,\,\,\,\,\,\,\,\ \mbox{and}  \,\,\,\,\,\,\,\,\,\,\  \hom_{\class C}(U,X):=\int_{c\in\C}\Hom_{\class K}(U(c),X(c)).
\nonumber
\end{equation}
Both $W\otimes_{\class C} X$ and $\hom_{\class C}(U,X)$ are objects of $\M$;
for definitions and existence of these formulas the reader may consult \cite[\S\S7.4 and 7.6]{riehl_book}.
The tensor product functor again commutes with colimits in both variables and for all objects $c\in\class C$, $F\in(\class C,\class M)$, and $G\in(\class C^{\mathrm{o}},\class K)$, it satisfies the formula ${\class C}(-,c)\otimes_{\class C}F\cong F(c)$.
Analogously, the internal hom functor $\hom_\C$ sends colimits to limits in the first variable, preserves limits in the second variable, and it satisfies the formula $\hom_\C({\class C}(c,-),F)\cong F(c)$.

Since for any $X\in\M^\C$, the functors
\[ -\tens{\C}X\colon\K^\C\to\M \qquad\mbox{and}\qquad \hom_\C(-,X)\colon(\K^\C)^\mathrm{o}\to\M \]
are right and left exact, respectively, we can construct the corresponding derived functors (left and right, respectively) using the projective resolutions in $\K^\C$. In order to differentiate them from usual Tor and Ext functors in abelian categories, we denote them by
\[ \tor_n^{\C}(-,X)\colon\K^\C\to\M \qquad\mbox{and}\qquad \ext^n_\C(-,X)\colon(\K^\C)^\mathrm{o}\to\M. \]
Needless to say, if $\M$ is the category of left $A$--modules over a $k$--algebra $A$, then $\tor_n^{\C}(-,X)$ coincides with $\Tor_n^{\C}(-,X)$, where in the latter functor we interpret $X$ as the underlying $k$--linear functor $\C\to\K$, forgetting the $A$--module action. Analogously, $\ext^n_{\C}(-,X)$ coincides with $\Ext^n_{\C}(-,X)$ in that case.

In general, however, one should be a little careful as to when these derived functors give rise to long exact sequences. If $0\to U\to V\to W\to 0$ is a short exact sequence in $\K^{\C\op}$ and $X\in\M^\C$, we always get a long exact sequence
\[ \cdots\to \tor_1^\C(V,X)\to \tor_1^\C(W,X)\to U\tens{\C}X\to V\tens{\C}X\to W\tens{\C}X\to 0 \]
and similarly, if $0\to U\to V\to W\to 0$ is a short exact sequence in $\K^{\C}$, we obtain a long exact sequence
\[ 0\to \hom_\C(X,U)\to \hom_\C(X,V)\to \hom_\C(X,W)\to \ext^1_\C(X,U)\to \cdots \]
This is standard -- one constructs a degreewise split exact sequence $0\to P^U_\bullet\to P^V_\bullet\to P^W_\bullet\to 0$ of projective resolutions in $\K^{\C\op}$ or $\K^\C$ and applies $-\otimes_{\C}X$ or $\hom_\C(X,-)$ to it, respectively. More care is needed in the case when $0\to X\to Y\to Z\to 0$ is a short exact sequence in $\M^\C$ and $\M$ does not have exact coproducts or products.

\begin{definition}\label{def:tens-and-hom-exactness}
A short exact sequence $0\to X\to Y\to Z\to 0$ in $\M^\C$ will be called $\coprod$--\textit{exact} if for any collections of sets $I_c$ indexed by objects $c\in\C$, the coproduct
\[ 0 \to \coprod_{c\in\C} X(c)^{(I_c)} \to \coprod_{c\in\C}Y(c)^{(I_c)} \to \coprod_{c\in\C}Z(c)^{(I_c)} \to 0 \]
is exact in $\M$. Dually, $0\to X\to Y\to Z\to 0$ is called $\prod$--\textit{exact} if for any collections of sets $I_c$, $c\in\C$, the product
\[ 0 \to \prod_{c\in\C} X(c)^{I_c} \to \prod_{c\in\C}Y(c)^{I_c} \to \prod_{c\in\C}Z(c)^{I_c} \to 0 \]
is exact in $\M$\end{definition}

\begin{lemma}\label{lem:long-exact-seq}
Let $W\in\K^{\C\op}$ and $0\to X\to Y\to Z\to 0$ be a $\coprod$--exact short exact sequence in $\M^\C$. Then there is a long exact sequence
\[ \cdots\to \tor_1^\C(W,Y)\to \tor_1^\C(W,Z)\to W\tens{\C}X\to W\tens{\C}Y\to W\tens{\C}Z\to 0. \]
Dually, given $U\in\K^{\C}$ and $0\to X\to Y\to Z\to 0$ a $\prod$--exact short exact sequence in $\M^\C$, there is a long exact sequence
\[ 0\to \hom_\C(U,X)\to \hom_\C(U,Y)\to \hom_\C(U,Z)\to \ext^1_\C(U,X)\to \cdots \]
\end{lemma}

\begin{proof}
Let $P_\bullet\to W$ be a projective resolution of $W$ in $\K^{\C\op}$. We can construct it so that for each $n\ge 0$, the component $P_n$ is a coproduct of representable functors $P_n = \coprod_{c\in I_c^n}\C(-,c)^{(I_c^n)}$. Then $P_n\otimes_{\C}X\cong\coprod_{c\in\C} X(c)^{(I_c^n)}$ and similarly for $Y$ and~$Z$. Appealing to the $\coprod$--exactness, $0\to P_\bullet\otimes_{\C} X\to P_\bullet\otimes_{\C} Y\to P_\bullet\otimes_{\C} Z\to 0$ is an exact sequence of complexes in $\M$. This yields the desired long exact sequence of homologies. The other case is dual.
\end{proof}
\end{ipg}

\begin{ipg}\textbf{(Radicals)}
\label{subsec:radicals}
The radical of a small $k$--linear category $\C$, is the subfunctor $\rad_{\class C}(-,-)$ of ${\class C}(-,-)$ defined by:
\[\rad_{\class C}(c,d):=\{f\colon c\rightarrow d\,\, |\,\, \forall g\colon d\rightarrow c; \,\, \mbox{the map}\,\, \mathrm{id}_{c}-g\circ f\,\, \mbox{is invertible}\}.\]
If $c$ and $d$ are objects in $\class C$ such that the endomorphism rings $\mathrm{End}_{\class C}(c)$ and $\mathrm{End}_{\class C}(d)$ are local, then $\rad_{\class C}(c,d)$ is isomorphic to the $k$--vector space of all non-isomorphisms from $c$ to $d$, see \cite[Prop.~A.3.5]{elements} for a proof. In particular, if $c\ncong d$ then $\rad_{\class C}(c,d)\cong\Hom_{\class C}(c,d)$ as $k$--vector spaces.
\end{ipg}

\begin{ipg}\textbf{(Simple functors)}
\label{subsec:simple_functors_prelim}
In addition to the data given in the beginning of this section, we in some cases assume that for all $c\in\C$ the endomorphism ring $\mathrm{End}_{\class C}(c)$ is local. Then it is well known (but see also \cite[Prop.~17.19]{AF92}) that, for each object $c$, the contravariant functor $S^{c}:=\Hom_{\C}(-,c)/\rad_{\C}(-,c)$ and the covariant functor $S_{c}:=\Hom_{\C}(c,-)/\rad_{\C}(c,-)$ are simple objects of $\K^{\C}$; in fact all simple covariant/contravariant functors are of this form. 
\end{ipg}

\begin{ipg}\textbf{(Quivers)}
\label{quivers}
A \textit{quiver} $Q$ is a quadruple $(Q_0,Q_1,s,t)$ where $Q_0,Q_1$ are sets and $s,t\colon Q_1\to Q_0$ are maps.
Elements of $Q_0$ are called \textit{vertices} and elements of $Q_1$ are called \textit{arrows}.
An arrow $\alpha\in Q_1$ with $s(\alpha)=i$ and $t(\alpha)=j$ is usually depicted as $\alpha\colon i\rightarrow j$, the vertex $i$ is called the \textit{source} of $\alpha$ and the vertex $j$ the \textit{target} of $\alpha$.
Thus, a quiver is essentially a directed graph with possible multiple arrows or loops with the same source and target.

A \textit{path} of length $n\geqslant 1$ is a formal composite $p=\alpha_n\circ\alpha_{n-1}\circ\cdots\circ\alpha_1$ of arrows $\alpha_i$ with $s(\alpha_{i+1})=t(\alpha_i)$ for all $i=1,...,n-1$. To any vertex $i\in Q_0$ we associate a \textit{trivial path} $e_i$ of length zero having $i$ as its source and target; we assume that trivial paths act like identities when composing paths. Given a field $k$, a \textit{relation} in $Q$ is a formal $k$--linear combintation of paths having the same source and target. For a specified set $I$ of relations in $Q$, we call the pair $(Q,I)$, also denoted by $Q_I$, a \textit{quiver with relations}.

If $(Q,I)$ is a quiver with relations over a field $k$, we can construct  a $k$--linear category $kQ_I$; the so-called $k$--linearization of $(Q,I)$.
Its objects are the vertices of $Q$ and, for all vertices $q, q'$, $\Hom_{kQ_I}(q,q')$ is the free $k$--vector space on the set of paths from $q$ to $q'$ modulo the ideal generated by the relations of $I$.  That is, we factor out all linear combinations of paths from $q$ to $q'$ which are of the form $\sum_{j=1}^n \lambda_j \cdot p_j r_j q_j$ for some $n\ge 0$, relations $r_j\in I$, scalars $\lambda_j\in k$ and paths $p_j,q_j$. The composition rule is induced by concatenation of paths.
In the special case where $I=\emptyset$, we denote the free path category just by $kQ$.
\end{ipg}

\begin{ipg}\textbf{(Cotorsion pairs)}
\label{subsec:cotorsion_pairs}
We recall a key concept going back to~\cite{Salce-cotorsion}. For a class $\X$ in an abelian category $\M$, we denote classes orthogonal with respect to the Yoneda Ext-functor:
\begin{align*}
\rightperp{\X}&:=\{M\in\class M\,\,|\,\, \forall X\in\X,\,\, \Ext^{1}_{\M}(X,M)=0\}, \\
\leftperp{\X}&:=\{M\in\class M\,\,|\,\, \forall X\in\X,\,\, \Ext^{1}_{\M}(M,X)=0\}.
\end{align*}

\begin{definition}
\label{def:cot_pairs}
A pair of subcategories $(\X,\Y)$ in an abelian category $\M$ is a \textit{cotorsion pair} if $\X=\leftperp{\Y}$ and $\Y=\rightperp{\X}$. It is called \textit{complete} if for any object $M$ in $\M$, there exist short exact sequences $0\rightarrow M\rightarrow Y\rightarrow X \rightarrow 0$ and  $0\rightarrow Y'\rightarrow X'\rightarrow M \rightarrow 0$ with $X,X' \in \X$ and $Y,Y' \in\Y$.  It is called \textit{hereditary} if for all $X\in\X$, $Y\in\Y$ and $i\geqslant 1$, we have $\Ext_{\class M}^{i}(X,Y)=0$.
\end{definition}

\begin{remark}\label{rem:cotorsion-pairs-and-(co)products}
Clearly, the classes $\X$ and $\Y$ are closed under extensions in $\M$.
Unlike covariant Hom--functors, covariant $\Ext^1$--functors might not commute with products if products are not exact in $\M$. Still, the right hand side of a cotorsion pair is closed under all products that exist in $\M$, see~\cite[Prop.~8.3]{CF07} or \cite[Cor.~A.2]{CoupekStovicek}. Similarly, the left hand side of a cotorsion pair is closed under all coproducts that exist in $\M$.
\end{remark}

Hereditary cotorsion pairs are characterized under mild assuptions by the classes having more closure properties. We call a class $\X\subseteq\M$ \textit{generating} if each $M\in\M$ is a quotient of some $X\in\X$ and a class $\Y\subseteq\M$ \textit{cogenerating} if each $M\in\M$ is embeds into some $Y\in\Y$. These two conditions are always satisfied if $(\X,\Y)$ is a complete cotorsion pair.

\begin{lemma}
\label{lem:hereditary}
Let $(\A,\B)$ be a cotorsion pair such that $\A$ is generating and $\B$ is cogenerating in $\M$. Then the following are equivalent:
\begin{enumerate}
\item[(i)] $(\A,\B)$ is hereditary,
\item[(ii)] $\A$ is closed under kernels of epimorphisms,
\item[(iii)] $\B$ is closed under cokernels of monomorphisms.
\end{enumerate}
\end{lemma}

\begin{proof}
This is \cite[Lemma 6.17]{Jan_survey} (see also~\cite[Lemma 4.25]{SaSt11} and, for complete cotorsion pairs, \cite[Prop.~1.1.11]{Becker}).
\end{proof}

Given a class of objects $\X$ in $\M$, we denote by $\mathrm{Mono}(\X)$ (resp., $\mathrm{Epi}(\X)$) the monomorphisms (resp., epimorphisms) in $\M$ with cokernel (resp., kernel) in $\X$ and call such morphisms $\X$--\textit{monomorphisms} (resp., $\X$--\textit{epimorphisms}). 
The following two lemmas are then mostly a variation of \cite[Ch.VIII, Lemma 3.1]{Bel}.

\begin{lemma}\label{lem:Ext-orthogonal}
If $\M$ is abelian, $\X$ is a class of objects in $\M$, and $Y\in\M$, then
\begin{enumerate}
\item $Y\in\rightperp{\X}$ if and only if $\M(f,Y)$ is surjective for any $f\in\mathrm{Mono}(\X)$,
\item $Y\in\leftperp{\X}$ if and only if $\M(Y,f)$ is surjective for any $f\in\mathrm{Epi}(\X)$.
\end{enumerate}
\end{lemma}

\begin{proof}
We will only prove the first part, the second is similar. The `only if' part follows immediately if we apply $\M(-,Y)$ to the short exact sequence $0\to F\overset{f}\to E\to\coker(f)\to 0$. Conversely, suppose we have $X\in\X$ and a short exact sequence $0\to Y\overset{f}\to E\to X\to 0$. Then $f\in\mathrm{Mono}(\X)$ and if $\M(f,Y)$ is surjective, the sequence must split.
\end{proof}

\begin{remark}\label{rem:conditional-exactness-of-(co)products}
An immediate consequence of the lemma is that if $(\X,\Y)$ is a cotorsion pair in $\M$ and $\Y$ is cogenerating, then a coproduct of $\X$--monomorphisms is again an $\X$--monomorphism. This is despite the fact that $\M$ might not have exact coproducts.
Indeed, if $f_i\colon W_i\hookrightarrow Z_i$ is a collection of $\X$-monomorphism such that $\coprod W_i$ and $\coprod Z_i$ exist, then $\M(\coprod f_i,Y)\cong\prod\M(f_i,Y)$ is a surjective map of abelian groups for each $Y\in\Y$. Since $Y$ is cogenerating, there is a monomorphism of the form $j\colon\coprod W_i\hookrightarrow Y$ in $\M$ with $Y\in \Y$. This has a preimage under $\M(\coprod f_i,Y)$, so there is a morphism $k\colon\coprod Z_i\to Y$ such that $k\circ\coprod f_i= j$. It follows that $\coprod f_i$ is a monomorphism whose cokernel is isomorphic to $\coprod\coker(f_i)$, so belongs to $\X$ by Remark~\ref{rem:cotorsion-pairs-and-(co)products}.

Dually, if $\X$ is generating in $\M$, the class of $\Y$--epimorphisms is closed under products (even if $\M$ does not have exact products).
\end{remark}

On the other hand, we have the following observation by Hovey~\cite{hovey}, where we need to introduce corresponding terminology first.

\begin{definition}
Given two morphisms $l$ and $r$ is a category, we say that $l$ (resp., $r$) has the \textit{left} (resp., \textit{right}) \textit{lifting property} with respect to the morphism $r$ (resp., $l$), if for any commutative square in $\M$ given by the solid arrows:
\[
\xymatrix@C=2pc{
A \ar[r] \ar[d]_{l} & C \ar[d]^-{r} \\
B \ar[r] \ar@{-->}^-{h}[ur] & D
}
\]
there exists a dotted arrow $h$, as indicated, such that the two triangles commute.
\end{definition} 

\begin{lemma}\label{lem:Ext-to-box-orthogonal}
If $(\X,\Y)$ is a cotorsion pair in an abelian category $\M$ and we have $l\in\mathrm{Mono}(\X)$ and $r\in\mathrm{Epi}(\Y)$, then $l$ has the left lifting property with respect to $r$.
\end{lemma}

\begin{proof}
The idea appears in the proof of~\cite[Prop.~4.2]{hovey}, and the statement is explicitly proved in~\cite[Ch.VIII, Lemma~3.1]{Bel}, \cite[Lemma~6.1.10]{enochs-jenda2} or~\cite[Lemma~5.14]{Jan_survey}.
\end{proof}
\end{ipg}

In fact, Hovey's insight in \cite{hovey} (presented in a more crystallized form \cite[Thm.~5.13]{Jan_survey}) was that there was a much tighter relation between complete cotorsion pairs $(\X,\Y)$ and pairs of classes of the form $(\mathrm{Mono}(\X),\mathrm{Epi}(\Y))$.

\begin{definition}\label{def:wfs}
A \textit{weak factorization system} in a category $\M$, is a pair $(\class L,\R)$ of classes of morphisms in $\M$ such that:
\begin{itemize}
\item[-] $\mathcal{L}$ is precisely the class of morphisms in $\M$ that have the left lifting property with respect to all morphisms in $\mathcal{R}$.
\item[-] $\R$ is precisely the class of morphisms in $\M$ that have the right lifting property with respect to all morphisms in $\mathcal{L}$.
\item[-] Any morphism $h$ in $\M$ admits a factorization as $h=f\circ g$, where $g$ is a morphism in $\mathcal{L}$ and $f$ is a morphism in $\R$.
\end{itemize}
\end{definition}

\begin{definition}
\label{def:abelian_wfs}
A weak factorization system $(\mathcal{L},\mathcal{R})$ in an abelian category $\M$ is called \textit{abelian} if the following hold:
\begin{itemize}
\item[(i)] A morphism $f$ is in $\mathcal{L}$ if and only if $f$ is a monomorphism and the morphism  $0\rightarrow\coker(f)$ is in $\mathcal{L}$.
\item[(ii)] A morphism $g$ is in $\mathcal{R}$ if and only if $g$ an epimorphism and the morphism $\ker(g)\rightarrow 0$ is in $\mathcal{R}$.
\end{itemize}
\end{definition}

If $\M$ is an abelian category, for a class $\A$ of morphisms in $\M$ we denote by 
$\coker(\mathcal{A})$ (resp., $\ker(\mathcal{A})$), the class of objects in $\M$ isomorphic to $\coker(f)$ (resp., $\ker(f)$) for some morphism $f$ in $\A$.

\begin{proposition}[{\cite{hovey}, \cite[Thm.~5.13]{Jan_survey}}]
\label{prop:mappings}
Let $\M$ be an abelian category. The mappings:
\[
(\mathcal{L},\mathcal{R}) \mapsto (\coker(\mathcal{L}),\ker(\R)) \qquad\mbox{and}\qquad (\X,\Y)\mapsto (\mathrm{Mono}(\X),\mathrm{Epi}(\Y))\]
define mutually inverse bijections between abelian weak factorization systems $(\mathcal{L},\mathcal{R})$  and complete cotorsion pairs $(\X,\Y)$ in $\M$.
\end{proposition}

\begin{ipg}\textbf{(Abelian model structures)}
\label{subsec:abelian_models}
We recall the main aspects of the theory of abelian model structures, introduced by Hovey~\cite{hovey}. In order to state the main result of \cite{hovey} we need to briefly recall a few definitions from the theory of Quillen model categories; references for this material include \cites{hoveybook,Hir,riehl_book}. 

We now recall the definition of a Quillen model structure.

\begin{definition}
\label{def:model_category}
Let $\M$ be a complete and cocomplete category. A \textit{model structure} on $\M$ consists of three classes of morphisms, $\mathit{cof}, \mathit{fib}$ and $\mathit{weak}$, which are called \textit{cofibrations}, \textit{fibrations} and \textit{weak equivalences}, respectively, such that the following hold: 
\begin{enumerate}	
\item[(i)] The pairs  $(\mathit{cof}\cap\mathit{weak},\mathit{fib})$ and $(\mathit{cof},\mathit{weak}\cap\mathit{fib})$ are weak factorization systems in $\M$. 
\item[(ii)] The class $\mathit{weak}$ is closed under retracts and given a composable pair of morphisms $f$ and $g$ in $\M$, if two of $f, g$ and $f\circ g$ belong to the class $\mathit{weak}$, then so does the third.
\end{enumerate}
\end{definition}

The next definition is essentially due to Hovey~\cite{hovey}. We present them with the formalism introduced in \cite{Jan_survey}.

\begin{definition}
\label{def:abelian_model_category}
A model structure on a complete and cocomplete abelian category $\M$ is called an \textit{abelian model structure} if its weak factorzation systems $(\mathit{cof}\cap\mathit{weak},\mathit{fib})$ and $(\mathit{cof},\mathit{weak}\cap\mathit{fib})$ are abelian in the sense of Definition~\ref{def:abelian_wfs}.
\end{definition}

Notice that, if $\M$ is an abelian model structure and we apply Proposition~\ref{prop:mappings} to its abelian weak factorization systems $(\mathit{cof}\cap\mathit{weak},\mathit{fib})$ and $(\mathit{cof},\mathit{weak}\cap\mathit{fib})$, we get that the classes $\C=\coker(\mathit{cof})$ and $\F=\ker(\mathit{fib})$ of cofibrant, resp., fibrant objects in $\M$ and the class $\W$ of weakly trivial (i.~e.\ weakly equivalent to zero) objects
satisfy the following conditions:

\begin{enumerate}	
\item[(i)] The pairs $(\C,\W\cap\F)$ and $(\C\cap\W,\F)$ are complete cotorsion pairs in $\M$.
\item[(ii)] The class $\W$ is closed under direct summands and given a short exact sequence $0\to K\to L\to M\to 0$ in $\M$, if two of $K, L$ and $M$ belong in the class $\W$, then so does the third.
\end{enumerate}

Conversely, suppose we are given a triple $(\C,\W,\F)$ of classes of objects in $\class M$ satisfying the conditions above. 
Then Proposition~\ref{prop:mappings} produces an abelian model structure on $\M$. This is the form in which Hovey's original result  \cite[Thm.~2.2]{hovey} is stated and we refer for details there or to~\cite{Jan_survey}.

If $\class M$ is an abelian category with an abelian model structure, where $\C, \F$ and $\W$ denote the classes of cofibrant, fibrant and weakly trivial objects respectively, we abbreviate by saying that $(\C,\W,\F)$ is a \textit{Hovey triple} on $\class M$, which in addition is called \textit{hereditary} in case the associated  complete cotorsion pairs are hereditary. We point out that hereditariness is important in order to equip the homotopy category of $\M$ with a triangulated structure, for this the reader may consult \cite{Gil_survey} as we will not expand on it here.
\end{ipg}

\begin{ipg}\textbf{(Filtrations)}
In this part, let $\M$ be a cocomplete abelian category.
Our next steps are inspired by~\cite[\S4]{PositselskiRosicky}. 
First we define the notion of filtrations and filtered objects. As in~\cite{PositselskiRosicky}, we will not assume any exactness conditions on direct limits since we wish to apply our results to categories like $\M^\C$, where $\C$ is a small $k$-linear category.

\begin{definition}[{\cite[Def.~4.3]{PositselskiRosicky}}]
\label{def:filt}
Let $\sigma$ be an ordinal number and $\X$ a class of objects in $\M$. A well ordered direct system $\mathcal{D}:=
(M_{\alpha}\, |\, i_{\alpha,\beta}\colon M_{\alpha}\rightarrow M_{\beta})_{\alpha<\beta\leqslant\sigma}$ of objects and morphisms in $\M$ is called an $\X$--\textit{filtration} of an object $M\in\M$, if the following hold:
\begin{itemize}
\item[(i)] For each limit ordinal $\beta\leqslant\sigma$ we have $M_{\beta}=\colim_{\alpha<\beta}M_{\alpha}$.
\item[(ii)] For all $\alpha<\sigma$ the morphism $i_{\alpha,\alpha+1}$ is a monomorphism.
\item[(iii)] $M_0=0$ and $M=M_{\sigma}$. 
\item[(iv)] For all $\alpha<\sigma$ the cokernel of the monomorphism $i_{\alpha,\alpha+1}$ belongs in $\X$.
\end{itemize}
Property (i) is sometimes referred to as \textit{continuity} of the direct system or, in the homotopy-theoretic jargon, that the direct system is a \textit{$\sigma$-sequence} (see \cite[\S10.2]{Hir} or \cite[\S2.1.1]{hoveybook}).
We denote by $\Filt(\X)$ the class of objects of $\M$ that admit $\X$--filtrations.
\end{definition}

The key result is the so-called Eklof Lemma. It is known under this name for module categories~\cite[Lemma~6.2]{Gobel-Trlifaj}, while a version for opposite categories of module categories is sometimes called the Lukas Lemma \cite[Lemma~6.37]{Gobel-Trlifaj}.
Here we recover with a simpler proof a general version of this result from~\cite[Lemma 4.5]{PositselskiRosicky}, as well as \cite[Lemmas~6.6 and 6.8]{hj2016}, which unifies and vastly generalizes the cases for modules.

\begin{proposition}[The Eklof/Lukas Lemma] \label{prop:eklof}
Let $\M$ be a cocomplete abelian category and $\Y\subseteq\M$ be a class of objects. Then any $\leftperp{\Y}$--filtered object belongs to $\leftperp{\Y}$; so $\Filt(\leftperp{\Y})\subseteq\leftperp{\Y}$ in symbols.
\end{proposition}

\begin{proof}
Let $(M_{\alpha}\, |\, i_{\alpha,\beta}\colon M_{\alpha}\rightarrow M_{\beta})_{\alpha<\beta\leqslant\sigma}$ be an $\leftperp{\Y}$--filtration of some $M\in\M$. By Lemma~\ref{lem:Ext-orthogonal}(2), we need to prove that the morphism $0\to M$ has the left lifting property with respect to all $r\in\mathrm{Epi}(\Y)$. By the assumption and Lemma~\ref{lem:Ext-to-box-orthogonal}, we know that $i_{\alpha,\alpha+1}$ for each $\alpha<\sigma$ has the left lifting property with respect to all $r\in\mathrm{Epi}(\Y)$. The conclusion follows from the standard fact that the class of morphisms having the left lifting property with respect to all $r\in\mathrm{Epi}(\Y)$ is closed under so-called transfinite compositions; see e.g.~\cite[Lemma 10.3.1]{Hir}.
\end{proof}

\begin{remark} \label{rem:filt}
As mentioned, there is in general no reason why, given an $\X$--filtration $(M_{\alpha}\, |\, i_{\alpha,\beta})_{\alpha<\beta\leqslant\sigma}$, the morphisms $i_{\alpha\beta}$ would be monomorphisms unless $\beta=\alpha+n$ for $n$ finite. We refer to \cite[Examples 4.4]{PositselskiRosicky} for particular examples where this is not the case.

However, if $\rightperp\X$ is a cogenerating class in $\M$, then $i_{\alpha\beta}$ will be monomorphisms for each $\alpha<\beta\le\sigma$. This happens e.g.\ when $\M$ has enough injectives.

Indeed, by the proof of Proposition~\ref{prop:eklof} (and in particular by \cite[Lemma 10.3.1]{Hir}), each $i_{\alpha\beta}\colon M_{\alpha}\rightarrow M_{\beta}$ has the left lifting property with respect to all $r\in\mathrm{Epi}(\rightperp{\X})$. When $\rightperp{\X}$ is cogenerating, we have the following solid square with $Y\in\rightperp{\X}$ and a monomorphism $j$:

\[
\xymatrix@C=2pc{
M_{\alpha} \ar[r]^-{j} \ar[d]_{i_{\alpha\beta}} & Y \ar[d] \\
M_{\beta} \ar[r] \ar@{-->}^-{h}[ur] & 0
}
\]
Since the dotted arrow exists, $i_{\alpha\beta}$ must be a monomorphism, as required.

When we are in this situation that $\rightperp{\X}$ is cogenerating, we will think of all objects $M_\alpha$ as subobjects of $M_\sigma$ and  usually write
$0=M_0\subseteq \cdots\subseteq M_{\alpha}\subseteq M_{\alpha+1}\subseteq\cdots \subseteq M_{\sigma}=M$ for a filtration. 
\end{remark}
\end{ipg}

\begin{ipg}\textbf{(Fully faithful adjoint triples)}
Finally, we recall basic facts about adjoint triples involving fully faithful functors. To that end, suppose we have categories $\C$, $\D$ and two adjunctions $(L,M)\colon \D \rightleftarrows \C$ and $(M,R)\colon \C \rightleftarrows \D$,
\[
 \xymatrix@C=2pc{
\C \ar[rr]|-{M} & &  \D \ar@/_0.9pc/[ll]_-{L} \ar@/^0.9pc/[ll]^-{R}.
}
\]
We denote by $\lambda\colon LM\to\Id_\C$ and $\eta\colon\Id_\D\to ML$ the counit and the unit of the first adjunction and by $\epsilon\colon MR\to\Id_\D$ and $\mu\colon\Id_\C\to RM$ the counit and the unit of the second adjunction, respectively

It is a folklore result that $L$ is fully faithful if and only if $R$ is such, so that $\eta$ is an isomorphism if and only if $\epsilon$ is an isomorphism. In fact, we can be more precise.

\begin{lemma}[{\cite[Lemma 1.3]{DyckhoffTholen}}] \label{lem:inverse-unit}
In the situation above, if $\eta$ is invertible, so is $\epsilon$ and $\epsilon^{-1}$ is given by the composition
\[
\xymatrix@1{
\Id_\D \ar[rr]^-{\eta} && ML \ar[rr]^-{M\mu L} && MRML \ar[rr]^-{MR\eta^{-1}} && MR.
}
\]
Conversely, if $\epsilon$ is invertible, so is $\eta$ and $\eta^{-1}$ is given by the composition
\[
\xymatrix@1{
ML \ar[rr]^-{ML\epsilon^{-1}} && MLMR \ar[rr]^-{M\lambda R} && MR \ar[rr]^-{\epsilon} && \Id_\D.
}
\]
\end{lemma}

For the rest of the subsection, we will assume that $L$ and $R$ are fully faithful. In that case, there is a distinguished natural transformation $\tau\colon L\to R$ given by the following lemma.

\begin{lemma}\label{lem:jump-over-middle-adjoint}
Given an adjunction triple $(L,M,R)$ with $L$ and $R$ fully faithful and the adjunction units and counits denoted as above, then the composition of natural transformations
\[
L \xrightarrow{\mu L} RML \xrightarrow{R\eta^{-1}} R
\]
coincides with the composition
\[
L \xrightarrow{L\epsilon^{-1}} LMR \xrightarrow{\lambda R} R.
\]
\end{lemma}

\begin{proof}
By the naturality of $\lambda$, the following diagram commutes, and the equality $L\eta^{-1}=\lambda L$ follows from the triangular identities for the adjunction $(L,M)$,
\[
\xymatrix{
LML \ar[rr]^-{LM\mu L} \ar[d]_-{L\eta^{-1}=\lambda L} && LMRML \ar[rr]^-{LMR\eta^{-1}} \ar[d]^-{\lambda RML} && LMR  \ar[d]^-{\lambda R} \\
L \ar[rr]^-{\mu L} && RML \ar[rr]^-{R\eta^{-1}} && R
}
\]
Using that $\epsilon^{-1} = MR\eta^{-1} \circ M\mu L \circ \eta$ by the previous lemma, the commutativity of the outer rectangle of the last diagram implies the desired equality.
\end{proof}

If we denote $\tau = R\eta^{-1} \circ \mu L = \lambda R \circ L\epsilon^{-1}$ as indicated, we have for each $Y\in\D$ a natural morphism $\tau_Y\colon L(Y) \to R(Y)$. If $Y$ is of the form $Y=M(X)$ for some $X\in\C$, we can say more. This will be the crux in Section~\ref{sec:Fun} allowing us, for very special adjunction triples, efficiently lift objects from $\D$ to $\C$ by constructing such factorisations of the natural morphism $\tau$.

\begin{lemma}\label{lem:abstract-latching-matching-factorisation}
Given an adjunction triple $(L,M,R)$ with $L$ and $R$ fully faithful and the adjunction units and counits and $\tau\colon L\to R$ denoted as above, then for any $X\in\C$ we have the following commutative triangle:
\[
\xymatrix{
LM(X) \ar[rr]^-{\tau_{M(X)}} \ar[dr]_-{\lambda_X} && RM(X). \\
& X \ar[ur]_-{\mu_X}
}
\]
\end{lemma}

\begin{proof}
By the naturality of $\lambda$, the rectangle below commutes, and the equality $\epsilon^{-1}M=M\mu$ follows by the triangle identities for the adjunction $(M,R)$,
\[
\xymatrix{
LM(X) \ar[rr]^-{LM\mu_X=L\epsilon_{M(X)}^{-1}} \ar[d]_-{\lambda_X} && LMRM(X) \ar[d]^-{\lambda_{RM(X)}} \\
X \ar[rr]_-{\mu_X} && RM(X).
}
\]
The conclusion follows from the commutativity since $\tau_{M(X)} = \lambda_{RM(X)} \circ L\epsilon_{M(X)}^{-1}$ by the definition of $\tau$.
\end{proof}
\end{ipg}

\section{Linear Reedy categories}
\label{sec:Reedy_cats}

In this section we first define direct and inverse linear categories and discuss some examples. Linear Reedy categories are introduced in Definition \ref{dfn:Reedy_cat}.

\begin{definition}
\label{dfn:direct/inverse}
Let $k$ be a field and let $\class C$ be a small $k$--linear category. We call $\class C$ \textit{direct} (resp., \textit{inverse}) if the following hold:
\begin{itemize}
\item[-] For all objects $c$ in $\class C$; $\mathrm{End}_{\class C}(c)\cong k$, as $k$--vector spaces.
\item[-] There exists an ordinal number $\lambda$ and a degree function $\deg\colon\mathrm{Obj}(\class C)\rightarrow \lambda$, such that any non-zero, non-endomorphism $c\rightarrow c'$ in $\class C$ satisfies $\deg(c)<\deg(c')$ (resp., $\deg(c)>\deg(c')$).
\end{itemize}
\end{definition}

\begin{remark}
\label{rem:opp_direct/inverse}
It follows from Definition \ref{dfn:direct/inverse} that if $\C$ is a $k$--linear direct (resp., inverse) category then $\C^{\mathrm{o}}$ is a $k$--linear inverse (resp., direct) category.
\end{remark}

\begin{remark}
\label{rem:non_isos_rad}
Observe that if $\class C$ is a $k$--linear direct or inverse category and $f\colon c\rightarrow d$ is a non-zero, non-endomorphism in $\class C$, then $f$ is a non-isomorphism. Thus, in view of \S\ref{subsec:radicals}, for any two objects $c$ and $d$ in $\class C$ with $c\neq d$, we have ${\class C}(c,d)=\rad_{\class C}(c,d)$.
\end{remark}

\begin{example}
\label{ex:rooted_quivers}
(Left and right rooted quivers). 
Let $Q$ be a quiver. As in \cite{MR2100360}, we can associate to $Q$ a transfinite sequence of subsets of the vertex set $Q_{0}$, by putting $V_{0}:=\emptyset$, 
for any successor ordinal $\alpha=\beta +1$,
\[V_{\alpha}:=\{i\in Q_{0}\,|\, \mbox{each arrow with target}\,\, i \,\, \mbox{has source in}\,\,
V_\beta
\},
\]
and for any limit ordinal $\alpha$, $V_{\alpha}:=\cup_{\beta<\alpha}V_{\beta}$.

This sequence is in fact ascending, i.e., if $\gamma<\delta$ then $V_{\gamma}\subset V_{\delta}$, see \cite[Lemma~2.7]{hj2016}. The quiver $Q$ is called \textit{left rooted} if there exists an ordinal $\lambda$ with $V_{\lambda}=Q_{0}$. We observe that given a left rooted quiver $Q$, for every vertex $i$ in $Q_0$, there exists a unique ordinal $\beta$ such that $i\in V_{\beta+1}\setminus V_{\beta}$. This defines a function $\deg\colon Q_0\rightarrow \lambda$ with $\deg(i):=\beta$. In fact, we claim that the following hold for $Q$ left rooted:
\begin{itemize}
\item[(i)] The function $\deg$ is such that for any path $p$ in $Q$, starting at $i$ and ending at $j$, with $i\neq j$, we have $\deg(i)<\deg(j)$.
\item[(ii)] There is no non-trivial path from a vertex to itself.
\end{itemize}
In \cite[Corollary~2.8]{hj2016} it is proved that if $i\notin V_{\beta}$ and $j\in V_{\beta+1}$ then there is no arrow from $i$ to $j$ in $Q_1$. From this it follows that there are no loops in $Q$ and that given an arrow $a\colon i\rightarrow j$ (necessarily with $i\neq j$) the inequality $\deg(i)<\deg(j)$ necessarily holds. From these observations it follows by an inductive argument that (i) and (ii) hold.

From the above discussion it follows that given a left rooted quiver $Q$, with $Q=V_{\lambda}$, its $k$--linearization $kQ$ as in \S\ref{quivers} satisfies the following:
\begin{itemize}
\item[(i)] The function $\deg$, interpreted as $\deg\colon\mathrm{Obj}(kQ)\rightarrow \lambda$,
satisfies for any non-zero non-endomorphism $f\colon c\rightarrow d$ in $kQ$ the inequality $\deg(c)<\deg(d)$.
\item[(ii)] For all objects $c$ in $kQ$ we have $\End_{kQ}(c)\cong k$. 
\end{itemize}
Hence $kQ$ is a direct $k$--linear category in the sense of Definition~\ref{dfn:direct/inverse}.

Using the notion of \textit{right rooted} quivers, as in \cite[Sec.~9]{EEGR-injective-quivers}, and dual arguments, one can show the corresponding statement relating  $k$--linearizations of right rooted quivers and inverse $k$--linear categories.
\end{example}

Certain quivers with relations may also serve as examples of $k$--linear direct or inverse categories, as the next example shows.

\begin{example} (Quivers with relations).
Let $(Q,I)$ be a quiver with relations over a field $k$ and
recall the $k$--linearization $kQ_{I}$ of $(Q,I)$ from \ref{quivers}.
Assume further that $I$ consists of formal $k$--linear combintation of paths of length at least two.
We claim that $kQ_{I}$ is a $k$--linear direct (resp.\ inverse)  category if and only if $Q$ is left (resp.\ right) rooted.

Indeed, if $Q$ is left rooted, we may argue as in Example \ref{ex:rooted_quivers} and use the same degree function on the (common) set of objects of $kQ$ and $kQ_I$, which in particular will satisfy the conditions of Definition \ref{dfn:direct/inverse}, proving that $kQ_{I}$ is direct $k$-linear.
For the converse direction, if $kQ_{I}$ is direct with degree function $\mathrm{deg}\colon Q_0\rightarrow \lambda$, we may consider a (possibly transfinite) filtration of the vertex set $Q_0$ based on the degrees of the vertices. That is, for each $\alpha<\lambda$ we define $V_\alpha$ as the set of all vertices of $i\in Q_0$ with $\deg(i)<\alpha$. Since any arrow $a\colon i\to j$ of $Q$ induces a non-zero morphism $i\to j$ in $kQ_I$ (relations are linear combinations of paths of length at least two by assumption), it follows that each arrow with target $j$ of degree $\alpha$ has source in $V_\alpha$.
It follows that $Q$ is left rooted. 
The remaining assertion involving inverse categories and right rooted quivers can be seen using dual arguments.
\end{example}

The next definition encompasses the concepts of direct and inverse $k$--linear categories from Definition \ref{dfn:direct/inverse}. It is a linear analogue of the classic definition of a Reedy category \cite[Def.~2.1]{riehl}.

\begin{definition}
\label{dfn:Reedy_cat}
Let $k$ be a field and let $\class C$ be a small $k$--linear category. Then $\class C$ is called \textit{$k$--linear Reedy} if it satisfies the following:
\begin{itemize}
\item[-] There exists an ordinal number $\lambda$ and a degree function $\deg\colon\mathrm{Obj}(\class C)\rightarrow \lambda$, together with,
two $k$--linear subcategories $\class C^+$ and $\class C^-$ of $\class C$, both having the same objects as $\class C$, where $\class C^+$ is direct $k$--linear and $\class C^-$ is inverse $k$--linear (with respect to the same function $\deg$).
\item[-] Given any pair of objects $c,d \in\class C$, the composition in $\class C$ induces a 
$k$--linear isomorphism
\begin{equation}
\label{eq:Reedy-factorization}
\bigoplus_{e \in \mathrm{Obj}(\C)} \class C^+(e,d) \otimes_{k} \class C^-(c,e) \xrightarrow{\cong} \class C(c,d).
\end{equation}
\end{itemize}

From the defining properties of $\C^+$ and $\C^-$ it follows that the index set in the above displayed isomorphism can instead run through all objects $e$ of $\C$ such that $\deg(e)\leqslant\min\{\deg(c),\deg(d)\}$. 

Notice that  the isomorphism~\eqref{eq:Reedy-factorization} tells us that, for any (non-zero) morphism $f\colon c\rightarrow d$ in $\C$,
we require the existence of finitely many objects $e_1,\dots,e_n$, and morphisms $f_{i}^-\colon c\rightarrow e_i$ and $f_{i}^+\colon e_i\rightarrow d$ such that $f=f_1+\dots+f_n$ where $f_i=f_i^+\circ f_i^-$ for all $i$.
We refer to such an expression $f=\sum_{i}f_i^+\circ f_i^-$ as a \textit{Reedy factorisation} of $f$.
\end{definition}

\begin{remark}
\label{rem:opp_Reedy_cat}
If $\C$ is a $k$--linear Reedy category, its opposite category $\C^{\mathrm{o}}$ is a $k$--linear Reedy category where $(\C^{\mathrm{o}})^+=(\C^-)^{\mathrm{o}}$ and $(\C^{\mathrm{o}})^-=(\C^+)^{\mathrm{o}}$, cf.~Remark~\ref{rem:opp_direct/inverse}. 
\end{remark}

The next example describes direct and inverse linear categories as ``extreme'' examples of linear Reedy categories.

\begin{example} 
\label{ex:extreme Reedy cases} 
Let $\C$ be a $k$--linear inverse category. Then $\C$ is $k$--linear Reedy where $\C^-=\C$\, and\, $\C^+$ has the same objects as $\C$ and morphisms satisfying the following rule:
\begin{equation}
\Hom_{\C^+}(c,d)\cong 
\begin{cases}
k\cdot{\mathrm{id}_{c}} , & \text{if}\ c=d \\
0, & \text{otherwise.}
\end{cases}
\nonumber
\end{equation}
The dual observations apply in case $\C$ is $k$--linear direct.
\end{example}

\begin{example}
	\label{ex:classical_Reedy}
Suppose that $\class D$ is classical Reedy category as defined in~\cite[Def.~15.1.2]{Hir} for instance. That is (if we allow for transfinite degrees of objects), 
\begin{itemize}
	\item[-] There exists an ordinal number $\lambda$ and a degree function $\deg\colon\mathrm{Obj}(\class D)\rightarrow \lambda$, together with, two subcategories $\class D^+$ and $\class D^-$ of $\class D$, both having the same objects as $\class D$, such that non-identity morphisms in $\class D^+$ raise the degree and non-identity morphisms in $\class D^-$ lower the degree, and
	\item[-] each morphism $g\colon d\to d'$ in $\class D$ has a unique factorization $g=g^+\circ g^-$ such that $g^+$ is a morphism in $\class D^+$ and $g^-$ in $\class D^-$.
\end{itemize}
Then the $k$--linearization $\class C=k\class D$ (i.e.\ the $k$-linear category $\C$ having the same objects as $\class D$ and such that morphisms in $\class D$ form bases of homomorphism spaces of $\class C$) is linear Reedy.
\end{example}

\begin{example} \label{ex:linear-simplices}
As already mentioned, a standard example of a classical Reedy category is the category $\mathbf{\Delta}$ with finite non-zero ordinals (traditionally denoted as $[n]=\{0,1,\dots,n\}$ for $n\ge 0$) as objects, the weakly monotone functions between the ordinals as morphisms, and the degree function given by $\deg([n])=n$. In particular $k\mathbf{\Delta}$ is linear Reedy, and so is for each $N\ge 0$ the full subcategory $k\mathbf{\Delta}_{\le N}$ whose objects are only the ordinals $[0], [1], \dots, [N]$. For $N=1$, $k\mathbf{\Delta}_{\le 1}$ has an explicit description: It is isomorphic to $kQ_I$ (recall \S\ref{quivers}) for the quiver
\begin{equation}
\label{eq:qh_ex}
\xymatrix@C=2pc{
[0] \ar@/^0.8pc/[r]^-{d^0} \ar@/_0.8pc/[r]_-{d^1} & [1] \ar[l]|\hole|{s}
}
\nonumber
\end{equation}
with relations $\{s\circ d^0 -e_{[0]}, s\circ d^1-e_{[0]}\}$.
\end{example}

\begin{example}
\label{ex:qh}
Let $k$ be a field and let $Q$ be the quiver
\begin{equation}
\label{eq:qh_ex2}
 \xymatrix@C=2pc{
 0 \ar@/^0.4pc/[r]^-{\alpha} & 1 \ar@/^0.4pc/[l]^-{\beta}
 }
 \nonumber
 \end{equation}
together with the relation $I=\{\beta\circ \alpha\}$. It is easy to see that the  $k$--linear category $\C:=kQ_{I}$ (as in \S\ref{quivers}) is Reedy, where we associate to the objects of $\C$ (= vertices of $Q$) a degree function as labeled and the direct and inverse subcategories are as follows:
\[
\C^+=(0\xrightarrow{\alpha} 1) \qquad \mbox{and} \qquad \C^-=( 0\xleftarrow{\beta}1).
\]

By slightly abusing the notation we write again $I$ for the ideal of the algebra $kQ$ generated by the relation $\beta\circ \alpha$. The finite dimensional algebra $A=KQ/I$ is quasi-hereditary, with heredity chain $A e_0A\subseteq A(e_0+e_1)A=A$ (in the sense of~\cite{DR2}). Notice here that the idempotent of lowest degree sits at the bottom of the filtration. This indicates a possible  connection between Reedy categories and quasi-hereditary algebras; we explore this in Section \ref{sec:QH}.
\end{example}

\section{Reedy categories and quasi-hereditary algebras}
\label{sec:QH}
In this section we fix a field $k$ and a $k$--linear Reedy category $\C$, with a degree function $\mathrm{deg}\colon\mathrm{Obj}(\C)\rightarrow\lambda$, and we denote by $\mathcal{K}^{\mathcal{C}}$ the category of $k$--linear functors from $\C$ to the category $\mathcal{K}$ of $k$--vector spaces.

We define the class of standard functors (Definition~\ref{def:standards}) and prove the results mentioned in the introduction. We then specialize the situation further to Reedy algebras (Definition~\ref{def:Reedy_algebra}) which are shown to be quasi-hereditary with an exact Borel subalgebra (Theorem~\ref{thm:Reedy_is_qh}).

\subsection{Standard functors}
\label{subsec:standard}
For every pair of objects $c,d$ in $\C$ and an ordinal number $\alpha\leqslant\lambda$ we denote by $\I^\C_{<\alpha}(c,d)$ the subspace of $\C(c,d)$
which is the image of the map induced by composition in $\C$,
\begin{equation}
\label{eq:trace}
\bigoplus_{\deg(d_i)<\alpha} \class C^+(d_i,d) \otimes_{k} \class C^-(c,d_i) \to \C(c,d).
\end{equation}
In other words, $\I^\C_{<\alpha}(c,d)$ consists of those morphisms from $c$ to $d$ whose Reedy factorization is indexed by objects of degree strictly smaller than $\alpha$. 
It is straightforward to check that for any fixed $\alpha$, $\I^\C_{<\alpha}(-,-)$ is a two-sided ideal of $\C$ as a pre-additive category.

\begin{remark}
\label{rem:Reedy_quotient}
Given $\alpha<\lambda$, there is a Reedy category which we (by slightly abusing the notation) denote by $\C/\I^\C_{<\alpha}$. Its objects are those objects $c\in\mathrm{Obj}(\C)$ whose degree is $\geqslant\alpha$, and given two such objects $c,d$, we define $\C/\I^\C_{<\alpha}(c,d):=\C(c,d)/\I^\C_{<\alpha}(c,d)$. The composition is the one induced on cosets by the composition in $\C$. The degree function naturally restricts to $\deg\colon\Obj(\C/\I^\C_{<\alpha})\to\lambda\setminus\alpha$, where $\lambda\setminus\alpha$ is well-ordered with the ordering induced from $\lambda$, so that $\lambda\setminus\alpha$ is canonically order-isomorphic to a unique ordinal $\lambda'$.

Note that given $c,d\in\mathrm{Obj}(\C)$ of degree $\geqslant\alpha$, then $\C^+(c,d)\cap\I^\C_{<\alpha}(c,d)=0=\C^-(c,d)\cap\I^\C_{<\alpha}(c,d)$.
In particular, the full subcategory of $\C^+$ given by objects of degree $\geqslant\alpha$ can be identified with a $k$--linear subcategory of $\C/\I^\C_{<\alpha}$ which we denote by $(\C/\I^\C_{<\alpha})^+$, and we can define $(\C/\I^\C_{<\alpha})^-$ in the same vein. Finally we leave it for the reader to check that the composition in $\C/\I^\C_{<\alpha}$ induces for any $c,c'\in\mathrm{Obj}(\C)$ of degree $\geqslant\alpha$ a $k$--linear isomorphism
\[ 
\bigoplus_{c'' \in \mathrm{Obj}(\C), \deg(c'')\geqslant\alpha} (\C/\I^\C_{<\alpha})^+(c'',c') \otimes_{k} (\C/\I^\C_{<\alpha})^-(c,c'') \xrightarrow{\cong} \C/\I^\C_{<\alpha}(c,c').
\]
\end{remark}

\begin{definition}
\label{def:standards}
For every object $c$ of degree $\alpha$ in $\C$ we define the \textit{standard left $\C$--module of $c$} as $\Delta_{c}:=\C(c,-)/\I^\C_{<\alpha}(c,-)$
and the \textit{standard right $\C$--module of $c$} as $\Delta^{c}:=\C(-,c)/\I^\C_{<\alpha}(-,c)$
(both are naturally projective functors on $\C/\I^\C_{<\alpha}$).
\end{definition}

We consider the left adjoint to the restriction functor from left $\C$--modules to left $\C^-$--modules, which we denote by $\C\otimes_{\C^-}-$. Using this functor, in the next result we show that the inverse subcategory $\C^-$ plays a role analogous to exact Borel subalgebras of quasi-hereditary algebras, as defined in \cite{Koe}.

\begin{theorem}
\label{thm:Reedy_Borel}
For any object $c$ of $\C$ the following hold:
\begin{itemize}
\item[(i)]  There is an isomorphism $\C^-(c,-)/\I^{\C^-}_{<\mathrm{deg}(c)}(c,-)\cong S_{c}$, where $S_{c}$ is the simple left $\C^-$--module at $c$. In other words, the inverse subcategory $\C^-$ has all of its standard modules simple.
\item[(ii)] The induction functor $\C\otimes_{\C^-}-$ is exact.
\item[(iii)] There is an isomorphism $\Delta_{c}\cong\C\otimes_{\C^-}S_{c}$ of left $\C$--modules.
\end{itemize}
\end{theorem}

\begin{proof}
(i) It suffices to prove that $\I^{\C^-}_{<\mathrm{deg}(c)}(c,-)=\rad_{\C^-}(c,-)$.

From the defining properties of inverse $k$--linear categories it follows that the isomorphisms of $\C^-$ are precisely the non-zero endomorphisms of $\C^-$.
In particular, the endomorphism rings of objects of $\C^-$ are local and so the radical morphisms in $\C^-$ are precisely the non-isomorphisms by \S\ref{subsec:radicals}.
Thus, any non-zero morphism in $\rad_{\C^-}(c,d)$ strictly lowers the degree, and conversely, any morphism from $c$ to $d$ in $\C^-$ that strictly lowers the degree is in $\rad_{\C^-}(c,d)$.

(ii) It is well known (see e.g. \cite[Sect.~4]{Murfet_ringoids}) that if $X$ is any functor in $\mathcal{K}^{\C^-}$ and $c$ is an object of $\C$, then  
\[(\C\otimes_{\C^-}X)(c)\, =\, \C(-,c)\otimes_{\C^-}X;\] a tensor product of functors as in Section \ref{sec:preliminaries}. 
Using the Reedy decomposition property we obtain an isomorphism of right $\C^-$--modules,
\[
\C(-,c)\,\, \cong\bigoplus\limits_{\mathrm{deg}(d)\leqslant\mathrm{deg}(c)}\C^+(d,c)\otimes_{k}\C^-(-,d).
\]
By putting the two identities together, we deduce that
\[
(\C\otimes_{\C^-}X)(c)\,\cong
\bigoplus\limits_{\mathrm{deg}(d)\leqslant\mathrm{deg}(c)}\C^+(d,c)\otimes_{k} X(d).
\]
Since the tensor product of vector spaces $\C^+(c,d)\otimes_{k}-$ is exact and so is the direct sum of vector spaces, the induction functor $\C\otimes_{\C^-}-$ is exact as well.

(iii) We consider a projective presentation of the simple left $\C^-$--module $S_c$, which is as follows (well known -- see e.g. \cite[Lemma~7.17]{HJ_qshaped}).
\[
\xymatrix@C=2pc{
\bigoplus\limits_{f\in\mathrm{rad_{\C^-}}(c,d)} \C^-(d,-)  \ar[r] & \C^-(c,-) \ar[r] & S_c \ar[r] & 0.  
}
\]
After applying on this the induction functor, which is
right exact,
we obtain an exact sequence of left $\C$--modules
\[
\xymatrix@C=2pc{
\bigoplus\limits_{f\in\mathrm{rad_{\C^-}}(c,d)} \C(d,-)  \ar[r] & \C(c,-) \ar[r] & \C\otimes_{\C^-}S_c \ar[r] & 0.  
}
\]
It now suffices to prove that the image of the left-most map, which is $\sum_{f\in\mathrm{rad_{\C^-}}(c,d)}\mathrm{Im}\, \C(f,-)$, is isomorphic to $\I^\C_{<\mathrm{deg}(c)}(c,-)$, as left $\C$--modules. In fact we will show that for every object $e$ of $\C$ we have an equality
\[\sum\limits_{f\in\mathrm{rad_{\C^-}}(c,d)}\mathrm{Im}\, \C(f,e) = \I^\C_{<\mathrm{deg}(c)}(c,e).\]
We observed in the proof of (i) that any non-zero non-isomorphism $f\colon c\rightarrow d$ in $\C^-$ (thus also non-endomorphism) lowers the degree, hence the Reedy factorization of any morphism in $\mathrm{Im}\, \C(f,e)$ is indexed by objects of degree $<\mathrm{deg}(c)$. This proves the containment ``$\subseteq$''. For the reverse direction, let $h\in\I^\C_{<\mathrm{deg}(c)}(c,e)$ and consider its Reedy factorization $h=\sum_{i}h_i^+\circ h_i^-$, where for every $i$, the target $d_i$ of $h_i^-$ (= source of $h_i^+$) has degree $<\mathrm{deg}(c)$. Notice in particular that for all $i$ the morphism $h_i^-$ belongs to $\mathrm{rad_{\C^-}}(c,d_i)$. Clearly $h\in\sum_{h_i^-:c\rightarrow d_i}\mathrm{Im}\, \C(h_i^-,e)$, thus $h$ belongs to the left hand side of the displayed equality.
\end{proof}

Our next goal is to show that the left $\C$--modules $\C(c,-)$ admit a (transfinite) filtration by the class of standard functors (recall Definition~\ref{def:filt}).
We first prove a special case for inverse categories.

\begin{proposition}
\label{prop:inverse_by_simples}
Assume that $\C$ is inverse $k$--linear with a degree function $\deg\colon\mathrm{Obj}(\C)\rightarrow\lambda$. Then any functor $X\in\mathcal{K}^{\C}$ is filtered by simple functors. In more detail, $X$ admits a transfinite filtration
\[0=X_{<0}\subseteq X_{<1}\subseteq \cdots\subseteq X_{<\alpha}\subseteq X_{<\alpha+1}\subseteq\cdots \subseteq X_{<\lambda}=X,\]
where for all $\alpha<\lambda$ the left $\C$--module $X_{\alpha}:=X_{<\alpha+1}/X_{<\alpha}$ is isomorphic to  $\oplus_{\mathrm{deg(d)}=\alpha}S_d^{(\dim X(d))}$.
\end{proposition}

\begin{proof}
Fix a functor $X\in\mathcal{K}^{\C}$. We set $X_{<0}:=0$ and for any ordinal $1\leqslant\alpha\leqslant\lambda$, we define 
a $k$--linear functor $X_{<\alpha}\colon\C\rightarrow\mathcal{K}$, on objects by the formula
\[
X_{<\alpha}(d)=
\begin{cases}
X(d), &  \deg(d)<\alpha \\
0, & \text{otherwise,}
\end{cases}
\]
while for any morphism $f\colon d\rightarrow d'$, if both $\deg(d)$ and $\deg(d')$ are strictly smaller than $\alpha$, then $X_{<\alpha}(f):=X(f)$, otherwise $X(f):=0$.

Since $\C$ is inverse, for any ordinal $\alpha\ < \lambda$ there is a well defined natural transformation $i_{\alpha,\alpha+1}\colon X_{<\alpha}\rightarrow X_{<\alpha+1}$ where for any object $c\in\C$, we set $i_{\alpha,\alpha+1}^c:=\mathrm{id}_{X(c)}$ if $\deg(c)<\alpha$, otherwise $i_{\alpha,\alpha+1}^c:=0$  (cf. the proof of \cite[Thm.~7.9]{hj2016}). Notice that for all $\alpha < \lambda$ the morphism $i_{\alpha,\alpha+1}$ is a monomorphism and that for any limit ordinal $\sigma\leqslant\lambda$ we have $X_{<\sigma}=\cup_{\alpha<\sigma}X_{<\alpha}$. In this way we obtain a well-ordered direct system $(X_{<\alpha}\, |\, i_{\alpha\beta}\colon A_{<\alpha}\rightarrow A_{<\beta})_{\alpha<\beta\leqslant\lambda}$ satisfying conditions (i)--(iii) of Definition \ref{def:filt}.

Now, for an ordinal $\alpha<\lambda$ and $d\in\C$, the functor $X_{\alpha}:=X_{<\alpha+1}/X_{<\alpha}$ satisfies 
\[
X_{\alpha}(d)=
\begin{cases}
X(d), &  \deg(d)=\alpha \\
0, & \text{otherwise,}
\end{cases}
\]
while for any morphism $f$ we have $X_{\alpha}(f)\neq 0$ only when $f$ is an invertible endomorphism of an object of degree $\alpha$. Hence $X_{\alpha}\cong \oplus_{\mathrm{deg(d)}=\alpha}S_d^{(\dim X(d))}$ as needed.
\end{proof}

This brings us to the following result on filtrations of (restrictions of) representables by standards in the Reedy setting.

\begin{theorem}
\label{thm:proj_by_standards}
Let $\C$ be a $k$--linear Reedy category with degree function $\mathrm{deg}\colon\mathrm{Obj}(\C)\rightarrow\lambda$, let $c$ be an object in $\C$ of degree $\alpha$, and let $\beta\leqslant\lambda$ be an ordinal. Then the left $\C_{<\beta}$--module $\C(c,-)|_{\C_{<\beta}}$ is filtered by standard functors from $\K^{\C_{<\beta}}$ corresponding to objects of degrees $<\min(\alpha+1,\beta)$.
\end{theorem}

\begin{remark}
\label{rem:proj_by_standards}
If we take $\beta=\lambda$, the theorem simply says that the representable functor $\C(c,-)$ can be filtered by standard functors from $\K^\C$ corresponding to objects of degrees $\leqslant\alpha$.
Moreover, it will follow from the proof that the filtration can be taken such that it is indexed by an ordinal successor, the last filtration factor is isomorphic to $\Delta_c$, and all other filtration factors are isomorphic to standard functors of objects of degree $<\alpha$.
\end{remark}

\begin{proof}[Proof of Theorem~\ref{thm:proj_by_standards}]
Put $\gamma:=\min(\alpha+1,\beta)$.
Proposition~\ref{prop:inverse_by_simples} provides us with a continuous direct system of $\C_{<\beta}^-$--modules:
\begin{equation}
\label{eq:filtration of C^-}
0 \subseteq\, \I^{\C^-}_{<1}(c,-)|_{\C_{<\beta}^-}\subseteq \I^{\C^-}_{<2}(c,-)|_{\C_{<\beta}^-}\subseteq  \cdots\subseteq \I^{\C^-}_{<\gamma}(c,-)|_{\C_{<\beta}^-}=\C^-(c,-)|_{\C_{<\beta}^-}.
\end{equation}
Observe that this notation agrees precisely with the one from (\ref{eq:trace}); cf. Example~\ref{ex:extreme Reedy cases}.

For all ordinals $\delta<\gamma$ the quotient 
\[
Q^{\C^-}_{\delta}(c,-):=
\frac{\I^{\C^-}_{<\delta+1}(c,-)|_{\C_{<\beta}^-}}{\I^{\C^-}_{<\delta}(c,-)|_{\C_{<\beta}^-}}
\]
is isomorphic to a coproduct of simple $\C_{<\beta}$-modules $S_d$ where $d$ runs through a certain set of objects of degree $\delta$. Consider the induction functor $\C_{<\beta}\otimes_{\C_{<\beta}^-}-$ and note that for all $\delta<\gamma$ and $d\in\C_{<\beta}$ we have isomorphisms
\begin{eqnarray*}
\C_{<\beta} \otimes_{\C_{<\beta}^-} \I^{\C^{-}}_{<\delta}(c,-)|_{\C_{<\beta}^-} (d) & \cong &
\C_{<\beta}(-,d)\otimes_{\C^-_{<\beta}}\I^{\C^{-}}_{<\delta}(c,-)|_{\C_{<\beta}} \\
&\cong & \bigoplus_{d_i}\C^+_{<\beta}(d_i,d)\otimes_{k}\C_{<\beta}^-(-,d_i)\otimes_{\C^-_{<\beta}}\I^{\C^{-}}_{<\delta}(c,-)|_{\C_{<\beta}} \\
&\cong & \bigoplus_{d_i}\C^+_{<\beta}(d_i,d)\otimes_{k}\I^{\C^{-}}_{<\delta}(c,d_i)|_{\C_{<\beta}} \\
&\cong & \I^\C_{<\delta}(c,d). 
\end{eqnarray*}

Here, the first and the third isomorphism are standard properties of the induction functor, and the second and the fourth isomorphisms follow from the  Reedy decomposition.
Thus the induction functor, which is exact from Theorem \ref{thm:Reedy_Borel}(ii), when applied to the continuous direct system of $\C_{<\beta}^-(c,-)$--modules  displayed in (\ref{eq:filtration of C^-}), produces a continuous direct system of $\C_{<\beta}$--modules,   
\begin{equation}\label{eq:filtration_by_standards}
0 \subseteq\, \I^\C_{<1}(c,-)|_{\C_{<\beta}} \subseteq \I^\C_{<2}(c,-)|_{\C_{<\beta}} \subseteq \cdots\subseteq \I^\C_{<\gamma}(c,-)|_{\C_{<\beta}} = \C(c,-)|_{\C_{<\beta}},
\end{equation}
where all the successive quotients involved are isomorphic to  coproducts of functors of the form $\C_{<\beta}\otimes_{\C_{<\beta}^-}S_d$, for $\deg(d)<\gamma$. Hence the proof is finished once we employ Theorem \ref{thm:Reedy_Borel}(iii).
\end{proof}

In a $k$--linear abelian category a non-empty finite set of objects $\{
\Delta_i\}_{i=1}^{n}$ is called an \textit{exceptional collection} if $\mathrm{End}(\Delta_i)\cong k$, $\Hom(\Delta_i,\Delta_j)\neq 0$ implies $i\leqslant j$ and $\mathrm{Ext}^n(\Delta_i,\Delta_j)\neq 0$ implies $i<j$ for all $n\geqslant1$.
In the next result we show that $\mathcal{K}^{\mathcal{C}}$ admits an (infinite, a priori) collection of objects that satisfy dual conditions to the above on Hom/Ext spaces.
We also record a related vanishing result for tensor product and Tor spaces.

\begin{theorem}
\label{thm:Reedy_exceptional}
Let $\C$ be a $k$--linear Reedy category with degree function $\mathrm{deg}\colon\mathrm{Obj}(\C)\rightarrow\lambda$.
Then for any objects $c,d$ in $\C$ the following hold:
\begin{itemize}
\item[(i)] $\End_{\C}(\Delta_{c})\cong k \cong \Delta^{c}\otimes_{\K^\C}\Delta_{c}.$
\item[(ii)] $\Hom_{\C}(\Delta_{c},\Delta_{d})\neq 0 \Rightarrow \mathrm{deg}(d)<\mathrm{deg}(c)$ or $d=c$.
\item[(iii)] For all $n\geqslant1$, $\Ext^{n}_{\C}(\Delta_{c},\Delta_{d})\neq 0 \Rightarrow  \mathrm{deg}(d)<\mathrm{deg}(c)$.
\item[(iv)] For all $n\geqslant0$, $\Tor_{n}^{\C}(\Delta^{c},\Delta_{d})\neq 0 \Leftrightarrow n=0$ and $c=d$.
\end{itemize}
\end{theorem}

\begin{proof}
To prove that $\mathrm{End}_{\mathcal{K}^{\C}}(\Delta_{c})\cong k$ in (i) and also to prove statement (ii), we first notice that $\C^+(c,d)\cong\C(c,d)/\I^\C_{<\alpha}(c,d)$ as $k$--vector spaces, where $\alpha=\deg(c)$.
In fact, $\C^+(c,-)\cong\C(c,-)/\I^\C_{<\alpha}(c,-)$ also as left $\C^+$--modules (recall Remark~\ref{rem:Reedy_quotient}), so that $\Delta_c|_{\C^+}\cong\C^+(c,-)$ for each $c\in\C$. Since the restriction functor $\K^{\C}\to\K^{\C^+}$ is clearly faithful, we have for each $c,d\in\C$ an inclusion
\[
\Hom_{\C}(\Delta_c,\Delta_d) \subseteq
\Hom_{\C^+}(\Delta_c,\Delta_d) \cong
\C^+(d,c).
\]
Here, the last isomorphism is just the Yoneda lemma.
Now, if $c=d$ then $\C^+(c,c)\cong k$ which proves that $\mathrm{End}_{\mathcal{K}^{\C}}(\Delta_{c})\cong k$. In addition, for all $c \ne d$, if $\C^+(d,c)\neq 0$, then $\deg(d)<\deg(c)$. This proves statement (ii). 

In order to prove (iii), we consider a projective resolution $P_\bullet\to S_c$ of $S_c$ in $\K^{\C^-}$ as before. Since $\C^-$ is inverse, one can inductively construct the resolution in such a way that $P_0=\C^-(c,-)$ and for each $n\geqslant1$ we have $P_n\cong\bigoplus_{i\in I_n}\C^-(c_i,-)$ with $\deg(c_i)<\deg(c)$ for all $i\in I_n$. Since each $P_n$ then has a filtration by simple functors of degrees strictly smaller than $\deg(c)$ by Proposition~\ref{prop:inverse_by_simples}, the same is true for the syzygy modules $\Omega^n(S_c):=\im(P_{n}\to P_{n-1})$.
By Theorem~\ref{thm:Reedy_Borel}, $\C\otimes_{\C^-}P_\bullet$ is a projective resolution of $\Delta_c\cong\C\otimes_{\C^-}S_c$.
Moreover, again by Theorem~\ref{thm:Reedy_Borel}, the corresponding $n$-th syzygy $\C\otimes_{\C^-}\Omega^n(S_c)$ of $\Delta_c$ in $\K^\C$ is filtered by standard functors of objects of degree $<\deg(c)$ for each $n\geqslant1$.

Now we fix $n\geqslant1$ and an object $d$ with $\deg(d)\geqslant\deg(c)$ and prove that $\Ext^n_{\C}(\Delta_c,\Delta_d)=0$. Since every element of the $\Ext$-group is represented by a homomorphism $\C\otimes_{\C^-}\Omega^n(S_c)\to\Delta_d$ in $\K^\C$, it suffices to prove that $\Hom_{\C}(\C\otimes_{\C^-}\Omega^n(S_c),\Delta_d)=0$. Taking into account the filtration of $\C\otimes_{\C^-}\Omega^n(S_c)$ by standard functors $\Delta_e$ with $\deg(e)<\deg(c)\leqslant\deg(d)$, it suffices to prove that $\Hom_{\C}(\Delta_e,\Delta_d)=0$. However, the latter follows from part (ii).

Finally, we finish the proof of (i) and prove (iv). 
Consider a projective resolution $Q_\bullet\to\Delta_{d}$ with $Q_0=\C(d,-)$ and $Q_n\cong\bigoplus_{i\in I_n}\C(d_i,-)$ where $\deg(d_i)<\deg(d)$ for all $i\in I_n$ and $n\geqslant1$, as in the previous paragraph. Then
\[ \Delta^{c}\tens{\C}Q_n \cong \Delta^{c}\tens{\C}\bigoplus_{i\in I_n}\C(d_i,-) \cong \bigoplus_{i\in I_n}\Delta^{c}(d_i). \]
Since $\Delta^{c}\cong(\C^\mathrm{o})^+(c,-)\cong\C^{-}(-,c)$ as right $\C^-$--modules, the $k$--vector spaces $\Tor_{n}^{\C}(\Delta^{c},\Delta_{d})$ are computed as homologies of a complex of the form
\[ \cdots\to\bigoplus_{i\in I_3}\C^-(d_i,c)\to\bigoplus_{i\in I_2}\C^-(d_i,c)\to\bigoplus_{i\in I_1}\C^-(d_i,c)\to\C^-(d,c) \]
If $\deg(c)>\deg(d)$ or $\deg(c)=\deg(d)$, but $c\ne d$, then the entire complex vanishes and so do the Tor spaces for all $n\ge 0$.
If $c=d$, then $\C^-(d,c)\cong k$ and all the other terms of the complex vanish. Consequently, $\Delta^{c}\otimes_{\C}\Delta_{d}\cong k$ and $\Tor_{n}^{\C}(\Delta^{c},\Delta_{d})=0$ for all $n\geqslant1$. If $\deg(c)<\deg(d)$, we may work dually to the above and consider a projective resolution $P^{\bullet}\rightarrow \Delta^c$ with $P^0=\C(-,c)$ and $P^n\cong\bigoplus_{j\in J_n}\C(-,c_j)$ where $\deg(c_j)<\deg(c)$ for all $j\in J_n$ and $n\geqslant1$. Then
\[ P^n\tens{\C}\Delta_d \cong \bigoplus_{j\in J_n}\C(-,c_j) \tens{\C}\Delta_d \cong \bigoplus_{j\in J_n}\Delta_{d}(c_j). \]
Since $\Delta_d \cong\C^+(d,-)$ as left $\C^+$--modules, the $k$--vector spaces $\Tor_{n}^{\C}(\Delta^{c},\Delta_{d})$ may be computed as homologies of a complex of the form
\[ \cdots\to\bigoplus_{j\in J_3}\C^+(d,c_j)\to\bigoplus_{j\in J_2}\C^+(d,c_j)\to\bigoplus_{j\in J_1}\C^+(d,c_j)\to\C^+(d,c).\]
Since $\deg(c_j)<\deg(c)<\deg(d)$ we deduce that $\Tor_{n}^{\C}(\Delta^{c},\Delta_{d})=0$ for all $n\geqslant 0$ which finishes the proof.
\end{proof}

As a consequence, we can prove a characterization of left $\C$--modules filtered by standard functors analogous to~\cite[Lemma 1.4]{DR}. For that purpose, we denote for a left $\C$--module $M$ and an ordinal $\alpha\leqslant\lambda$ by $\operatorname{Tr}_{\alpha}M\subseteq M$ the $\C$--submodule $M$ generated by the images of all morphisms $\C(c,-)\to M$ with $\deg(c)<\alpha$.
It follows that for each object $c$ of $\C$ with $\deg(c)<\alpha$, we have $\operatorname{Tr}_{\alpha}M(c)=M(c)$, so that $(M/\operatorname{Tr}_{\alpha}M)(c)=0$.

Note that the assignment $M\mapsto\operatorname{Tr}_{\alpha}M$ defines an additive endofunctor of the category of left $\C$--modules which preserves direct unions and sends epimorphisms to epimorphisms.
Thus, given a left $\C$--module $M$ and a submodule $N\subseteq M$, we obtain a short exact sequence
\[ 0 \to \operatorname{Tr}_{\alpha}M\cap N \to \operatorname{Tr}_{\alpha}M \to \operatorname{Tr}_{\alpha}(M/N) \to 0 \]
and clearly $\operatorname{Tr}_{\alpha}(N)\subseteq \operatorname{Tr}_{\alpha}M\cap N$. In particular,
\begin{equation}\label{eq:property-of-Tr}
\operatorname{Tr}_{\alpha}(M/N)=(\operatorname{Tr}_{\alpha}M)/N
\end{equation}
whenever $\operatorname{Tr}_{\alpha}N=N$. On the other hand, if $\operatorname{Tr}_{\alpha}(M/N)=0$ then any morphism from $\C(c,-)$ to  $M$ with $\deg(c)<\alpha$ factors through $N$, hence $\operatorname{Tr}_{\alpha}M=\operatorname{Tr}_{\alpha}N$.
Finally,
\[ 0=\operatorname{Tr}_{0}M\subseteq \operatorname{Tr}_{1}M\subseteq\cdots\subseteq \operatorname{Tr}_{\alpha}M\subseteq \operatorname{Tr}_{\alpha+1}M\subseteq\cdots \subseteq \operatorname{Tr}_{\lambda}M=M \]
is a filtration of $M$ for each $M$.

In the proof of the next result we will make use of the Crawley--J\o{}nsson--Warfield theorem~\cite[Thm.~26.5]{AF92} which states that if a module is a direct sum of countably generated modules with local endomorphism rings, then so is every direct summand. It is originally proved for modules over associative rings, but the same proof works for the category of left $\C$--modules $\mathcal{K}^{\C}$.

\begin{proposition}\label{prop:std_filtration}
Let $\C$ be a $k$--linear Reedy category with degree function $\mathrm{deg}\colon\mathrm{Obj}(\C)\rightarrow\lambda$ and $M$ be a left $\C$--module.
Then $M$ is filtered by standard functors if and only if for each $\alpha<\lambda$, the factor module $\operatorname{Tr}_{\alpha+1}M/\operatorname{Tr}_{\alpha}M$ is a direct sum of copies of standard modules $\Delta_c$ with $\deg(c)=\alpha$.
\end{proposition}

\begin{proof}
We only need to prove the `only if' part. Suppose $M$ has a filtration
\[ 0=M_0\subseteq M_1\subseteq \cdots\subseteq M_{\beta}\subseteq M_{\beta+1}\subseteq\cdots \subseteq M_{\sigma}=M.\]
with $M_{\beta+1}/M_\beta\in\operatorname{Add}(\{\Delta_c\mid c\in\operatorname{Obj}(c)\})$ for each $\beta<\sigma$. The various $\Delta_c$ have local endomorphism rings by Theorem~\ref{thm:Reedy_exceptional}(i), thus in fact, each $M_{\beta+1}/M_\beta$ is then a direct sum of standard modules by the Crawley--J\o{}nsson--Warfield theorem~\cite[Thm.~26.5]{AF92}, so we can refine the filtration and assume without loss of generality that each $M_{\beta+1}/M_\beta$ is isomorphic to $\Delta_{c_\beta}$ for some object $c_\beta$ of $\C$.

We prove the statement about $\operatorname{Tr}_{\alpha+1}M/\operatorname{Tr}_{\alpha}M$ by induction on the length $\sigma$ of the filtration. The case $\sigma=0$ being trivial, we first assume that $\sigma=\rho+1$ is an ordinal successor, so that we have a short exact sequence $0\to M_\rho\to M\to \Delta_{c_\rho}\to 0$ and also
\begin{equation}\label{eq:std_filtration}
0\to M_\rho/\operatorname{Tr}_{\deg(c_\rho)}M_\rho\to M/\operatorname{Tr}_{\deg(c_\rho)}M_\rho\to \Delta_{c_\rho}\to 0
\end{equation}
Note that $\operatorname{Tr}_{\deg(c_\rho)}(\Delta_{c_\rho}) = 0$.
Indeed, recall that $\Delta_{c_\rho}|_{\C^+}\cong\C^+(c_\rho,-)$ from the proof of Theorem~\ref{thm:Reedy_exceptional}, so that $\Hom_{\C}(\C(d,-),\Delta_{c_\rho})\cong\Delta_{c_\rho}(d)=0$ whenever $\deg(d)<\deg(c_\rho)$.
Thus, we have 
\begin{equation}\label{eq:std_filtration_unification}
\operatorname{Tr}_{\deg(c_\rho)}M_\rho=\operatorname{Tr}_{\deg(c_\rho)}M.
\end{equation}
Moreover, since by inductive hypothesis $M_\rho/\operatorname{Tr}_{\deg(c_\rho)}M_\rho$ is filtered by standard modules $\Delta_d$ with $\deg(d)\geqslant\deg(c_\rho)$, the sequence~\eqref{eq:std_filtration} splits thanks to Theorem~\ref{thm:Reedy_exceptional}(iii) and the Eklof lemma (recalled in Proposition~\ref{prop:eklof}). 
Thanks to~\eqref{eq:property-of-Tr} and~\eqref{eq:std_filtration_unification}, the split exact sequence~\eqref{eq:std_filtration} also induces a split exact sequence
\begin{equation}\label{eq:std_filtration_over_Tr_rho+1}
0\to \operatorname{Tr}_{\deg(c_\rho)+1}M_\rho/\operatorname{Tr}_{\deg(c_\rho)}M_\rho\to \operatorname{Tr}_{\deg(c_\rho)+1}M/\operatorname{Tr}_{\deg(c_\rho)}M\to \Delta_{c_\rho}\to 0
\end{equation}
and one also observes by comparing~\eqref{eq:std_filtration} and~\eqref{eq:std_filtration_over_Tr_rho+1} that the inclusion $M_\rho\subseteq M$ induces an isomorphism $M_\rho/\operatorname{Tr}_{\deg(c_\rho)+1}M_\rho\cong M/\operatorname{Tr}_{\deg(c_\rho)+1}M$.

All in all, we have proved that the canonical map
$\operatorname{Tr}_{\alpha+1}M_\rho/\operatorname{Tr}_{\alpha}M_\rho\to \operatorname{Tr}_{\alpha+1}M/\operatorname{Tr}_{\alpha}M$ is either isomorphism (if $\alpha\ne\deg(c_\rho)$) or a split inclusion with cokernel isomorphic to $\Delta_{c_\rho}$ (if $\alpha=\deg(c_\rho)$), so $\operatorname{Tr}_{\alpha+1}M/\operatorname{Tr}_{\alpha}M$ is a direct sum of copies of standard modules $\Delta_c$ with $\deg(c)=\alpha$ if $M_\rho$ has this property.

Let $\sigma$ is a limit ordinal and $\alpha<\lambda$. Then, by the previous paragraph and the fact that both $\operatorname{Tr}_\alpha$ and $\operatorname{Tr}_{\alpha+1}$ preserve direct unions, 
\[ 0=\frac{\operatorname{Tr}_{\alpha+1}M_0}{\operatorname{Tr}_{\alpha}M_0}\to \cdots\to \frac{\operatorname{Tr}_{\alpha+1}M_{\beta}}{\operatorname{Tr}_{\alpha}M_{\beta}}\to \frac{\operatorname{Tr}_{\alpha+1}M_{\beta+1}}{\operatorname{Tr}_{\alpha}M_{\beta+1}}\to\cdots\to \frac{\operatorname{Tr}_{\alpha+1}M_\sigma}{\operatorname{Tr}_{\alpha}M_\sigma}=\frac{\operatorname{Tr}_{\alpha+1}M}{\operatorname{Tr}_{\alpha}M}\]
is a filtration of $\operatorname{Tr}_{\alpha+1}M/\operatorname{Tr}_{\alpha}M$. All the inclusions $\operatorname{Tr}_{\alpha+1}M_{\beta}/\operatorname{Tr}_{\alpha}M_{\beta}\to \operatorname{Tr}_{\alpha+1}M_{\beta+1}/\operatorname{Tr}_{\alpha}M_{\beta+1}$ are split with cokernels isomorphic to zero or a standard module $\Delta_d$ with $\deg(d)=\alpha$. It follows (by transfinite induction on $\sigma$) that $\operatorname{Tr}_{\alpha+1}M/\operatorname{Tr}_{\alpha}M$ is a direct sum of such standard modules, as required.
\end{proof}

\begin{corollary}\label{cor:std_filtration_summands}
Let $\C$ be a $k$--linear Reedy category and suppose that $M\in\K^\C$ is filtered by standard modules. Then so is any direct summand of $M$.
\end{corollary}

\begin{proof}
If $N$ is a summand of $M$, then $\operatorname{Tr}_{\alpha+1}N/\operatorname{Tr}_{\alpha}N$ is a summand of $\operatorname{Tr}_{\alpha+1}M/\operatorname{Tr}_{\alpha}M$ for each $\alpha$.
By virtue of Proposition~\ref{prop:std_filtration}, we must prove that $\operatorname{Tr}_{\alpha+1}N/\operatorname{Tr}_{\alpha}N$ is a direct sum of standard modules corresponding to objects of $\C$ of degree $\alpha$. By assumption and Proposition~\ref{prop:std_filtration}, we know that $\operatorname{Tr}_{\alpha+1}M/\operatorname{Tr}_{\alpha}M$ is of this form. Since the various standard modules corresponding to objects of $\C$ of degree $\alpha$ have local endomorphism rings, by the Crawley--J\o{}nsson--Warfield theorem~\cite[Thm.~26.5]{AF92} we deduce that the summand $\operatorname{Tr}_{\alpha+1}N/\operatorname{Tr}_{\alpha}N$ satisfies the desired property.
\end{proof}

\subsection{Simple functors}
\label{subsec:simple_functors}

In order to relate linear Reedy categories to quasi-hereditary algebras properly, we also need to understand simple left $\C$--modules. We first recall a characterization of projective modules with a unique simple quotient.

\begin{definition}
Given be a small $k$--linear category $\C$, a left $\C$--module $P$ and a submodule $M\subseteq P$, we call $M$ \textit{superfluous} in $P$ if it satisfies the following property:
Whenever $L\subseteq P$ is another $\C$--submodule and $L+M=P$, then necessarily $L=P$.
\end{definition}

\begin{lemma}
\label{lem:projective-unique-simple-quotient}
Let $\C$ be a small $k$--linear category and $P\in\K^\C$ be a non-zero projective left $\C$--module (i.e.\ a direct summand of a direct sum of representable functors). Then the following are equivalent: 
\begin{enumerate}
\item The endomorphism ring $\End_\C(P)$ is local.
\item $P$ has a submodule which is both maximal and superfluous.
\item $P$ has a unique maximal $\C$--submodule (and hence also a unique simple quotient).
\end{enumerate}
\end{lemma}

\begin{proof}
This is \cite[Prop.~17.19]{AF92} combined with standard facts about projective covers.

Indeed, the equivalence (a)$\Leftrightarrow$(c) in \cite[Prop.~17.19]{AF92}, whose proof works also for rings with several objects, says that $\End_\C(P)$ is local if and only if $P$ is a projective cover of a simple left $\C$--module in $\K^\C$. By unraveling the definition of projective cover on \cite[p.~199]{AF92}, one sees that $P$ is a projective cover of a simple $\C$--module if and only if there is a submodule $M\subseteq P$ which is both maximal and superfluous. This proves the equivalence between (1) and (2).

To see that (2) is equivalent to (3), assume first that $M\subseteq P$ is maximal and superfluous. If $M'\subseteq P$ is another maximal submodule, we cannot have $M'+M=P$ (as this would imply $M'=P$), so necessarily $M'=M$. Thus, if $M\subseteq P$ is maximal and superfluous, then there are no other maximal submodules of $P$ except for $M$. If conversely $M$ is a unique maximal submodule of $P$ and $L\subsetneq P$ is a proper submodule, then $L\subseteq M$ by the maximality of $M$, so $L+M\subsetneq P$. Hence $M$ is superfluous.
\end{proof}

As an immediate consequence, we obtain the following property of standard modules over linear Reedy categories.

\begin{lemma}
\label{lem:top_of_Delta}
Let $\C$ be a linear Reedy category. Then 
for any object $c\in\C$, the standard left $\C$--module $\Delta_c$ has a unique maximal $\C$--submodule (and hence also a unique simple quotient).
\end{lemma}

\begin{proof}
Let $\alpha=\deg(c)$. By Remark~\ref{rem:Reedy_quotient}, $\Delta_c$ can be interpreted as the projective left $\C/\I^\C_{<\alpha}$--module $\C/\I^\C_{<\alpha}(c,-)$. Since $\End_{\C}(\Delta_c)=\End_{\C/\I^\C_{<\alpha}}(\Delta_c)$ is local by Theorem~\ref{thm:Reedy_exceptional}(i), it follows from Lemma~\ref{lem:projective-unique-simple-quotient} that $\Delta_c=\C/\I^\C_{<\alpha}(c,-)$ has a unique maximal submodule as a $\C/\I^\C_{<\alpha}$--module, and so also as a $\C$--module. In fact, this maximal submodule is by \S\ref{subsec:simple_functors_prelim} none other than $\rad_{\C/\I^\C_{<\alpha}}(c,-)$.
\end{proof}

\begin{remark}
As we already know, one can for each $d\in\mathrm{Obj}(\C)$ identify $\Delta_c(d)\cong\C^+(c,d)$ as $k$--vector spaces.
Under that identification, we claim that $\rad_{\C/\I^\C_{<\alpha}}(c,d)$ identifies with the subspace formed by those $f\colon c\to d$ such that $gf\in\I^\C_{<\alpha}(c,c)$ for each $g\in\C^-(d,c)$, that is, $gf=0$ in $\C/\I^\C_{<\alpha}(c,c)$.

Since any $g\colon d\rightarrow c$ in $\C/\I^\C_{<\alpha}(d,c)$ is represented by a morphism in $\C^-(d,c)$ by Remark~\ref{rem:Reedy_quotient}, this subspace is contained in $\rad_{\C/\I^\C_{<\alpha}}(c,d)$ by definition of the radical (recalled in \S\ref{subsec:radicals}).
On other hand, if $f\colon c\to d$ is contained in 
 $\rad_{\C/\I^\C_{<\alpha}}(c,d)\subseteq \C/\I^\C_{<\alpha}(c,d)\subseteq \C^+(c,d)$, then we have that $gf$ is non-invertible in $\End_{\C/\I^\C_{<\alpha}}(c)\cong k$ for each $g\in\C^-(d,c)$. 
Hence $gf$=0 in $\C/\I^\C_{<\alpha}(c,c)$ for each $g\in\C^-(d,c)$. 

\end{remark}

Given an object of $\C$, we denote the simple quotient of $\Delta_c$ by $L_c$. Now we can classify simple objects in a linear Reedy category.

\begin{theorem}
\label{thm:Reedy_simples}
Let $\C$ be a $k$--linear Reedy category. Then the assignment $c\mapsto L_c$ provides a bijection between
\begin{itemize}
\item[(i)] objects of $\C$ and
\item[(ii)] isomorphism classes of simple left $\C$--modules.
\end{itemize}
\end{theorem}

\begin{proof}
We must prove that each simple left $\C$--module $L$ is isomorphic to one of the form $L_c$ for a unique object $c$ of $\C$. 
To this end, let $c$ be an object of minimal degree such that there is a non-zero (hence surjective) homomorphism of left $\C$--modules $p\colon\C(c,-)\to L$.
Note that the submodule $\I^\C_{<\alpha}(c,-)$ of $\C(c,-)$ coincides with the sum of the images of all morphisms of left $\C$--modules $\C(d,-)\to\C(c,-)$, where $d$ runs over all objects of degree $<\alpha$. In particular, $p$ factors through the canonical surjection $\C(c,-)\to\Delta_c$ and there is an epimorphism $\Delta_c\to L$. This implies that $L_c\cong L$ by the uniqueness part of Lemma~\ref{lem:top_of_Delta}.

Next suppose that $L_c\cong L_d$ as left $\C$--modules. Then $c$ and $d$ have the same degree $\alpha$ by the previous paragraph. Moreover, both $\Delta_c$ and $\Delta_d$ are projective left $\C/\I^\C_{<\alpha}$--modules by Remark~\ref{rem:Reedy_quotient} and have local endomorphism rings by Theorem~\ref{thm:Reedy_exceptional}(i). Hence, by \cite[Prop.~17.19]{AF92} $\Delta_c$ and $\Delta_d$ are projective 
covers of $L_c\cong L_d$ in $\K^{\C/\I^\C_{<\alpha}}$, so they are isomorphic as left $\C/\I^\C_{<\alpha}$--modules. By the Yoneda lemma, it follows that $c\cong d$ in $\C/\I^\C_{<\alpha}$, but as both the objects are of the lowest degree in $\C/\I^\C_{<\alpha}$ as a linear Reedy category, this implies $c=d$ by the Reedy factorization.
Thus also $c=d$ in $\C$ by the construction of $\C/\I^\C_{<\alpha}$ (Remark~\ref{rem:Reedy_quotient}).
\end{proof}

\begin{remark}\label{rem:degree_of_simples}
The bijection of Theorem~\ref{thm:Reedy_simples} allows us to define a degree $\deg(L)$ of a simple left $\C$--module $L$ as the degree $\deg(c)$ of the object $c\in\mathrm{Obj}(\C)$ such that $L\cong L_c$.
\end{remark}

\begin{remark}\label{rem:degree-of-std-composition-factors}
Given $c\in\mathrm{Obj}(\C)$, we have $\Delta_c(c)\cong\C^+(c,c)\cong k$. Since $L_c$ is a factor of $\Delta_c$, we also have $L_c(c)\cong k$. Thus, the kernel $K_c$ of the canonical projection $\Delta_c\to L_c$ vanishes on all objects $d$ of $\C$ such that $\deg(d)\leqslant\deg(c)$. In particular, if $L_d$ is a simple subfactor of $K_c$, then $\deg(d)>\deg(c)$.
\end{remark}

\subsection{Finite Reedy categories}
\label{subsec:finite_Reedy}
We call a $k$--linear Reedy category \textit{finite} if it contains only a finite number of objects and is hom-finite (i.e. the $k$--vector space of morphisms between any two objects is finite dimensional). We assume that degree functions of finite Reedy categories map to non-negative integers.

For a finite Reedy category $\C$ with objects $c_0,...,c_n$ we write $P:=\oplus_{i=1}^{n}\C(c_i,-)$ for the finitely generated projective generator of the functor category $\mathcal{K}^{\C}$. It is well known, see e.g. \cite[II~Thm.~1.3]{Bass}, that the functor $\Hom_{\C}(P,-)$ induces an equivalence from the category $\mathcal{K}^{\C}$ to the category of left modules over the finite dimensional algebra $A:=\mathrm{End}(P)\op$.
This leads us to the following concept.

\begin{definition}
\label{def:Reedy_algebra}
Let $A$ be a finite dimensional $k$--algebra with a complete set of orthogonal idempotents $\{e_0,\dots,e_n\}$. 
Then $A$ is called \textit{Reedy} if there is a function $\mathrm{deg}\colon\{e_0,\dots,e_n\}\rightarrow\mathbb{N}$ and subalgebras $A^+$ and $A^-$ of $A$ containing the idempotents $e_0,\dots,e_n$, such that:
\begin{itemize}
\item[-] The subalgebra $A^+$ satisfies $e_iA^+e_i\cong k$ for all $i$, and for all $i\neq j$ the following implication holds:
$e_jA^+e_i\neq 0 \Rightarrow \deg(e_j)>\deg(e_i).$
\item[-] The subalgebra $A^-$ satisfies $e_iA^-e_i\cong k$ for all $i$, and for all $i\neq j$ the following implication holds:
$e_jA^-e_i\neq 0 \Rightarrow \deg(e_j)<\deg(e_i).$
\item[-] For any $i, j$ the multiplication in $A$ induces an isomorphism of $k$--vector spaces
\[
\bigoplus_{l=0}^{n} \,e_jA^+e_l\, \otimes_{k}\, e_lA^-e_i \xrightarrow{\cong} e_jAe_i.
\]
In other words, if we denote by $A^0$ the semisimple algebra $A^0 := A^+\cap A^-=\oplus_{i=0}^n k\cdot e_i$, the last condition says that the multiplication map induces an isomorphism of $k$--vector spaces
\[ A^+ \otimes_{A^0} A^- \xrightarrow{\cong} A. \]
\end{itemize}
\end{definition}

\begin{remark} \label{rem:Reedy-alg-non-primitive}
Let us emphasize here that the idempotents $e_0,\dots,e_n$ in the above definition need not be primitive and a Reedy algebra may easily be non-basic. This is because representable functors on a (even finite) $k$--linear Reedy category need not be indecomposable and they may have summands isomorphic to other representable functors.

The simplest example of this phenomenon is the finite $k$--linear Reedy category $\C=k\mathbf{\Delta}_{\le1}\cong kQ_I$ from Example~\ref{ex:linear-simplices}. In that case the morphism $\C(s,-)\colon\C([0],-)\to\C([1],-)$ is a section, so $\C([1],-)$ has a non-trivial direct summand isomorphic to $\C([0],-)$.
\end{remark}

We also remark that to any Reedy finite dimensional $k$--algebra $A$ with a complete set of idempotents $\{e_0,\dots,e_n\}$ and a degree function $\deg\colon\{e_0,\dots,e_n\}\to\mathbb{N}$ we can associate a $k$--linear Reedy category $\C_{A}$, whose objects are $e_0,...,e_n$ and morphisms given by the rule $\Hom_{\C_{A}}(e_i,e_j):=e_jAe_i$. Thus, following the comments before Definition~\ref{def:Reedy_algebra}, if $P:=\oplus_{i=0}^{n}\C_{A}(e_i,-)$ we obtain an equivalence of categories 
\begin{equation}
\label{eq:equivalence}
\Hom_{\C_{A}}(P,-)\colon \mathcal{K}^{\C_A}\rightarrow\mathrm{Mod}(A).
\end{equation}

In this way, we can transfer the  results already obtained for the category $\mathcal{K}^{\C_A}$ to the category of (left) modules over the Reedy algebra $A$. We will apply this observation in Theorem \ref{thm:Reedy_is_qh} to prove that Reedy algebras are quasi-hereditary algebras with exact Borel subalgebras.
We first recall a few concepts from the theory of quasi-hereditary algebras \cites{CPS,DR}.

We fix a finite dimensional $k$--algebra $A$ and denote by $L(i)$, where $i$ runs through a suitable finite set $\Lambda$, representatives of isomorphism classes of all simple $A$--modules and by $P(i)$ their projective covers.
Moreover, following~\cites{DR,Conde18}, we assume that $\Lambda$ is endowed with a partial ordering denoted by $\unlhd$. For each $i\in\Lambda$, we define the \textit{standard module} $\Delta(i)$ as the largest quotient of $P(i)$ with composition factors only of the form $L(j)$ for $j\unlhd i$. We point out that $\Delta(i)=P(i)/\sum_{j\not\unlhd i}\operatorname{Tr}_{P(j)} P(i)$, where the module in the denominator denotes the submodule of $P(i)$ generated by the images of all homomorphisms from $P(j)$ to $P(i)$. We denote by $\pi_i\colon P(i)\rightarrow\Delta(i)$ the natural surjection.
With this notation we give the following:

\begin{definition}[{\cite[Def.~2.11]{Conde18}}]\label{def:qh}
The algebra $(A,\unlhd)$ is called \textit{quasi-hereditary} if for all $i\in\Lambda$ the following hold:
\begin{itemize}
\item[(i)] $L(i)$ occurs exactly once in the composition series of $\Delta(i)$, i.e., the kernel of the canonical surjection $\Delta(i)\to L(i)$ admits a finite filtration $0=M_{-1}\subseteq M_0\subseteq M_1\subseteq\cdots\subseteq M_k$ where for all $\alpha=0,...,k$ the module $M_{\alpha}/M_{\alpha-1}$ is isomorphic to some $L(j)$, for $j\lhd i$.
\item[(ii)] The kernel of the surjection $\pi_i\colon P(i)\rightarrow\Delta(i)$ is filtered by $\{\Delta(j)\, |\, j\rhd i\}$, i.e., it admits a finite filtration $0=N_{-1}\subseteq N_0\subseteq N_1\subseteq\cdots\subseteq N_\ell$ where for all $\alpha=0,...,\ell$ the module $N_{\alpha}/N_{\alpha-1}$ is isomorphic to some $\Delta(j)$, for $j\rhd i$.
\end{itemize}
\end{definition}

The following concept was introduced in \cite[\S2]{Koe}.

\begin{definition}
\label{def:Borel}
Let $(A,\unlhd)$ be a quasi-hereditary finite dimensional $k$--algebra and let $B$ be a subalgebra of $A$ having the same number of simple modules as $A$; we denote them $S(i)$, $i\in\Lambda$. Then $B$ is called an \textit{exact Borel subalgebra} of $A$ if the following hold:
\begin{itemize}
\item[(i)] The algebra $(B,\unlhd)$ is quasi-hereditary with simple standard modules.
\item[(ii)] The induction functor $A\otimes_{B}-$ from $B$--modules to $A$--modules is exact.
\item[(iii)] For each simple $B$--module $S(i)$ there is an isomorphism $\Delta(i)\cong A\otimes_{B}S(i)$.
\end{itemize}
\end{definition}

We point out that quasi-hereditary algebras can be also defined via the concept of \textit{heredity chains}, which are certain filtrations of the algebra by two-sided ideals generated by idempotents, see \cite{DR2}. In a heredity chain, the idempotents appearing at the bottom of the filtration are maximal in the quasi-hereditary ordering. 
We are going to prove in Theorem~\ref{thm:Reedy_is_qh} below that Reedy algebras are quasi-hereditary.
However, for Reedy algebras it will be more convenient to have the idempotents of lowest degree at the bottom of the filtration, cf.~Example~\ref{ex:qh}. This justifies the ordering introduced in Theorem~\ref{thm:Reedy_is_qh}.

Let now $A$ be a Reedy algebra with a function $\deg\colon\{e_0,...,e_n\}\rightarrow\mathbb{N}$. We temporarily denote the standard $A$-modules, that is images of the standard functors from Definition~\ref{def:standards} under the equivalence of categories displayed in~\eqref{eq:equivalence}, by $\tilde\Delta(i)$. That is, $\tilde{\Delta}(i)=Ae_i/\sum_{\deg(e_j)<\deg(e_i)}\operatorname{Tr}_{Ae_j} Ae_i$. We know from Theorem~\ref{thm:Reedy_simples} that each $\tilde{\Delta}(i)$ has a unique simple factor which we denote by $L(i)$ (as for quasi-hereditary algebras) and all simple left $A$--modules arise in this way. Again, the projective cover of $L(i)$ will be denoted by $P(i)$. As we know from Remark~\ref{rem:Reedy-alg-non-primitive}, the projective left $A$--modules $Ae_i$ may be decomposable, so $P(i)$ is in general only a direct summand of $Ae_i$. The following lemma makes more precise what the complement looks like.

\begin{lemma}\label{lem:Reedy-alg-decomposition-of-representables}
Let $A$ be a Reedy algebra with a degree function $\deg\colon\{e_0,...,e_n\}\rightarrow\mathbb{N}$, let $i\in\{0,\dots,n\}$, and let $P(i)$ be a projective cover of $L(i)$ (using the notation just above). Then we have a direct sum decomposition $Ae_i \cong P(i) \oplus Q(i)$, where all indecomposable summands of $Q(i)$ must be of the form $P(j)$ with $\deg(j)<\deg(i)$.
\end{lemma}

\begin{proof}
Recall that we have a short exact sequence of the form
\[
\xymatrix@1{
0 \ar[r] & K(i) \ar[r]^-{\subseteq} & Ae_i \ar[r]^-{\rho_i} & \tilde{\Delta}(i) \ar[r] & 0
},
\]
where $K(i):=\sum_{\deg(e_j)<\deg(e_i)}\operatorname{Tr}_{Ae_j} Ae_i$ is filtered by $\tilde{\Delta}(j)$'s for some $j\in\{0,\dots,n\}$ with $\deg(j)<\deg(i)$ by Theorem~\ref{thm:proj_by_standards}.
Since the functor assigning to a left $A$--module $M$ the semisimple quotient factor $M/\rad(M)$ is right exact and $\tilde{\Delta}(j)/\rad\tilde{\Delta}(j)\cong L(j)$ for each $j$ by Lemma~\ref{lem:top_of_Delta}, it follows that $K(i)/\rad K(i)$ has only composition functors of the form $L(j)$ with $\deg(j)<\deg(i)$.

The projective cover $P(i)$ of $L(i)$ is also a projective cover of $\tilde{\Delta}(i)$ (again by Lemma~\ref{lem:top_of_Delta}); write $f\colon P(i)\rightarrow \tilde{\Delta}(i)$ for this projective cover. 
Hence there exists a homomorphism $\pi\colon Ae_i\rightarrow P(i)$ such that $f\circ \pi =\rho_i$. Since $f$ is a projective cover and $\rho_i$ is surjective we deduce that $\pi$ is also surjective by \cite[Corollary~5.15]{AF92}. 
Thus, we have a decomposition $Ae_i=P(i)\oplus Q(i)$, where we can without loss of generality assume that $P(i)$ is a submodule of $Ae_i$, and $Q(i)$ is a complement of $P(i)$ contained in $K(i)$. In particular, $Q(i)$ is a summand of $K(i)$, so $Q(i)/\rad Q(i)$ only has composition factors of the form $L(j)$ with $\deg(j)<\deg(i)$. However, $Q(i)$ being projective, we know that $Q(i)\twoheadrightarrow Q(i)/\rad Q(i)$ is a projective cover, so all indecomposable summands of $Q(i)$ must be of the form $P(j)$ with $\deg(j)<\deg(i)$.
\end{proof}

Now we can prove one of the main theorems of this paper.
We would like to warn the reader the established ordering conventions between Reedy categories and quasi-hereditary algebras do not agree.
The natural ordering on the simples of a Reedy algebra is that imposed by the degree, while in the quasi-hereditary ordering a simple of lowest Reedy degree will be of maximal weight and vice versa, cf. Example~\ref{ex:qh}.

\begin{theorem}
\label{thm:Reedy_is_qh}
Let $A$ be a Reedy finite dimensional algebra as in Definition~\ref{def:Reedy_algebra} and order the simple modules of $A$ by setting $L(i)\lhd L(j)$ if and only if $\deg(e_i)>\deg(e_j)$. Then $(A,\unlhd)$ is a quasi-hereditary algebra and $A^-$ is an exact Borel subalgebra.
\end{theorem}

\begin{proof}
Under the equivalence of categories displayed in~\eqref{eq:equivalence}, for each $i=0,...,n$ the $\C_A$--modules $\C_{A}(e_i,-)$, $\C_{A}^-(e_i,-)$, $\Delta_{e_i}$ and $L_{e_i}$ are sent to the $A$--modules $Ae_i$, $A^-e_i$, $\tilde{\Delta}(i)$ and $L(i)$, respectively.

We claim that $\tilde{\Delta}(i)\cong\Delta(i)$, where the latter has been introduced in the paragraph just above Definition~\ref{def:qh}. Recall that
\[
\tilde{\Delta}(i) = Ae_i/\sum_{\deg(e_j)<\deg(e_i)}\operatorname{Tr}_{Ae_j} Ae_i
\]
while
\[
\Delta(i) = P(i)/\sum_{j\not\unlhd i}\operatorname{Tr}_{P(j)} P(i),
\]
where $P(i)$ is a projective cover of the simple module $L(i)$.
First of all, notice that if $j\in\{0,\dots,n\}$ is distinct from $i$ and $\deg(e_j)=\deg(e_i)$, then $\Hom_A(Ae_j,Ae_i)\cong e_jAe_i=\bigoplus_{\deg(l)<\deg(i)} \,e_jA^+e_l\, \otimes_{k}\, e_lA^-e_i$, so $\operatorname{Tr}_{Ae_j}Ae_i\subseteq\sum_{\deg(l)<\deg(i)}\operatorname{Tr}_{Ae_l}Ae_i$. Hence
\[
\tilde{\Delta}(i) = Ae_i/\sum_{j\not\unlhd i}\operatorname{Tr}_{Ae_j} Ae_i
\]
by the definition of the partial order $\lhd$.
Next, if we decompose $Ae_i=P(i)\oplus Q(i)$ as in Lemma~\ref{lem:Reedy-alg-decomposition-of-representables}, then clearly $Q(i)\subseteq\sum_{j\not\unlhd i}\operatorname{Tr}_{Ae_j} Ae_i$, so that
\[ \tilde{\Delta}(i) = P(i)/\sum_{j\not\unlhd i}\operatorname{Tr}_{Ae_j} P(i). \]

By Lemma~\ref{lem:Reedy-alg-decomposition-of-representables}, the direct sum $Q=\oplus_{j\not\unlhd i}Ae_j$ is isomorphic to a direct sum of copies of the projective modules $P(j)$ with ${j\not\!\!\unlhd\,\,i}$, and each such $P(j)$ indeed appears as a summand of $Q$. Hence $\sum_{j\not\unlhd i}\operatorname{Tr}_{Ae_j} P(i) = \operatorname{Tr}_{Q} P(i) = \sum_{j\not\unlhd i}\operatorname{Tr}_{P(j)} P(i)$. This proves the claim.

Keeping this in mind, condition (i) in Definition~\ref{def:qh} follows immediately from Remark~\ref{rem:degree-of-std-composition-factors}. Regarding condition (ii), recall that the projective cover $P(i)$ of $L(i)$ is also a projective cover of $\Delta(i)$ by Lemma~\ref{lem:top_of_Delta}. Moreover, if $f\colon P(i)\to\Delta(i)$ is the projective cover as a morphism, we can identify $f$ with a restriction of the canonical projection $Ae_i\to\Delta(i)$ in such a way that the kernel of $f$ is a summand of the kernel of $Ae_i\to\Delta(i)$.
We discussed this in the proof of Lemma~\ref{lem:Reedy-alg-decomposition-of-representables}; see also~\cite[Lemma~17.17]{AF92}. 
Since the kernel of the canonical projection $Ae_i\to\Delta(i)$ has a filtration by standard modules $\Delta(j)$ with $j\rhd i$ by Theorem~\ref{thm:proj_by_standards},
the kernel of $f$ has such a filtration too, by (the proof of) Corollary~\ref{cor:std_filtration_summands}. This shows that $(A,\unlhd)$ is quasi-hereditary.
 
We may apply the same proof to the Reedy algebra $A^-$, whose simple modules are indexed by the same set indexing the simple modules of $A$ by Theorem~\ref{thm:Reedy_Borel}(i), Theorem~\ref{thm:Reedy_simples} and the equivalence of categories displayed in~\eqref{eq:equivalence}. Hence, $(B,\unlhd)$ is quasi-hereditary with simple standard modules.
The induction functor $A\otimes_{A^-}-$ is exact by Theorem~\ref{thm:Reedy_Borel}(ii) and sends simple $A^-$--modules to standard $A$--modules by Theorem~\ref{thm:Reedy_Borel}(iii). Hence, $A^-$ is an exact Borel subalgebra of $A$.
\end{proof}

\section{Functors indexed by linear Reedy categories}
\label{sec:Fun}

In the classical (non-additive) theory of Reedy categories, functors with source Reedy categories are usually constructed by induction on the degree, see for instance \cite[Ch.~15]{Hir}. Here we transfer some of these basic facts to the realm of $k$--linear categories. We follow closely the elegant exposition of Riehl and Veriti \cite{riehl}.

\begin{definition}\textnormal{(Filtrations of Reedy categories).}
\label{def:filtrations}
Let $\class C$ be a small $k$--linear Reedy category with a degree function $\deg\colon\mathrm{Obj}(\class C)\rightarrow \lambda$. For every (non-zero) ordinal $\alpha\leqslant\lambda$ we define $\C_{<\alpha}$ as the full subcategory of $\C$ formed by objects of degree strictly smaller than $\alpha$. 
In the case of a successor ordinal $\alpha=\beta+1$ sometimes we also denote $\C_{<\alpha}$ by $\C_{\leqslant\beta}$.
\end{definition}

It is not hard to see that the subcategories just defined are in fact full Reedy subcategories of $\C$. 
In this way we obtain a filtration of $\C$ by full $k$--linear Reedy subcategories, which can be (partially) depicted by 
  \[\C_{\leqslant 0}\subseteq\C_{\leqslant 1}\subseteq\cdots\C_{<\alpha}\subseteq\C_{<\alpha+1}\subseteq\cdots\subseteq\C_{<\lambda}=\C.\]

\begin{ipg}\textbf{Setup.}
\label{setup}
In the rest of the paper we fix the following data, which is a special case of what we considered in Section~\ref{sec:preliminaries}.
\begin{itemize}
\item[•] $k$ is a field and $\M$ is a bicomplete $k$--linear abelian category.
\item[•] $\C$ is a $k$--linear Reedy category with a degree function $\deg\colon\mathrm{Obj}(\C)\rightarrow \lambda$.
\item[•] $\M^{\C}$ is the category of $k$--linear functors from $\C$ to $\M$.
\end{itemize}
\end{ipg}

For an ordinal $\alpha <\lambda$, we consider the fully faithful embedding $i\colon\class C_{<\alpha}\rightarrow \class C$. The restriction functor $\res_{\alpha}\colon \mathcal{M}^{\mathcal{C}}\rightarrow \mathcal{M}^{\mathcal{C_{<\alpha}}}$, admits a left adjoint $\sk_{\alpha}$, and a right adjoint $\cosk_{\alpha}$ (the Kan extension functors), see for instance \cite[Sect.~4]{Murfet_ringoids}. Pictorially:
\begin{equation}
\label{eq:adj_triple}
 \xymatrix@C=2pc{
\mathcal{M}^{\mathcal{C}} \ar[rr]|-{\res_{\alpha}} & &  \mathcal{M}^{\mathcal{C_{<\alpha}}} \ar@/_0.9pc/[ll]_-{\mathrm{sk}_{\alpha}} \ar@/^0.9pc/[ll]^-{\mathrm{cosk}_{\alpha}}.
}
\end{equation}
We briefly recall how these functors are defined. Fix an object $c$ in $\mathcal{C}$. If $X$ is a functor in $\class M^{\class C_{<\alpha}}$, then we define
\begin{equation}
\label{eq:sk}
\sk_{\alpha}X(c) := \C(i(-),c)\tens{\class C_{<\alpha}}X,
\end{equation}
where the right hand side of (\ref{eq:sk}) is just the tensor product of functors as in \S\ref{subsec:ringoids}; taken over the subcategory $\class C_{<\alpha}$. To further illustrate the situation, if $\M$ is the category of left $A$--modules over a $k$--algebra $A$, the counit $l^{\alpha}$ of the adjunction $(\sk_{\alpha},\res_{\alpha})$, evaluated at any $Y\in\M^\C$ and $c\in\C$, is simply the morphism 
\begin{equation}
\label{eq:counit}
l^{\alpha,Y}_{c}:=l_{c}^{Y}\colon\C(i(-),c)\tens{\class C_{<\alpha}} \res_\alpha(Y) \longrightarrow Y(c); \quad \phi\otimes\xi\mapsto Y(\phi)(\xi).
\end{equation}
For general $\M$, the counit can be similarly constructed using for each object $d$ of $\class C_{<\alpha}$ the composition $\C(c,d)\otimes Y(c)\to \M(Y(c),Y(d))\otimes Y(c)\to Y(d)$, where the first part is given by the action of $Y$ on morphisms, $\C(c,d)\to\M(Y(c),Y(d))$, and the second one is the evaluation morphism as constructed in~\S\ref{subsec:ringoids}. We will leave out details as we will obtain a more convenient description of the counit in Theorem~\ref{thm:standard-(co)induction}.

Analogously, the right adjoint to the restriction can be defined as
\begin{equation}
\label{eq:cosk}
\cosk_{\alpha}X(c):=\hom_{\mathcal{C_{<\alpha}}}(\class C(c,i(-)),X),
\end{equation}
where $\hom_{\mathcal{C_{<\alpha}}}$ denotes the weighted limit construction as in \S\ref{subsec:ringoids}.
If $\M$ happens to be the category of left $A$--modules over a $k$--algebra $A$, $\hom_{\mathcal{C_{<\alpha}}}$ coincides with the vector space of all natural transformations from $\class C(c,i(-))$ to $X$, which are both functors with source the subcategory $\C_{<\alpha}$. 
The unit $m^{\alpha}$ of the adjunction $(\res_{\alpha},\cosk_{\alpha})$, evaluated at $Y\in\M^\C$ and $c\in\C$, is in this case the morphism
\begin{equation}
\label{eq:unit}
m^{\alpha,Y}_{c}:=m^{Y}_{c}\colon Y(c)\rightarrow \hom_{\mathcal{C_{<\alpha}}} \left(\class C(c,i(-)),Y \right); \quad \xi \mapsto [f\mapsto Y(f)(\xi)].
\end{equation}
For an abstract bicomplete $k$--linear category $\M$, it is possible to construct the unit using the compositions
$Y(c)\to\hom\big(\M(Y(c),Y(d)),Y(d)\big)\to\hom\big(\C(c,d),Y(d)\big)$ of the unit of adjunction mentioned at the end of~\S\ref{subsec:enriched} with the action of $Y$ on morphisms. 
A more convenient general description of $m^{\alpha,Y}_{c}$ will be again obtained in Theorem~\ref{thm:standard-(co)induction}.

Since the functor $i\colon\class C_{<\alpha}\rightarrow \class C$ is fully faithful, the functors $\sk_{\alpha}$ and $\cosk_{\alpha}$ are also fully faithful \cite[Prop.~4.23]{Kelly}. Thus the unit $\eta^{\alpha}$ of $(\sk_{\alpha},\res_{\alpha})$ and the counit $\epsilon^{\alpha}$ of $(\res_{\alpha},\cosk_{\alpha})$ are both isomorphisms. They induce a natural transformation $\tau^{\alpha}\colon \sk_{\alpha}\rightarrow \cosk_{\alpha}$, which can be, by Lemma~\ref{lem:jump-over-middle-adjoint}, equivalently given as the map
\[\sk_{\alpha}\xrightarrow{m^{\alpha}\circ\,\sk_{\alpha}}\cosk_{\alpha}\circ\res_{\alpha}\circ\sk_{\alpha}\xrightarrow{\cosk_{\alpha}\circ\,(\eta^{\alpha})^{-1}}\cosk_{\alpha},\]
or the map
\[\sk_{\alpha}\xrightarrow{\sk_{\alpha}\circ\,(\epsilon^{\alpha})^{-1}}\sk_{\alpha}\circ\res_{\alpha}\circ\cosk_{\alpha}\xrightarrow{l^{\alpha}\circ\,\cosk_{\alpha}}\cosk_{\alpha}.\]

\begin{remark}
\label{rem:restricted_adj}
For an ordinal $\alpha <\lambda$, we may use a restricted version of the adjoint triple (\ref{eq:adj_triple}) relative to the embedding $\C_{<\alpha}\rightarrow \C_{<\alpha+1}$, that is, an adjoint triple
\begin{equation}
\label{eq:adj_triple_relative}
 \xymatrix@C=2pc{
\mathcal{M}^{\mathcal{C}_{<\alpha+1}} \ar[rr]|-{\res_{\alpha}} & &  \mathcal{M}^{\mathcal{C}_{<\alpha}} \ar@/_0.9pc/[ll]_-{\mathrm{sk}_{\alpha}} \ar@/^0.9pc/[ll]^-{\mathrm{cosk}_{\alpha}},
}
\end{equation}
where we keep the same notation as in (\ref{eq:adj_triple}). The above discussion for the adjunction (\ref{eq:adj_triple}) carries over. It will be clear from context whether we use (\ref{eq:adj_triple}) or (\ref{eq:adj_triple_relative}) when referring to the functors $\mathrm{sk}_{\alpha}$, $\mathrm{cosk}_{\alpha}$ and the natural transformation $\tau^{\alpha}_c\colon \sk_{\alpha}\rightarrow \cosk_{\alpha}$.
\end{remark}

Given a functor from $(\C_{<\alpha+1},\M)$, we have a factorisation $\sk_{\alpha}X\rightarrow X\rightarrow \cosk_{\alpha}X$ of $\tau^{\alpha,X}$ by Lemma~\ref{lem:abstract-latching-matching-factorisation}. In particular, for each object $c$ of degree $\alpha$, this restricts to a factorisation $\sk_{\alpha}X(c)\rightarrow X(c)\rightarrow \cosk_{\alpha}X(c)$ of the map $\tau^{\alpha,X}_{c}$.

Now we are going to focus on when and in how many ways we can extend a functor from $(\C_{<\alpha},\M)$ to $(\C_{<\alpha+1},\M)$. We show that all that is necessary is, for every object $c$ of degree $\alpha$, to choose such a factorisation $\sk_{\alpha}X(c)\rightarrow X(c)\rightarrow \cosk_{\alpha}X(c)$ of the canonical map $\tau^{\alpha,X}_{c}$. In the case of ordinary Reedy categories this condition is sufficient, see \cite[Theorem~15.2.1]{Hir} or \cite[Lemma~3.10]{riehl}. We need to carry over this fact to the $k$--linear context.

\begin{proposition}\textnormal{(cf.~\cite[3.10/3.11]{riehl})}
\label{prop:Riehl_Veriti}
Let $X$ be a $k$--linear functor in $(\class C_{<\alpha},\M)$ for some ordinal $\alpha<\lambda$. Then a family of factorizations $\sk_{\alpha}X(c)\rightarrow X(c)\rightarrow \cosk_{\alpha}X(c)$ of the canonical maps $\tau^{\alpha,X}_{c}$, indexed by all objects $c$ of degree $\alpha$, uniquely determines a functor in $(\class C_{<\alpha+1},\M)$ that coincides with $X$ when restricted to $\class C_{<\alpha}$.

In addition, given $X$ and $Y$ in $(\C,\M)$, an extension of a natural transformation $\phi\colon X_{<\alpha}\rightarrow Y_{<\alpha}$ to a natural transformation $\phi\colon X_{\leqslant\alpha}\rightarrow Y_{\leqslant\alpha}$ uniquely corresponds to a family $\{\phi_{c}\colon X(c)\rightarrow Y(c)\, |\, \deg(c)=\alpha\}$ of morphisms in $\M$ such that the following diagrams commute:
\begin{displaymath}
 \xymatrix@C=2pc{
 \sk_{\alpha}X_{<\alpha}(c)  \ar[r]^-{l_c^X} \ar[d]_-{\sk_{\alpha}\phi_{c}} & X(c)\ar[r]^-{m_c^X} \ar[d]^-{\phi_c} & \cosk_{\alpha}X_{<\alpha}(c)\ar[d]^-{\cosk_{\alpha}\phi_{c}} \\
\sk_{\alpha}Y_{<\alpha}(c) \ar[r]^-{l_c^Y} & Y(c)\ar[r]^-{m_c^Y} & \cosk_{\alpha}Y_{<\alpha}(c).
 }
 \end{displaymath}
\end{proposition}

\begin{proof}
From the given data the values $X(c)$ for $c$ of degree $\alpha$ are given. We need to  define how $X$ acts on the morphisms of $\C_{\leqslant\alpha}(c,d)$ where $\deg(c)=\alpha$ or $\deg(d)=\alpha$.
If both the degrees equal $\alpha$, then both $\C^+(c,d)$ and $\C^-(c,d)$ are either $0$ or $k$ and we have no choice for the action of $X$ on them from $k$--linearity.

Consider a non-zero non-endomorphism $f\colon c\rightarrow d$ in $\C^+(c,d)$ with $\deg(c)<\alpha=\deg(d)$.
Then we define $X(f)$ by the commutativity of the following diagram: 
\[
 \xymatrix@C=2pc{
 \sk_{\alpha}X(c)  \ar[r]^-{\cong} \ar[d]_-{(\sk_{\alpha}X)(f)} & X(c) \ar@{-->}[d]^-{X(f)}  \\
\sk_{\alpha}X(d) \ar[r] & X(d). 
 }
\]
Here, the unlabeled morphisms are part of the data given and the top one is an isomorphism since $\deg(d)<\alpha$. The resulting assignment $\C^+(c,d)\to\M(X(c),X(d))$ is clearly $k$--linear, i.e.\ $X(f+g)=X(f)+X(g)$ and $X(\lambda\cdot f)=\lambda\cdot X(f)$, for all $f,g\in\C^+(c,d)$ and $\lambda\in k$.
For the dual case where $f$ belongs to $\C^-(c,d)$ with $\deg(c)=\alpha>\deg(d)$, by a similar argument we define $X(f)$ by the commutativity of the following diagram: 
\[
 \xymatrix@C=2pc{
X(c)  \ar[r] \ar@{-->}[d]_-{X(f)}  &  \cosk_{\alpha}X(c)  \ar[d]^-{(\cosk_{\alpha}X)(f)}  \\
X(d) \ar[r]^-{\cong} & \cosk_{\alpha}X(d).
 }
\]
One readily checks that in both cases, the following diagram commutes since the outer rectangle commutes (the horizontal compositions are components of the natural transformation $\tau^{\alpha,X}$):
\begin{equation}
\label{eq:combined-sk-cosk}
\vcenter{
\xymatrix@C=2pc{
 \sk_{\alpha}X(c) \ar[r] \ar[d]_-{(\sk_{\alpha}X)(f)} & X(c) \ar[r] \ar[d]^-{X(f)} & \cosk_{\alpha}X(c)  \ar[d]^-{(\cosk_{\alpha}X)(f)} \\
\sk_{\alpha}X(d) \ar[r] & X(d) \ar[r] & \cosk_{\alpha}X(d). 
}}
\end{equation}

Given a general morphism $f\in\C_{\leqslant\alpha}(c,d)$ with $\deg(c)=\alpha$ or $\deg(d)=\alpha$, we consider a Reedy factorization $f=\sum_{i}\lambda_{i}\cdot f_{i}^{+}\circ f_{i}^{-}$ using the isomorphism $\bigoplus_{e \in \mathrm{Obj}(\C_{\leqslant\alpha})} \C^+(e,d) \otimes_{k} \C^-(c,e) \xrightarrow{\cong} \class C(c,d)$, and put
\[ X(f) = \sum_i \lambda_{i}\cdot X(f_{i}^{+})\circ X(f_{i}^{-})\colon X(c)\longrightarrow X(d). \]
The resulting assignment $\C_{\leqslant\alpha}(c,d)\to\M(X(c),X(d))$ is again $k$--linear by the universal property of $\otimes_k$ and, clearly, the diagram~\eqref{eq:combined-sk-cosk} commutes for any $f\in\C_{\leqslant\alpha}(c,d)$ with any $c,d\in\mathrm{Obj}(\C_{\leqslant\alpha})$.

It remains to show that $X$ respects compositions.
To this end, consider morphisms $f\colon c\rightarrow d$ and $g\colon d\rightarrow e$ in $\C_{\leqslant\alpha}$ and their Reedy factorizations $f=\sum_{i}\lambda_{i}\cdot f_{i}^{+}\circ f_{i}^{-}$ and $g=\sum_{j}\lambda_{i}\cdot g_{j}^{+}\circ g_{j}^{-}$ respectively. By the additivity of $X$ it suffices to prove that $X(g_j)\circ X(f_i) = X(g_j\circ f_i)$ for all $i,j$; where $f_i=f_i^+\circ f_i^-$ and $g_j=g_j^+\circ g_j^-$. We consider the following diagram:  
\[
 \xymatrix@C=2pc{
 c\ar[rr]^-{f_i} \ar[dr]_-{f_i^-} & & d \ar[rr]^-{g_j} \ar[dr]_-{g_j^-} & &  e \\
 & c' \ar[ur]_-{f_i^+} \ar[rr]_-{h_{ij}} & & d' \ar[ur]_-{g_j^+} &  
 }
\]
In case the degree of $c'$ (resp. $d'$) is equal to $\alpha$, the morphism $f_i$ (resp. $g_j$) is a multiple of the identity and what we want to prove is clear. So assume that $\deg(c')$ and $\deg(d')$ are smaller than $\alpha$.
In that, case, both $X(g_j^-)\circ X(f_i^+)$ and $X(g_j^-\circ f_i^+)$ fit as the middle vertical arrow of the diagram
\[
\xymatrix@C=2pc{
 \sk_{\alpha}X(c') \ar[r]^-{\cong} \ar[d]_-{(\sk_{\alpha}X)(h_{ij})} & X(c') \ar[r]^-{\cong} \ar@{-->}[d] & \cosk_{\alpha}X(c')  \ar[d]^-{(\cosk_{\alpha}X)(h_{ij})} \\
\sk_{\alpha}X(d') \ar[r]^-{\cong} & X(d') \ar[r]^-{\cong} & \cosk_{\alpha}X(d')
}
\]
(recall diagram \eqref{eq:combined-sk-cosk}). Hence $X(g_j^-)\circ X(f_i^+)=X(g_j^-\circ f_i^+)$. For a similar reason, if we write $h_{ij}=\sum_{w}\mu_{w}\cdot h_{ijw}^+ \circ h_{ijw}^-$ for a Reedy factorization of $h_{ij}$, we also have: 
\begin{align*}
X(g_j^+)\circ  X(h_{ij}) \circ X(f_i^-) &= \sum_{w}\mu_{w}\cdot X(g_j^+)\circ X(h_{ijw}^+) \circ X(h_{ijw}^-)\circ X(f_i^-) \\
&= \sum_{w}\mu_{w}\cdot X(g_j^+\circ h_{ijw}^+) \circ X(h_{ijw}^-\circ f_i^-) \\
&= X(g_j^+\circ h_{ij}\circ f_i^-)  \\
&= X(g_j \circ f_i).
\end{align*}

Finally, as in \cite{riehl} the last part of the statement follows if we apply the first one to the $k$--linear functor $\tilde{X}\colon \C_{<a}\rightarrow \mathrm{Arr}(\M)$ which sends an object $c$ to the morphism $\phi_c\colon X(c)\rightarrow Y(c)$; here $\mathrm{Arr}(\M)$ denotes the arrow category of $\M$ (which is $k$--linear bicomplete abelian just as $\M$). 
\end{proof}

\subsection{Cofinality}
\label{subsec:cofinality}

In Theorem \ref{thm:cofinality} below we show that for a functor $X\colon\C\rightarrow \M$ and an object $c$ of degree $\alpha$, the (weighted) colimit $\sk_{\alpha}X(c)$ from (\ref{eq:sk}) is isomorphic to a colimit of $X$ weighted by certain morphisms that depend only on the direct subcategory $\C^+$. A dual result concerning  $\cosk_{\alpha}X(c)$ is also given. The analogous statement in the non-additive case can be found in \cite[Cor.~15.2.9]{Hir}.

For an object $c$ in $\C$ and an ordinal $\alpha<\lambda$ we consider the inclusion $i\colon \C_{<\alpha}\rightarrow \C$ and the following functors:
\begin{equation}
\C^+(i(-),c)\colon(\C^{+}_{<\alpha})^\mathrm{o}\rightarrow \mathcal{K},\qquad\qquad
\C^-(c,i(-))\colon\C^-_{<\alpha} \rightarrow \mathcal{K}.
\end{equation}

Note that in case $c$ is of degree $\alpha$ in $\C$, in view of Remark \ref{rem:non_isos_rad} there exist isomorphisms:
\begin{equation}
\C^+(i(-),c)\cong \rad_{\C^+}(i(-),c),
\qquad\qquad
\C^-(c,i(-))\cong \rad_{\C^-}(c,i(-)).
\end{equation}

\begin{definition}
\label{def:latching_matching}
If $c$ is an object of degree $\alpha$ in $\class C$ and $X\colon\C\rightarrow \M$ is a $k$--linear functor, then we call the \textit{latching object of $X$ at $c$} the object of $\M$ defined as
\[L_{c}X:=\C^{+}(i(-),c)\tens{\class C_{<\alpha}^+}\res_\alpha(X),\]
and we call the \textit{latching morphism of $X$ at $c$}, the following morphism in $\M$, 
\[
l^{\alpha,X}_{c}:=l_{c}^{X}\colon L_cX\longrightarrow X(c),
\]
which is analogous to the morphism in~\eqref{eq:counit} (see also the discussion below it).

Dually, we call the \textit{matching object of $X$ at $c$} the object of $\M$ defined as
\[M_{c}X:=\hom_{\C^-_{<\alpha}}\left( \C^{-}(c,i(-)),\res_\alpha(X)\right),\]
and we call the \textit{matching morphism of $X$ at $c$}, the following morphism in $\M$ analogous to~\eqref{eq:unit},
\[m^{\alpha,X}_{c}:=m^{X}_{c}\colon X(c)\rightarrow M_cX;\,\,\, \xi \mapsto [f\mapsto X(f)(\xi)].\]
\end{definition}

Our aim is to prove the following:

\begin{theorem}\textnormal{(Cofinality)}
\label{thm:cofinality}
Let $\C$ be a linear Reedy category, $c$ be an object in $\C$ of degree $\alpha$ and $X\in\M^\C$.
There exist the following natural isomorphisms in the abelian category $\M$,
\begin{itemize}
\item[(i)] $(\sk_{\alpha}\res_\alpha X)(c)\cong L_{c}X$
\item[(ii)] $(\cosk_{\alpha}\res_\alpha X)(c)\cong M_{c}X.$
\end{itemize}
\end{theorem}

Before proving Theorem~\ref{thm:cofinality}, we will establish preparatory results.

\begin{lemma}\cite[Lemma~3.5]{riehl}
\label{lem:weighted_Kan}
Let $c$ be an object in $\C$ of degree $\alpha$. There exist the following natural isomorphisms in the abelian category $\M$, 
\begin{itemize}
\item[(i)] $(\sk_{\alpha}\res_{\alpha}X)(c)\cong \sk_{\alpha}\C(i(-),c)\otimes_{\C} X$,
\item[(ii)] $(\cosk_{\alpha}\res_{\alpha}X)(c)\cong \hom_{\C}\left( \cosk_{\alpha}\C(c,i(-)), X\right) $
\end{itemize}
(notice that here we have used contravariant versions of the functors $\sk_{\alpha}$ and $\cosk_{\alpha}$).
\end{lemma}

\begin{proof}
It is not hard to see that the argument given by Riehl and Veriti \cite[Lemma~3.5]{riehl} carries over to our context, but we include it here for convenience. We only prove (i) as the proof of (ii) is dual. We have: 
\begin{align*}
(\sk_{\alpha}\res_\alpha X)(c) & \cong \C(i(-),c)\otimes_{\C_{<\alpha}} \mathrm{res}_{\alpha}X \\
 & \cong \C(i(-),c)\otimes_{\C_{<\alpha}} \left( \C(?,i(-))\otimes_{\C}X \right)  \\
 & \cong \left( \C(i(-),c)\otimes_{\C_{<\alpha}}  \C(?,i(-)) \right) \otimes_{\C}X  \\
 & \cong \sk_{\alpha}\C(i(-),c)\otimes_{\C} X. 
\end{align*}
The first isomorphism is by definition of $\sk_{\alpha}X(c)$, while the second isomorphism follows from the fact that $X(d)\cong\C(-,d)\otimes_{\C}X$ as recalled in \S\ref{subsec:ringoids}. The third isomorphism is a $k$--linear version of the Fubini Theorem \cite[eq.~(2.9)]{Kelly} and the last is given by the definition of a contravariant version of the functor $\sk_{\alpha}$.
\end{proof}

\begin{lemma}
\label{lem:map_for_cofinality}
Let $i\colon \C_{<\alpha}\rightarrow \C$ be the full embedding as before and let $c,d\in\Obj(\C)$. Then the canonical map
\begin{equation}
\label{eq:map_for_cofinality}
\C^+(i(-),c)\tens{\C_{<\alpha}^+} \C(d,i(-))
\stackrel{u}\longrightarrow 
\C(i(-),c)\tens{\C_{<\alpha}} \C(d,i(-))
\end{equation}
is an isomorphism.
\end{lemma}

\begin{proof}
We will define an explicit inverse
\begin{equation}
\label{eq:inverse_for_cofinality}
\C(i(-),c)\tens{\C_{<\alpha}} \C(d,i(-))
\stackrel{v}\longrightarrow 
\C^+(i(-),c)\tens{\C_{<\alpha}^+} \C(d,i(-))
\end{equation}
as follows. Given an element
\[ \phi\otimes\psi\in\C(i(-),c)\tens{\C_{<\alpha}} \C(d,i(-)), \]
we consider a Reedy factorization $\phi=\sum_{i} \phi_i^+\phi_i^-$ and put $v(\phi\otimes\psi):=\sum_{i} \phi_i^+\otimes\phi_i^-\psi$. Once we prove that the map is well defined, it is completely straightforward to check that it is inverse to $u$ and the proof will be finished.

To this end, note that the assignment $\phi\otimes\psi \mapsto \sum_{i} \phi_i^+\otimes\phi_i^-\psi$ gives a well-defined map
\[
\bigoplus\limits_{\deg(e)<\alpha} \C(i(e),c)\otimes_{k}\C(d,i(e)) \stackrel{w}\longrightarrow \C^+(i(-),c)\tens{\C_{<\alpha}^+} \C(d,i(-)).
\]
Indeed, the map $w$ equals the composition
\begin{multline*}
\bigoplus\limits_{\deg(e)<\alpha} \C(i(e),c)\otimes_{k}\C(d,i(e))\longrightarrow \\
\bigoplus\limits_{\deg(e),\deg(f)<\alpha} \C^+(i(f),c)\otimes_{k}\C^-(i(e),i(f))\otimes_{k}\C(d,i(e)) \longrightarrow \\
\bigoplus\limits_{\deg(e),\deg(f)<\alpha} \C^+(i(f),c)\otimes_{k}\C(d,i(f)) \longrightarrow
\C^+(i(-),c)\tens{\C_{<\alpha}^+} \C(d,i(-)),
\end{multline*}
where the first arrow is given by the Reedy factorisation, the second one by the composition in $\C$ and the last one by the definition of the tensor product $\otimes_{\C^+_{<\alpha}}$.
In order to prove that $w$ factorizes through the surjection 
\[
\bigoplus\limits_{\deg(e)<\alpha} \C(i(e),c)\otimes_{k}\C(d,i(e)) \to
\C(i(-),c)\tens{\C_{<\alpha}} \C(d,i(-))
\]
to a map $v$ as in \eqref{eq:inverse_for_cofinality}, we need to prove that for each chain of maps of the form
\[
\xymatrix@1{ \,c\, & \,i(f)\, \ar[l]_-\phi & \,i(e)\, \ar[l]_-\rho & \,d\,\ar[l]_-\psi, }
\]
we have $w(\phi\rho\otimes\psi) = w(\phi\otimes\rho\psi)$. For that, we consider a Reedy factorisation $\phi=\sum_{i}\phi_i^+\circ\phi_i^-$ and, for each $i$, a Reedy factorisation $\phi_i^-\rho=\sum_{j}\tau_{ij}^+\circ\tau_{ij}^-$. Note that $\phi\rho=\sum_{ij}(\phi_j^+\tau_{ij}^+)\circ\tau_{ij}^-$ is then also a Reedy factorisation. Thus,
\[
w(\phi\rho\otimes\psi) = \sum_{ij} \phi_j^+\tau_{ij}^+\otimes\tau_{ij}^-\psi.
\]
On the other hand, we have
\[
w(\phi\otimes\rho\psi) =
\sum_{i}\phi_i^+\otimes\phi_i^-\rho\psi =
\sum_{ij}\phi_i^+\otimes\tau_{ij}^+\tau_{ij}^-\psi = 
\sum_{ij}\phi_i^+\tau_{ij}^+\otimes\tau_{ij}^-\psi.
\]
The first equality is just the definition of $w$, the second one uses the Reedy decomposition of $\phi_i^-\rho$ and the last one is a defining relation in the tensor product $\C^+(i(-),c)\otimes_{\C_{<\alpha}^+} \C(d,i(-))$.
\end{proof}

Now we can give a proof of Theorem~\ref{thm:cofinality}.

\begin{proof}[Proof of~Theorem~\ref{thm:cofinality}]
We prove (i) since the proof of (ii) is similar. By combining the isomorphisms from the preceding lemmas, we obtain:
\begin{align*}
(\sk_{\alpha} \res_\alpha X)(c) &\cong \sk_{\alpha}\C(i(-),c)\otimes_{\C} X \\
&\cong (\C(i(-),c)\otimes_{\C_{<\alpha}} \C(?,i(-))) \otimes_{\C} X \\
&\cong (\C^+(i(-),c)\otimes_{\C_{<\alpha}^+} \C(?,i(-)) \otimes_{\C} X \\
&\cong \C^+(i(-),c)\otimes_{\C_{<\alpha}^+} (\C(?,i(-)) \otimes_{\C} X) \\
&\cong \C^+(i(-),c)\otimes_{\C_{<\alpha}^+} \res_{\alpha}X \\
&\cong L_cX(c).
\end{align*}
The first isomorphism is by Lemma~\ref{lem:weighted_Kan}(i), the second by definition of $\sk_{\alpha}\C(i(-),c)$, the third by~\eqref{eq:map_for_cofinality}, the fourth by Fubini \cite[eq.~(2.9)]{Kelly}, the fifth by Yoneda and the last by Definition~\ref{def:latching_matching}.
\end{proof}

In the $k$--linear setting, the following result is a more practical consequence of the arguments behind the proof of Theorem~\ref{thm:cofinality}.

\begin{theorem}~\label{thm:standard-(co)induction}
Let $\C$ be a linear Reedy category, $c$ be an object in $\C$ of degree $\alpha$ and $X\in\M^\C$. Recall the ideal $\I^\C_{<\alpha}\subseteq\C$ from \S\ref{subsec:standard}. Then:
\begin{itemize}
\item[(i)] The counit of adjunction $l^X\colon \sk_\alpha\res_\alpha X\to X$ is naturally isomorphic to the obvious morphism $\I^\C_{<\alpha}\otimes_\C X\to\C\otimes_\C X\cong X$.
\item[(ii)] The unit of adjunction $m^X\colon X\to\cosk_\alpha\res_\alpha X$ is naturally isomorphic to the obvious morphism $X\cong \hom_\C(\C,X)\to\hom_\C(\I^\C_{<\alpha},X)$.
\end{itemize}
\end{theorem}

Before proving the theorem, we point up the following corollary, which is an immediate consequence of Theorems~\ref{thm:cofinality} and~\ref{thm:standard-(co)induction} (recall also the functors $\tor$ and $\ext$ from \S\ref{subsec:ringoids} and left and right standard modules from Definition~\ref{def:standards}). 

\begin{corollary}\label{cor:latching-and-matching-via-standards}
In the setting of Theorem~\ref{thm:standard-(co)induction}, the following hold:
\begin{itemize}
\item[(i)] The cokernel of the latching morphism $l^X_c\colon L_cX\to X(c)$ of $X$ at $c$ is isomorphic to $\Delta^c\otimes_\C X\in\M$ and $l^X_c$ is a monomorphism if and only if $\tor_1^\C(\Delta^c,X)=0$.
\item[(ii)] The kernel of the matching morphism $m^X_c\colon X(c)\to M_cX$ of $X$ at $c$ is isomorphic to $\hom_\C(\Delta_c,X)$ in $\M$ and $m^X_c$ is an epimorphism if and only if $\ext^1_\C(\Delta_c,X)=0$.
\end{itemize}
\end{corollary}

\begin{proof}[Proof of Theorem~\ref{thm:standard-(co)induction}]
Note that the canonical map 
\[
\bigoplus\limits_{\mathrm{deg(e)<\alpha}}  \C^+(i(e),c)\otimes_{k} \C^-(d,i(e)) \longrightarrow
\C^+(i(-),c)\tens{\C_{<\alpha}^+} \C(d,i(-))
\]
is invertible. Indeed, the inverse is given as a composition
\begin{align*}
\C^+(i(-),c)\tens{\C_{<\alpha}^+} \C(d,i(-)) &\cong
\bigoplus\limits_{\deg(e)<\alpha} \left( \C^+(i(-),c)\tens{\C_{<\alpha}^+} \C^+_{<\alpha}(i(e),i(-)) \right) \otimes_{k}\C^-(d,i(e)) \\ 
&\cong \bigoplus\limits_{\mathrm{deg(e)<\alpha}}  \C^+(i(e),c)\otimes_{k} \C^-(d,i(e)),
\end{align*}
where the second isomorphism occurs by Yoneda's lemma since $\deg(e)<\alpha$.
More explicitly, the inverse sends an element $\phi\otimes\psi\in\C^+(i(e),c)\otimes_{k}\C(d,i(e))$ to $\sum \phi\psi_i^+ \otimes \psi_i^-$, where $\psi=\sum_{i} \psi^+_i \circ \psi^-_i$ is a Reedy factorization.
Combining this isomorphism with \eqref{eq:map_for_cofinality} from Lemma~\ref{lem:map_for_cofinality}, we obtain a canonical isomorphism
\begin{equation}
\label{eq:shulman_iso}
\I^\C_{<\alpha}(d,c) = \bigoplus\limits_{\mathrm{deg(e)<\alpha}}  \C^+(i(e),c)\otimes_{k} \C^-(d,i(e)) \longrightarrow \C(i(-),c)\tens{\C_{<\alpha}} \C(d,i(-))
\end{equation}

Now the computation in the proof of Theorem~\ref{thm:cofinality} tells us that $\sk_\alpha\res_\alpha X\cong\I^\C_{<\alpha}\otimes_\C X$ and it is not difficult to follow the isomorphism there to see that the counit $l^X\colon\sk_\alpha\res_\alpha X\to X$ comes from tensoring the inclusion $\I^\C_{<\alpha}\to\C$ by $X$.
The argument for the adjunction unit $X\to\cosk_\alpha\res_\alpha X$ is analogous.
\end{proof}

\begin{remark}
Shulman defined Reedy $\mathcal{V}$--categories for general bicomplete symmetric monoidal categories $\mathcal{V}$ and in~\cite[Thm.~9.12, eq.~9.13]{Shul} obtained an analogous isomorphism to~\eqref{eq:shulman_iso} in the case where $\deg(d) \geqslant \alpha$ in his setup. When specializing to the monoidal category of vector spaces over $k$, we can recover~\eqref{eq:shulman_iso} for $\deg(d) \geqslant \alpha$ also from this result.
\end{remark}

\section{Lifting cotorsion pairs}
\label{sec:cot_pairs}

In this Section, we are still working under Setup \ref{setup} and prove lifting of cotorsion pairs and their completeness and heredity from the ``ground'' category $\M$ to the functor category $\M^{\C}$.

Recall that if $X$ is a $k$--linear functor in $\M^{\C}$ and $c$ is an object in $\class C$ of degree $\alpha$, we have defined the latching morphism $l_{c}^{X}\colon L_cX\rightarrow X(c)$ and the matching morphism $m_{c}^{X}\colon X(c)\rightarrow M_cX$ and characterized their properties by homological methods (Corollary~\ref{cor:latching-and-matching-via-standards}).
This allows us to introduce the following classes of objects in~$\M^\C$.

\begin{definition}
\label{dfn:phi_psi}
If $\class S$ is a class of objects in $\class M$, we define the following classes of $k$--linear functors:
\begin{align*}
\Phi(\class S)&:=\{X\in \M^{\C}\,\, |\,\, \forall c\in\class C, \,\, l^{X}_{c}\,\, \mbox{is a monomorphism and}\,\, \coker l_{c}^{X}\in\class S\} \\
&\phantom{:}= \{X\in \M^{\C}\,\, |\,\, \forall c\in\class C, \,\, \Delta^c\otimes_{\C}X\in\class S\,\, \mbox{and}\,\,\tor^{\C}_1(\Delta^c,X)=0\},
\\
\Psi(\class S)&:=\{X\in \M^{\C}\,\, |\,\, \forall c\in\class C, \,\, m^{X}_{c}\,\, \mbox{is an epimorphism and}\,\, \ker m_{c}^{X}\in\class S\} \\
&\phantom{:}= \{X\in \M^{\C}\,\, |\,\, \forall c\in\class C, \,\, \hom_\C(\Delta_c,X)\in\class S\,\, \mbox{and}\,\,\ext_{\C}^1(\Delta_c,X)=0\}.
\end{align*}
\end{definition}

Here is our main result, which will be proved in the rest of the section:

\begin{theorem}
\label{thm:lifting_cot}
Let $(\A,\B)$ be a cotorsion pair in $\M$ such that $\A$ is generating and $\B$ is cogenerating. Then

\begin{enumerate}
\item\label{thm-item:lifting_cot_plain} The pair $(\Phi(\A),\Psi(\B))$ is a cotorsion pair in $\M^{\C}$ with $\Phi(\A)$ generating and $\Psi(\B)$ cogenerating.
\item\label{thm-item:lifting_cot_complete} If $(\A,\B)$ is a complete cotorsion pair, so is $(\Phi(\A),\Psi(\B))$.
\item\label{thm-item:lifting_cot_hereditary} If $(\A,\B)$ is a hereditary cotorsion pair, so is $(\Phi(\A),\Psi(\B))$.
\end{enumerate}
\end{theorem}

Before giving a proof, let us put the theorem in some context. Among others, it generalizes the main results of~\cite{hj2016}. Although these are stated for shape categories $\C$ which are non-linear, it is clear that our methods encompass those of \cite{hj2016}. More precisely, we have the following:

\begin{corollary}
\label{cor:generalize_HJ}
Let $(\class A,\class B)$ be a complete cotorsion pair in $\M$. Let $Q$ be a quiver and $\C:=kQ$ its $k$--linearization. Then the following hold:
\begin{itemize}
\item[(i)] If $Q$ is left-rooted then $(\Phi(\A),\B^{\C})$ is a complete cotorsion pair in $\M^{\C}$.
\item[(ii)] If $Q$ is right-rooted then $(\A^{\C},\Psi(\B))$ is a complete cotorsion pair in $\M^{\C}$.
\end{itemize}
These cotorsion pairs will be hereditary in case $(\A,\B)$ is.
\end{corollary}

\begin{proof}
(i) If $Q$ is left-rooted then $\C$ is direct $k$--linear by Example~\ref{ex:rooted_quivers}; hence Reedy, cf.~Example~\ref{ex:extreme Reedy cases}. In view of Theorem~\ref{thm:lifting_cot} it suffices to show that $\Psi(\B)=\B^{\C}$. But this follows easily from the definition of the class $\Psi(\B)$, since all the matching objects of any functor in $\M^{\C}$ are zero as $\C$ is direct. The proof of (ii) is dual.
\end{proof}

In the next example we describe the classes $\Phi(\A)$ and $\Psi(\B)$ for a quiver which is neither left nor right rooted.

\begin{example}
\label{ex:generalize_HJ}
Let $(\class A,\class B)$ be a complete cotorsion pair in $\class M$. 
Consider the quiver 
\[
 \xymatrix@C=2pc{
 0 \ar@/^0.4pc/[r]^-{\alpha} & 1 \ar@/^0.4pc/[l]^-{\beta}
 }
\]
with relation $I=\{\beta\circ \alpha\}$ and its associated $k$--linear Reedy category $\C$ as in Example~\ref{ex:qh}. According to Theorem~\ref{thm:lifting_cot} there exists a complete cotorsion pair $(\Phi(\A),\Psi(\B))$ in the category of representations $\M^{\C}$. By Definition~\ref{def:latching_matching} we know that a representation $X$ is in $\Phi(\A)$ if and only if there exist short exact sequences
$0\rightarrow \rad_{\C^+}(-,0)\otimes_{\C^+}X\rightarrow X(0)\rightarrow C_0\rightarrow 0$ and
$0\rightarrow \rad_{\C^+}(-,1)\otimes_{\C^+}X\rightarrow X(1)\rightarrow C_1\rightarrow 0$ where the cokernels $C_0$ and $C_1$ belong in $\A$. By an easy computation we deduce that  $X$ is in $\Phi(\A)$ if and only if there exists a short exact sequence $0\rightarrow X(0)\rightarrow X(1)\rightarrow C_1\rightarrow 0$ with all terms in $\A$. Similarly, one has that $Y$ is in $\Psi(\B)$ if and only if there exists a short exact sequence $0\rightarrow K_1\rightarrow Y(1)\rightarrow Y(0)\rightarrow 0$ with all terms in $\B$.
\end{example}

\subsection{Analogues of the Eklof Lemma for the action}

The first technical tool which we will need is an analogue of Proposition~\ref{prop:eklof}, but for orthogonals with respect to the functor $\ext^1_\C\colon(\K^\C)\op\times\M^\C\to\M$, as introduced in \S\ref{subsec:ringoids}.

\begin{lemma}\label{lem:enriched_eklof}
Suppose that $\A$ is a generating class of objects of $\M$, and let $Y\in\M^\C$ and $V\in\K^{\C}$ be such that 
\begin{enumerate}
\item[(i)] $V$ admits a filtration $(V_\alpha\mid\iota_{\alpha\beta})_{\alpha<\beta\le\sigma}$ with 
\item[(ii)] $\hom_\C(\coker(\iota_{\alpha,\alpha+1}),Y)\in\rightperp{\A}$ for each $\alpha<\sigma$ and
\item[(iii)] $\ext^1_\C(\coker(\iota_{\alpha,\alpha+1}),Y)=0$ for each $\alpha<\sigma$.
\end{enumerate}
Then $\hom(V,Y)\in\rightperp{\A}$ and $\ext^1_\C(V,Y)=0$.
\end{lemma}

\begin{remark}\label{rem:enriched_eklof}
If $\M$ has enough projective objects, Lemma~\ref{lem:enriched_eklof} takes a significantly easier form, for then we choose $\A$ to be the class of projectives. Then it simply says that given any $Y\in\M^\C$, the class $\ker\ext^1_\C(-,Y)\subset\K^\C$ is closed under filtrations, in complete analogy with Proposition~\ref{prop:eklof}.
\end{remark}

\begin{proof}[Proof of Lemma~\ref{lem:enriched_eklof}]
Denote $\B=\rightperp{\A}\subseteq\M$.
When we interpret $\hom_\C(-,Y)$ as a functor $\K^\C\to\M\op$, it turns the filtration of $V$ to a $\B$--filtration $(X_\alpha\mid\pi_{\alpha\beta}\colon X_\alpha\to X_\beta)$ of $X:=\hom_\C(V,Y)$ in $\M\op$, where $X_\alpha:=\hom_\C(V_\alpha,Y)$ and $\pi_{\alpha\beta}:=\hom_\C(\iota_{\alpha\beta},Y)$ for each $\alpha<\beta\le\sigma$. In particular, $X\in\B$ by Proposition~\ref{prop:eklof} applied to this filtration.

In order to prove that $\ext^1_\C(V,Y)=0$, it suffices to show that $\hom_\C(f,Y)$ is an epimorphism in $\M$ whenever $f\colon K\to L$ is a monomorphism in $\K^\C$ with $\coker f\cong V$ (in fact, we only need to consider one such $f$ with $L$ projective in $\K^\C$).
So let $f$ be such a monomorphism and consider a direct system $(L_\alpha\mid\lambda_{\alpha\beta})_{\alpha<\beta\le\sigma}$ constructed by pulling back the filtration of $V$ along the surjection $L\to V$:
\[
\xymatrix{
L_\alpha \ar[r]^{\lambda_{\alpha\beta}} \ar[d] & L_\beta \ar[r]^{\lambda_{\beta\sigma}} \ar[d] & L \ar[d] \\
V_\alpha \ar[r]^{\iota_{\alpha\beta}} & V_\beta \ar[r]^{\iota_{\beta\sigma}} & V.
}
\]
Clearly $\coker\iota_{\alpha\beta}\cong\coker\lambda_{\alpha\beta}$ for each $\alpha<\beta\le\sigma$, and since $\K^\C$ is a Grothendieck category, we also have that all the $\lambda_{\alpha\beta}$ are monomorphisms and $L_\beta=\colim_{\alpha<\beta}L_\alpha$ for each limit ordinal $\beta\le\sigma$.
Applying $\hom_\C(-,Y)\colon\K^\C\to\M\op$ to $(L_\alpha\mid\lambda_{\alpha\beta})_{\alpha<\beta\le\sigma}$, we obtain a direct system $(\hom_\C(L_\alpha,Y)\mid\hom_\C(\lambda_{\alpha\beta},Y))_{\alpha<\beta\le\sigma}$ in $\M\op$ which satisfies all the axioms of a $\B$--filtration from Definition~\ref{def:filt} except that possibly $\hom_\C(L_0,Y)\ne 0$. In fact, since $\lambda_{0\sigma}=f$ by construction, we have $\hom_\C(\lambda_{0\sigma},Y)=\hom_\C(f,Y)$.

Now, since each $\hom_\C(\lambda_{\alpha,\alpha+1},Y)$ is a $\B$--epimorphism in $\M$, it has the right lifting property with respect to any $\A$--monomorphisms in $\M$ (and so the left lifting property with respect to any $\A$--epimorphism in $\M\op$) by Lemma~\ref{lem:Ext-to-box-orthogonal}. It follows from \cite[Lemma 10.3.1]{Hir} that also $\hom_\C(f,Y)$ has the right lifting property with respect to any $\A$--monomorphisms in $\M$. Specializing this to a square in $\M$ of the form
\[
\xymatrix{
0 \ar[r] \ar[d] & \hom_\C(L,Y) \ar[d]^-{\hom_\C(f,Y)} \\
A \ar[r] \ar@{-->}[ur]^-{h} & \hom_\C(K,Y)
}
\]
with the lower horizontal row an epimorphism (which is possible since we assume $\A$ to be generating), we see that $\hom_\C(f,Y)$ must be an epimorphism, as required.
\end{proof}

A very similar argument shows also the following:

\begin{lemma}\label{lem:enriched_dual_eklof}
Suppose that $\B$ is a cogenerating class of objects of $\M$, and let $Y\in\M^\C$ and $V\in\K^{\C\op}$ be such that 
\begin{enumerate}
\item[(i)] $V$ admits a filtration $(V_\alpha\mid\iota_{\alpha\beta})_{\alpha<\beta\le\sigma}$ with 
\item[(ii)] $\coker(\iota_{\alpha,\alpha+1})\otimes_{\C}Y\in\leftperp{\B}$ for each $\alpha<\sigma$ and
\item[(iii)] $\tor_1^\C(\coker(\iota_{\alpha,\alpha+1}),Y)=0$ for each $\alpha<\sigma$.
\end{enumerate}
Then $V\otimes_{\C}Y\in\leftperp{\B}$ and $\tor_1^\C(V,Y)=0$.
\end{lemma}

\begin{proof}
Use the same argument as for Lemma~\ref{lem:enriched_eklof}. Just replace the occurrences of the functor $\hom_\C(-,Y)\colon\K^\C\to\M\op$ by $-\otimes_{\C}Y\colon\K^{\C\op}\to\M$ and note that $\tor_1^\C(V,Y)=0$ if and only if $f\otimes_{\C}Y$ is a monomorphism in $\M$ whenever $f\colon K\to L$ is a monomorphism in $\K^{\C\op}$ with $\coker f\cong V$ (in fact, we again only need to consider one such $f$ with $L$ projective in $\K^{\C\op}$).
\end{proof}

\subsection{Cotorsion pairs and diagrams}

Now we will continue our investigation of how cotorsion pairs in $\M$ relate to those in $\M^\C$.
Recall that by Theorem~\ref{thm:proj_by_standards} and Remark~\ref{rem:proj_by_standards}, we have a filtration of $\C(c,-)$ by standard functors of degree $\leqslant\alpha$, where $\alpha:=\deg(c)$, such that the last filtration factor is $\C(c,-)/\I^\C_{<\alpha}(c,-)\cong\Delta_c$, (see~\eqref{eq:filtration_by_standards}). Moreover, there is also an analogous filtration of $\C(-,c)$ by contravariant standard functors thanks to Remark~\ref{rem:opp_Reedy_cat}. As an immediate consequence we have the following lemma.

\begin{lemma}
\label{lem:Reedy_objects}
Let $(\A,\B)$ be a cotorsion pair in $\M$. Then the following hold:
\begin{itemize}
\item[(i)] If $X\in\Phi(\class A)$ then for all $c$ in $\C$, the objects $L_{c}X$ and $X(c)$ belong in $\class A$. 
\item[(ii)] If $X\in\Psi(\class B)$ then for all $c$ in $\C$, the objects $M_{c}X$ and $X(c)$ belong in $\class B$.
\end{itemize}
\end{lemma}

\begin{proof}
We only prove (i) as the proof of (ii) is similar.
Let $X$ be a $k$--linear functor in $\Phi(\A)$ and let $c$ be an object of $\class C$. Then both $\C(-,c)$ and $\I^\C_{<\alpha}(-,c)$ admit a filtration by contravariant standard functors of degree $\leqslant\alpha$ by the comment before the lemma. Hence $\I^\C_{<\alpha}(-,c)\otimes_{\C}X$ and $\C(-,c)\otimes_{\C}X$ are $\A$--filtered by the very definition of the class $\Phi(\A)$ and belong to $\A$ by Proposition~\ref{prop:eklof}. Now it remains to note that $\I^\C_{<\alpha}(-,c)\otimes_{\C}X\cong L_cX$ by Theorem~\ref{thm:standard-(co)induction} and also $\C(-,c)\otimes_{\C}X\cong X(c)$.
\end{proof}

The next proposition is a key step towards proving Theorem~\ref{thm:lifting_cot}(\ref{thm-item:lifting_cot_plain}).

\begin{proposition}
\label{prop:ext_vanishing}
Let $(\class A,\class B)$ be a cotorsion pair in $\class M$. Then for any $X\in\Phi(\A)$ and $Y\in\Psi(\B)$ we have $\Ext^{1}_{\M^{\C}}(X,Y)=0.$
\end{proposition}

\begin{proof}
Let $X\in\Phi(\A)$, $Y\in\Psi(\B)$ and consider a short exact sequence in $\M^{\C}$ 
\[
 \xymatrix@C=2pc{
0 \ar[r] & Y\ar[r]^-{\iota} & W\ar[r]^-{\epsilon} & X\ar[r] & 0.
 }
\] 
We will prove that for any (non-zero) $\alpha\leqslant\lambda$ there is a splitting of the short exact sequence in $\M^{\C^{<\alpha}}$,
\begin{equation}
\Sigma_\alpha\colon
\label{eq:Sigma}
 \xymatrix@C=2pc{
0 \ar[r] & Y_{<\alpha}\ar[r]^-{\iota_{<\alpha}} & W_{<\alpha}\ar[r]^-{\epsilon_{<\alpha}} & X_{<\alpha}\ar[r] & 0.
 }
 \end{equation}
Here by $(-)_{<\alpha}$ we denote restriction to the full Reedy subcategory $\C_{<\alpha}$, as in Definition \ref{def:filtrations}. 
For $\alpha=1$, since any morphism in $\C$ between two objects of degree zero is either zero or a non-zero multiple of the identity, we deduce that a splitting of \eqref{eq:Sigma} exists if and only if for any object $c$ of degree zero, the associated short exact sequence $0\rightarrow Y(c)\rightarrow W(c)\rightarrow X(c)\rightarrow 0$ splits in $\M$, which is the case since $X(c)\in\A$ and $Y(c)\in\B$ by Lemma~\ref{lem:Reedy_objects}.

For the successor ordinal case, we assume that we have constructed a splitting of $\Sigma_\alpha$ for some $\alpha$, and we will construct one for $\Sigma_{\alpha+1}$. From Proposition~\ref{prop:Riehl_Veriti} we know that the desired splitting of $\Sigma_{\alpha+1}$ is equivalent to the existence, for all objects $c$ of degree $\alpha$, of a natural morphism $\delta$ such that the following diagram
\begin{equation}
\label{eq:3by3}
\vcenter{
 \xymatrix@C=3pc{
\sk_{\alpha}Y_{<\alpha}(c)  \ar[d]_-{\iota_{<\alpha}^{c}} \ar[r] ^-{l_{c}^{Y}}   &   Y(c) \ar[d]^-{\iota_{c}} \ar[r]^-{m_{c}^{Y}} & \cosk_{\alpha}Y_{<\alpha}(c) \ar[d]^-{\iota_{c}^{<\alpha}} \\
\sk_{\alpha}W_{<\alpha}(c)\ar[r]^-{l_{c}^{W}}   \ar[d]_-{s_{<\alpha}^{c}}   &   W(c) \ar@{-->}[d]^-{\delta_{c}} \ar[r]^-{m_{c}^{W}} &  \cosk_{\alpha}W_{<\alpha}(c) \ar[d]^-{s_{c}^{<\alpha}} \\
\sk_{\alpha}Y_{<\alpha}(c)\ar[r]^-{l_{c}^{Y}}    &   Y(c) \ar[r]^-{m_{c}^{Y}} &  \cosk_{\alpha}Y_{<\alpha}(c)
  }
}
\end{equation}
is commutative and all vertical composites are identity morphisms. The outer vertical morphisms of this diagram, by abusing the notation, are the induction and coinduction of  $\iota_{<\alpha}$ and of its section $s_{<\alpha}$, when evaluated at $c$, which exist by the induction hypothesis. 
We consider the pushout given by the solid arrows,
\[
\vcenter{
 \xymatrix@C=3pc{
\sk_{<\alpha}Y(c)  \ar[d]_-{\iota_{<\alpha}^{c}} \ar[r] ^-{l_c^{Y}}   &   Y(c) \ar[d]_-{g_c} \\
\sk_{<\alpha}W(c)   \ar[r]^-{h_c}  & Q, \ar@{-->}@/_3ex/[u]_-{g'_c}
  }
}
\]
where $g_c$ is a split monomorphism (since $\iota_{<\alpha}^{c}$ is) with cokernel isomorphic to $\sk_{\alpha}X(c)$. 
By the pushout property and the fact that $l_c^Y \circ s^{c}_{<\alpha}\circ\iota^{c}_{<\alpha}=l_c^Y$, there is a unique left inverse $g'_c$ of $g_c$ such that $l_c^Y \circ s^{c}_{<\alpha}=g_c' \circ h_c$.
By the commutativity of the upper left square in~\eqref{eq:3by3} and the pushout property, there exists also a unique map $\gamma_c\colon Q\rightarrow W(c)$ such that $\gamma_c\circ g_c=\iota_c$ and $\gamma_c\circ h_c=l_c^W$. 
In particular, we obtain a commutative diagram with exact rows 
\[
\vcenter{
 \xymatrix@C=2pc{
 0\ar[r] &  \sk_{<\alpha}Y(c) \ar[r]^-{\iota^{c}_{<\alpha}} \ar[d]_-{l_c^Y} & \sk_{<\alpha}W(c) \ar[d]_-{h_c} \ar[r]  & \sk_{<\alpha}X(c) \ar@{=}[d] \ar[r] & 0 \\
 0\ar[r] &  Y(c) \ar[r]^-{g_c} \ar@{=}[d]& Q \ar[d]_-{\gamma_c} \ar[r]^-{}   & \sk_{<\alpha}X(c) \ar[d]^-{l_c^X} \ar[r] & 0 \\
0\ar[r] & Y(c)   \ar[r]^-{\iota_{c}} & W(c) \ar[r]&  X(c) \ar[r] & 0,
  }
}
\]
where $l_c^X$ is a monomorphism with $\coker(l_c^X)\in\A$ by the assumption that $X\in\Phi(\A)$. Hence $\gamma_c$ is a monomorphism with cokernel in $\A$.
Moreover, the following solid square commutes by the pushout property of $Q$, since the right action of $g_c\colon Y(c)\rightarrow Q$ on the difference $m_c^Y \circ g'_c - s_c^{<\alpha}\circ m_c^W \circ \gamma_c$ is zero, and the same holds for the morphism $h_c\colon \sk_{\alpha}W_{<\alpha}(c)\rightarrow Q$.
\begin{equation}
\label{eq:diagonal_delta}
\vcenter{
 \xymatrix@C=3pc{
Q \ar[d]_-{\gamma_c} \ar[r]^-{g'_{c}}   &   Y(c) \ar[d]^-{m_c^Y} \\
W(c)   \ar[r]_-{s_c^{<\alpha}\circ m_c^W} \ar@{-->}[ur]^-{\delta_c} &  \cosk_{<\alpha}Y(c). 
  }
}
\end{equation}
Since in addition, by assumption we have that $m_c^{Y}$ is an epimorphism with kernel in $\B$, we obtain by Lemma~\ref{lem:Ext-to-box-orthogonal} the existence of a diagonal morphism $\delta_c$, depicted by the dashed arrow in~\eqref{eq:diagonal_delta}, such that the two triangles there commute.

The morphism $\delta_c$ is a left inverse of $\iota_c$. Indeed, we have: 
\begin{align*}
\delta_c\circ\iota_c \; & = \; \delta_c \circ\gamma_c\circ g_c  \\
& = \; g_c' \circ g_c  \\
& = \; \mathrm{id}_{Y(c)}.
\end{align*}
In addition, the lower triangle of (\ref{eq:diagonal_delta}) implies commutativity of the bottom right square in (\ref{eq:3by3}), while commutativity of the bottom left square in (\ref{eq:3by3}) holds from the following computation:
\begin{align*}
\delta_c \circ l_c^W \; & = \; \delta_c \circ \gamma_c \circ h_c \\
& = \; g'_c \circ h_c \\
& = \; l_c^Y \circ s_{<\alpha}^{c}.
\end{align*}
This finishes the passage from $\alpha$ to $\alpha+1$.

In the limit ordinal case, the various splittings of the sequences $\Sigma_{\alpha'}$ for $\alpha'<\alpha$ glue together to produce a splitting of $\Sigma_{\alpha}$.
\end{proof}

In order to finish the proof of Theorem~\ref{thm:lifting_cot}(\ref{thm-item:lifting_cot_plain}), we consider two adjunctions as follows, where the left adjoints are depicted on top:
\begin{equation}
\label{eq:adj_with_Delta}
 \xymatrix@C=3pc{
\M \ar@/^0.5pc/[rr]^-{\Delta_c\otimes-} & &  \mathcal{M}^{\mathcal{C}} \ar@/^0.5pc/[ll]^-{\hom_\C(\Delta_c,-)}
}
\,\,\,\,\,\,\,\,\,\,\,\,\,\,
 \xymatrix@C=3pc{
\mathcal{M}^{\mathcal{C}} \ar@/^0.5pc/[rr]^-{\Delta^c\tens{\C}-} & &  \M \ar@/^0.5pc/[ll]^-{\hom(\Delta^c,-)}.
}
\end{equation}
Here, given $M\in\M$, $\Delta_c\otimes M\colon\C\to\M$ is the functor sending $d\in\C$ to $\Delta_c(d)\otimes M$. Similarly, $\hom(\Delta^c, M)\colon\C\to\M$ sends $d\in\C$ to $\hom(\Delta^c(d),M)$.
For a proof of these adjunctions we refer for instance to \cite[Prop.~3.8]{HJ_qshaped} -- their argument works in the present context.
It turns out that all the four functors in~\eqref{eq:adj_with_Delta} interact nicely with cotorsion pairs under mild assumptions.

\begin{proposition}
\label{prop:hom_tens_Delta_over_C}
Let $(\A,\B)$ be a cotorsion pair in $\M$. Then the following hold:
\begin{itemize}
\item[(i)]
If $\A$ is generating, the functors $\hom_\C(\Delta_c,-)\colon\M^\C\to\M$ and $M_c\cong\hom_\C(\I^\C_{<\alpha}(c,-),-)\colon\M^\C\to\M$ send $\Psi(\B)$--epimorphisms to $\B$--epimorphisms.
\item[(ii)] If $\B$ is cogenerating, the functors $\Delta^c\otimes_{\C}-\colon\M^\C\to\M$ and $L_c\cong$ \mbox{$\I^\C_{<\alpha}(-,c)\otimes_{\C}-\colon\M^\C\to\M$} send $\Phi(\A)$--monomorphisms to $\A$--mono\-morph\-isms.
\end{itemize}
\end{proposition}

\begin{proof}
We will only prove (ii) as the proof of (i) is given by dual arguments. 
Let $f\colon X\hookrightarrow Y$ be a $\Phi(\A)$--monomorphism and $Z:=\coker f\in\Phi(\A)$. In particular, the component maps $f_c\colon X(c)\to Y(c)$ are $\A$--monomorphisms by Lemma~\ref{lem:Reedy_objects}(i).
Since we assume that $\B$ is cogenerating, a coproduct of $\A$--monomorphisms in $\M$ is an $\A$--monomorphism again by Remark~\ref{rem:conditional-exactness-of-(co)products}, so the exact sequence $0\to X\to Y\to Z\to 0$ is $\coprod$--exact in $\M^\C$ in the sense of Definition~\ref{def:tens-and-hom-exactness}. Hence we have an exact sequence
\[
0 = \tor_1^\C(\Delta^c,Z) \to \Delta^c\tens{\C}X \to \Delta^c\tens{\C}Y \to \Delta^c\tens{\C}Z \to 0
\]
by Lemma~\ref{lem:long-exact-seq} and so $\Delta^c\otimes_{\C}f$ is also an $\A$--monomorphism. Similarly, we have an exact sequence
\[
\tor_1^\C(\I^\C_{<\alpha}(-,c),Z) \to \I^\C_{<\alpha}(-,c)\tens{\C}X \to \I^\C_{<\alpha}(-,c)\tens{\C}Y \to \I^\C_{<\alpha}(-,c)\tens{\C}Z \to 0
\]
Recall again the filtration of $\I^{\C}_{<\alpha}(-,c)$ by contravariant standard modules of degrees $<\alpha$ from~\eqref{eq:filtration_by_standards} in the proof of Theorem~\ref{thm:proj_by_standards} for $\C\op$. Then $\I^\C_{<\alpha}(-,c)\otimes_{\C}Z\in\A$ and $\tor_1^\C(\I^\C_{<\alpha}(-,c),Z)=0$ by Lemma~\ref{lem:enriched_dual_eklof}.
\end{proof}

For the remaining pair of functors in~\eqref{eq:adj_with_Delta}, we have the following lemma and proposition

\begin{lemma} \label{lem:induction-to-Reedy}
Let $(\A,\B)$ be a cotorsion pair in $\M$ and $c\in\C$.
Then $\Delta_c\otimes A\in\Phi(\A)$ for each $A\in\A$ and $\hom(\Delta^c,B)\in\Psi(\B)$ for each $B\in\B$.
\end{lemma}

\begin{proof}
We have short exact sequences $0\to\I^\C_{<\deg(d)}(d,-)\overset{l_d}\to\C(d,-)\to\Delta_d\to 0$ and $0\to\I^\C_{<\deg(d)}(-,d)\overset{l^d}\to\C(-,d)\to\Delta^d\to 0$ for each $d\in\D$ and, by Theorem~\ref{thm:Reedy_exceptional}(iv), $\Delta^c\otimes_{\C}l_d$ and $l^d\otimes_{\C}\Delta_c$ are isomorphisms whenever $d\ne c$. In order to understand the case where $d=c$, note that $\Delta_c(c)\cong k\cong\Delta^c(c)$ (see the proof of Theorem~\ref{thm:Reedy_exceptional}). Then both $\Delta^c\otimes_{\C}l_c$ and $l^c\otimes_{\C}\Delta_c$ are isomorphic to the map $0\to k$.

In order to prove that $\Delta_c\otimes A\in\Phi(\A)$, we need to show that $l^d\otimes_{\C}(\Delta_c\otimes A)$ is an $\A$--monomorphism for each $d\in\C$. By the Fubini Theorem, this is the same as proving that $(l^d\otimes_{\C}\Delta_c)\otimes A$ is an $\A$--monomorphism, which is clear from our computation of $l^d\otimes_{\C}\Delta_c$. 

Similarly, in order to prove that $\hom(\Delta^c,B)\in\Psi(\B)$, we need to show that $\hom_\C(l_d,\hom(\Delta^c,B))\cong\hom(\Delta^c\otimes_{\C}l_d,B)$ is a $\B$--epimorphism. This is again clear from our computation of $\Delta^c\otimes_{\C}l_d$.
\end{proof}

\begin{proposition}
\label{prop:hom_tens_Delta}
Let $(\A,\B)$ be a cotorsion pair in $\M$.
\begin{itemize}
\item[(i)] If $\B$ is cogenerating, then $\Delta_c\otimes-\colon\M\to\M^\C$ sends $\A$--monomorphisms to $\Phi(\A)$--monomorphisms.
\item[(ii)] If $\A$ is generating, then $\hom(\Delta^c,-)\colon\M\to\M^\C$ sends $\B$--epimorphisms to $\Psi(\B)$--epimorphisms.
\end{itemize}
\end{proposition}

\begin{proof}
We will only prove (i); the proof of (ii) is similar.
Let $f\colon K\hookrightarrow L$ be an $\A$--monomorphism in $\M$ and put $A:=\coker f\in\A$. If we apply $\Delta_c\otimes-\colon\M\to\M^\C$ to the short exact sequence $0\to K\to L\to A\to 0$ and evaluate at $d\in\C$, we obtain $0\to \C(c,d)\otimes K\to \C(c,d)\otimes L\to \C(c,d)\otimes A\to 0$, which is a coproduct of copies of the original sequence and is exact by Remark~\ref{rem:conditional-exactness-of-(co)products}. Thus, $\Delta_c\otimes_{\C}f$ is a monomorphism with cokernel $\Delta_c\otimes A$. The latter belongs to $\Phi(\A)$ by Lemma~\ref{lem:induction-to-Reedy}(i).
\end{proof}

Finally, we can give a proof of the first part of Theorem~\ref{thm:lifting_cot}.

\begin{proof}[Proof of Theorem~\ref{thm:lifting_cot}(\ref{thm-item:lifting_cot_plain}) and~(\ref{thm-item:lifting_cot_hereditary})]
Suppose that $(\A,\B)$ is a cotorsion pair in $\M$ with $\A$ generating and $\B$ cogenerating. We already know that $\Ext^1_{\M^\C}(\Phi(\A),\Psi(\B))=0$ from Proposition~\ref{prop:ext_vanishing}. On the other hand, given $X\in\leftperp{\Psi(\B)}$, we want to prove that $X\in\Phi(\A)$, i.e.\ that $\Delta^c\otimes_{\C}X\in\A$ and $\tor_1^\C(\Delta^c,X)=0$ for each $c\in\C$.
For the former, we will use the fact that $\A=\leftperp{\B}$ and Lemma~\ref{lem:Ext-orthogonal}. Given any $g\in\operatorname{Epi}(\B)$, we need to prove that $\M(\Delta^c\otimes_{\C}X,g)$ is surjective. By the second adjunction in~\eqref{eq:adj_with_Delta}, this is the same as proving that $\M^\C(X,\hom(\Delta_c,g))$ is surjective. However, this is the case by our assumption on $X$ since $\A$ is generating and so $\hom(\Delta_c,g)\in\operatorname{Epi}(\Psi(\B))$ by Proposition~\ref{prop:hom_tens_Delta}.
Showing that $\tor_1^\C(\Delta^c,X)=0$ is the same as showing that the latching map $l_c^X\cong l^c\otimes_{\C}X\colon\I^\C_{<\alpha}(-,c)\otimes_{\C}X\to\C(-,c)\otimes_{\C}X$ is a monomorphism, where $l^c\colon\I^\C_{<\alpha}(-,c)\hookrightarrow\C(-,c)$ is the inclusion and $\alpha=\deg(c)$. Since $\B$ is cogenerating, it suffices to show that $\M(l^c\otimes_{\C}X,B)$ is surjective for each $B\in\B$. Using the argument of \cite[Prop.~3.8]{HJ_qshaped}, we obtain natural isomorphisms
\begin{equation} \label{eq:adj-2-var-M-and-MC}
\begin{split}
\M(Y\tens{\C}X,Z) &\cong \M^\C(X,\hom(Y,Z)) \qquad \textrm{and} \\
\M^\C(Y'\otimes Z,X) &\cong \M(Z,\hom_\C(Y',X))
\end{split}
\end{equation}
in all the variables $Y\in\K^{\C\op}$, $Y'\in\K^\C$, $X\in\M^\C$ and $Z\in\M$ (the adjunctions~\eqref{eq:adj_with_Delta} are special cases).
So we are left with proving that $\M^\C(X,\hom(l^c,B))$ is surjective for each $B\in\B$. To see this, note $\hom(l^c,B)$ is an epimorphism (as $l^c_d\colon\I^\C_{<\alpha}(d,c)\hookrightarrow\C(d,c)$ is a split monomorphism since $\K$ is semisimple) and $\ker\hom(l^c,B)\cong\hom(\coker l^c,B)\cong\hom(\Delta^c,B)$ belongs to $\Psi(\B)$ by Lemma~\ref{lem:induction-to-Reedy}. Hence $\M^\C(X,\hom(l^c,B))$ is surjective since $X\in\leftperp{\Psi(\B)}$. This finishes the proof that $\Phi(\A)=\leftperp{\Psi(\B)}$ and the argument for the fact that $\rightperp{\Phi(\A)}=\Psi(\B)$ is similar.

Next, note that since representable functors $\C(c,-)$ are filtered by standard functors by Theorem~\ref{thm:proj_by_standards}, it follows by Lemma~\ref{lem:induction-to-Reedy} that $\C(c,-)\otimes A$ is $\Phi(\A)$-filtered for each $A\in\A$. Hence $\C(c,-)\otimes A\in\Phi(\A)$ for each $A\in\A$ for each $c\in\C$ by Proposition~\ref{prop:eklof}, and dually $\hom(\C(-,c),B)\in\Psi(\B)$ for each $B\in\B$ and $c\in\C$. Now, let $X\in\M^\C$. Then, for each $Z\in\M$ and $c\in\C$, the second isomorphism in~\eqref{eq:adj-2-var-M-and-MC} specializes to $\M^\C(\C(c,-)\otimes Z,X)\cong\M(Z,X(c))$ and one can check that $f\colon\C(c,-)\otimes Z\to X$ corresponds to the restriction of the component $f_c\colon\C(c,c)\otimes Z\to X(c)$ to $Z\cong(k\cdot\id_c)\otimes Z$ under this isomorphism. In particular, given an epimorphism $g\colon Z\to X(c)$ in $\M$, the isomorphism produces a morphism $f\colon\C(c,-)\otimes Z\to X$ in $\M^\C$ whose component $f_c$ is an epimorphism. If $\A$ is generating, we can for each $c\in\C$ find a morphism $\C(c,-)\otimes A_c\to X$ such that $A_c\in\A$ and the component at $c$ is an epimorphism. Thus, the induced map $\bigoplus_{c\in\C}\C(c,-)\otimes A_c\to X$ is an epimorphism in $\M^\C$ and $\bigoplus_{c\in\C}\C(c,-)\otimes A_c\in\Phi(\A)$ (cf.\ Remark~\ref{rem:cotorsion-pairs-and-(co)products}), so $\Phi(\A)$ is generating. By a dual argument, $\Psi(\B)$ is cogenerating.

Finally, suppose that $(\A,\B)$ is hereditary. To prove that $(\Phi(\A),\Psi(\B))$ is hereditary, it suffices to prove by Lemma~\ref{lem:hereditary} that $\Phi(\A)$ is closed under kernels of epimorphisms. To that end, 
consider a short exact sequence $0\to K\to Y\to X\to 0$ with $Y$ and $X$ in $\Phi(\A)$. In particular the inclusion $K\to Y$ is a $\Phi(\A)$--monomorphism, so we have the following commutative diagram with exact rows 
\[
 \xymatrix@C=2pc{
0 \ar[r] & L_cK \ar^{l_c^K}[d] \ar[r] &  L_cY \ar[d]^-{l_c^Y} \ar[r]  &  L_cX  \ar[r] \ar[d]^-{l_c^{X}} & 0 \\
0 \ar[r] & K(c) \ar[r] &Y(c) \ar[r] & X(c)  \ar[r] & 0, 
 }
\]
where the top row is exact by Proposition~\ref{prop:hom_tens_Delta_over_C}(ii) (together with Theorem~\ref{thm:standard-(co)induction}). Thus $l_c^K$ is a monomorphism and, by the Snake Lemma, we have a short exact sequence $0\rightarrow \coker(l_c^K)\rightarrow \coker(l_c^Y) \rightarrow \coker(l_c^X)\rightarrow 0$. Hence, from Lemma~\ref{lem:hereditary} applied to $(\A,\B)$, we deduce that $\coker(l_c^K)\in\A$ and that $K$ is in $\Phi(\A)$. 
One could also argue dually and use by Proposition~\ref{prop:hom_tens_Delta_over_C}(i) to prove that $\Psi(\B)$ is closed under cokernels of monomorphisms.
\end{proof}

\subsection{Lifting of approximation sequences}
\label{subsec:lifting}
In this subsection we are going to complete the proof of Theorem~\ref{thm:lifting_cot}(\ref{thm-item:lifting_cot_complete}).

\begin{proof}[Proof of Theorem~\ref{thm:lifting_cot}(\ref{thm-item:lifting_cot_complete})]
Consider a $k$--linear functor $X\colon \C\to \M$ and let $(\A,\B)$ be a complete cotorsion pair in $\M$. We must construct short exact sequences
\begin{equation}
\label{eq:left_approximation}
0\rightarrow Z\rightarrow Y\rightarrow X\rightarrow 0,
\end{equation}
where $Y$ is in $\Phi(\class A)$ and $Z$ is in $\Psi(\B)$, and dually
\begin{equation}
\label{eq:right_approximation}
0\rightarrow X\rightarrow Y'\rightarrow Z'\rightarrow 0,
\end{equation}
where $Y'$ is in $\Psi(\class B)$ and $Z'$ is in $\Phi(\A)$.
We will prove the existence of the short exact sequence~\eqref{eq:left_approximation}
and construct the desired approximation of $X$ by transfinite induction on $\operatorname{len}(\C):=\sup\{\deg(c)\mid c\in\C\}$, using Proposition~\ref{prop:Riehl_Veriti}.
Sequences of the type~\eqref{eq:right_approximation} are obtained by dual arguments.

If $\operatorname{len}(\C)=0$, then there are only objects of degree zero and any morphism between objects is either zero or an invertible endomorphism. In particular, $\M^\C\simeq \prod_{c\in\C}\M$.
Thus, since $(\A,\B)$ is a complete cotorsion pair in $\M$, there exists a short exact sequence $0\to Z_{c}\to Y_{c}\to X(c)\to 0$ with $Y_{c}$ in $\A$ and $Z_{c}$ in $\B$ and these sequences constiture a short eact sequence $0\to Z\to Y\to X\to 0$ in $\M^\C$ such that for any object $c$ in $\C$, $Y(c)=Y_{c}$ and $Z(c)=Z_{c}$.
Clearly, the latching map of $Y$ at $c$ is $0\rightarrow Y(c)$ and and the matching map of $Z$ at $c$ is the map $Z(c)\rightarrow 0$ for any objects $c\in\C$, so $Y\in\Phi(\A)$ and $Z\in\Psi(\B)$.

Assume now that $\alpha:=\operatorname{len}(\C)$ is a successor ordinal, i.e. $\alpha=\beta+1$, and that a desired approximation of $X_{<\alpha}\in\M^{\C_{<\alpha}}$ has been constructed,
\begin{equation}
\label{eq:given_ses}
0\to Z_{<\alpha}\to Y_{<\alpha}\to X_{<\alpha}\to 0
\end{equation}
with $Y_{<\alpha}$ is in $\Phi_{\C_{<\alpha}}(\class A)$ and $Z_{<\alpha}$ is in $\Psi_{\C_{<\alpha}}(\B)$. We will construct one for $X$ (note that we are abusing notation in the sense that $Y_{<\alpha}$ and $Z_{<\alpha}$ are not restrictions of any functors $\C\to\M$ at the moment, these are to be constructed yet).

To that end, we first claim that
\[
0\to \cosk_\alpha(Z_{<\alpha})(c)\to \cosk_\alpha(Y_{<\alpha})(c)\to \cosk_\alpha(X_{<\alpha})(c)\to 0
\]
is exact in $\M$ for each $c\in\C$ of degree $\beta$ (where $\cosk_\alpha(X_{<\alpha})(c)\cong M_cX$ by Theorem~\ref{thm:cofinality}, cf.~Remark~\ref{rem:restricted_adj}). Indeed, $\cosk_{\alpha}(Z)(c)=\hom_{\C_{<\alpha}}(\C(c,-)|_{\C_{<\alpha}},Z)$ (see~\eqref{eq:cosk}) and since $\C(c,-)|_{\C_{<\alpha}}$ is filtered by standard $\C_{<\alpha}$--modules by Theorem~\ref{thm:proj_by_standards}, we have $\cosk_{\alpha}(Z)(c)\in\B$ and $\ext^1_{\C_{<\alpha}}(\C(c,-)|_{\C_{<\alpha}},Z)=0$ by Lemma~\ref{lem:enriched_eklof}. Since $Y_{<\alpha}\to X_{<\alpha}$ is a $\B$--epimorphism, an analogous argument to the proof of Proposition~\ref{prop:hom_tens_Delta_over_C} tells us that~\eqref{eq:given_ses} is $\prod$--exact in the sense of Definition~\ref{def:tens-and-hom-exactness} and there is by Lemma~\ref{lem:long-exact-seq} an exact sequence
\begin{multline*}
0\to \cosk_\alpha(Z_{<\alpha})(c)\to \cosk_\alpha(Y_{<\alpha})(c)\to \cosk_\alpha(X_{<\alpha})(c)\to \\ \to \ext^1_{\C_{<\alpha}}(\C(c,-)|_{\C_{<\alpha}},Z)=0.
\end{multline*}
This proves the claim.

Similarly, we can apply $\sk_{\alpha}(-)(c)$ to~\eqref{eq:given_ses}, which is a right exact functor by~\eqref{eq:sk} and by using the ideas from the proof of Proposition~\ref{prop:hom_tens_Delta_over_C} again, we observe that $\sk_{\alpha}(Y)(c)\in\A$. As there is always a canonical natural transformation $\tau^\alpha\colon\sk_\alpha\to\cosk_\alpha$, we obtain in the end a commutative diagram with exact rows of the from:
\[
 \xymatrix@C=2pc{
& \sk_\alpha(Z_{<\alpha})(c) \ar^{\tau^{\alpha,Z}_{c}}[d] \ar[r]^-{\iota} & \sk_\alpha(Y_{<\alpha})(c) \ar[r]^-{\pi} \ar^{\tau^{\alpha,Y}_{c}}[d] & L_cX  \ar[r] \ar^{\tau^{\alpha,X}_{c}}[d] & 0 \\
0 \ar[r] & \cosk_\alpha(Z_{<\alpha})(c) \ar[r]^-{\iota'}  & \cosk_\alpha(Y_{<\alpha})(c) \ar[r]^-{\pi'} & M_{c}X  \ar[r] & 0. \\
 }
\]

Now, we consider the pullback of the diagram $\cosk_\alpha(Y_{<\alpha})(c)\to M_{c}X \xleftarrow{m^{X}_{c}} X(c)$, and we obtain a commutative diagram:
\[
\xymatrix@C=2pc{
\sk_\alpha(Y_{<\alpha})(c) \ar@/_/[ddr]_{\tau^{\alpha,Y}_{c}} \ar@/^/[drr]^{l^{X}_{c}\circ\pi} \ar@{-->}^-{{\exists ! u} }[dr]\\
& P \ar_-{\epsilon_{Y}}[d] \ar^-{\epsilon_{X}}[r] & X(c) \ar[d]^-{m^X_c} \\
& \cosk_\alpha(Y_{<\alpha})(c) \ar[r] & M_{c}X.
}
\]

The completeness of the cotorsion pair $(\A,\B)$ implies from Proposition~\ref{prop:mappings}, that there is a factorization of the morphism $u$,
\[
\label{eq:factor}
\xymatrix@C=2pc{
\sk_\alpha(Y_{<\alpha})(c) \ar_-{i}[dr] \ar^-{u}[rr]& & P, \\
& T_c \ar_-{p}[ur] & 
}
\]
where $i$ is a monomorphism with cokenel in $\class A$ and $p$ is an epimorphism with kernel in $\B$. Since we know that $\sk_\alpha(Y_{<\alpha})(c)\in\A$, we also have $T_{c}\in\A$. In total, we have a commutative diagram with exact rows,

\[
 \xymatrix@C=2pc{
& \sk_\alpha(Z_{<\alpha})(c) \ar@{-->}[d]^-{j} \ar[r]^-{\iota} & \sk_\alpha(Y_{<\alpha})(c) \ar[r]^-{\pi} \ar^{i}[d] & L_cX  \ar[r] \ar^{l_{c}^{X}}[d]  & 0 \\
0 \ar[r] & \ker(\epsilon_{X}\circ p) \ar[r] \ar@{-->}^-{\psi^{Z}_c}[d] & T_{c}  \ar^-{p}[d]  \ar^-{\epsilon_{X}\circ p}[r] & X(c) \ar@{=}[d] \ar[r] & 0 \\  
0 \ar[r] &  \cosk_\alpha(Z_{<\alpha})(c) \ar@{=}[d] \ar[r]  & P  \ar^-{\epsilon_{Y}}[d]  \ar^-{\epsilon_{X}}[r] & X(c) \ar^-{m^{X}_{c}}[d] \ar[r] & 0 \\  
0 \ar[r] & \cosk_\alpha(Z_{<\alpha})(c) \ar[r]^-{\iota'}  & \cosk_\alpha(Y_{<\alpha})(c) \ar[r]^-{\pi'} & M_{c}X  \ar[r] & 0. \\
 }
\]
where the dashed morphisms are naturally induced on the kernels. Also, from the Snake Lemma, for instance, we can see that $\psi^{Z}_{c}$ is an epimorphism with $\ker(\psi^{Z}_{c})\cong\ker(p)$, which belongs in $\B$.

We then define $Z(c):= \ker(\epsilon_{X}\circ p)$ and $Y(c):=T_{c}$. 
In a more compact form we have obtained the commutative diagram with exact rows
\[
 \xymatrix@C=2pc{
& \sk_\alpha(Z_{<\alpha})(c) \ar@{-->}[d]^-{j} \ar[r]^-{\iota} & \sk_\alpha(Y_{<\alpha})(c) \ar[r]^-{\pi} \ar^{i}[d] & L_cX  \ar[r] \ar^{l_{c}^{X}}[d]  & 0 \\
0 \ar[r] & Z(c) \ar[r] \ar@{-->}^-{\psi^{Z}_c}[d] & Y(c)  \ar^-{\epsilon_{Y}\circ p}[d]  \ar^-{\epsilon_{X}\circ p}[r] & X(c) \ar[d]^-{m_c^X} \ar[r] & 0 \\  
0 \ar[r] & \cosk_\alpha(Z_{<\alpha})(c) \ar[r]^-{\iota'}  & \cosk_\alpha(Y_{<\alpha})(c) \ar[r]^-{\pi'} & M_{c}X  \ar[r] & 0. \\
 }
\]

From Proposition~\ref{prop:Riehl_Veriti} we deduce that the construction given above defines functors $Z\colon \C\to \M$ and $Y\colon\C\to\M$ that fit into a short exact sequence 
\begin{equation}
\label{eq:obtained_ses}
0\to Z\to Y\to X\to 0.
\end{equation}
Note that $Y$ has $i$ as its latching morphism at $c$, which is a monomorphism with cokernel in $\A$, and that $Z$ 
has $\psi^Z_c$ as its matching morphism at $c$, which is an epimorphism with kernel in $\B$, by construction. By Definition~\ref{def:latching_matching}, the latching morphisms of $Y$ and the matching morphisms of $Z$ at objects $d\in\C_{<\alpha}$ agree with those of $Y_{<\alpha}$ and $Z_{<\alpha}$, respectively.
Hence $Y$ is in $\Phi_{\C}(\A)$ and $Z$ is in $\Psi_{\C}(\B)$ and~\eqref{eq:obtained_ses} is a sequence of the form that we were looking for.

Finally, we treat the case where $\alpha$ is a limit ordinal. Note that given $X\in\M^\C$, the various approximations $0\rightarrow Z_{<\alpha'}\rightarrow Y_{<\alpha'}\rightarrow X_{\alpha'}\rightarrow 0$ given for all $\alpha'<\alpha$ are compatible and fit together to an approximation $0\rightarrow Z_{<\alpha}\rightarrow Y_{<\alpha}\rightarrow X_{<\alpha}\rightarrow 0$ of $X_{<\alpha}$ that has the desired properties. 
\end{proof}

\section{The Reedy abelian model structure}
\label{sec:Reedy_model}

In this section we prove that, under Setup~\ref{setup}, hereditary Hovey triples (and thus abelian model structures) lift from $\M$ to $\M^{\C}$. 
Recall that from Theorem \ref{thm:lifting_cot} hereditary and complete cotorsion pairs lift from $\M$ to $\M^{\C}$. In order to apply this to abelian model structures we need the following:

\begin{proposition}
\label{prop:intersection-with-W}
Let $\A,\W\subseteq\M$ be full subcategories. If
\begin{itemize}
\item $\A\cap\W$ is closed under filtrations in $\M$ (i.e.\ $\Filt(\A\cap\W)\subseteq\A\cap\W$), and
\item $\W$ is closed under cokernels of monomorphisms in $\M$,
\end{itemize}
then $\Phi(\class A \cap\class W)=\Phi(\class A)\cap\class W^{\class C}$.
\end{proposition}

\begin{proof}
If $X$ is in $\Phi(\A\cap\W)$, then clearly $X\in\Phi(\A)$. Moreover, $X(c)\in\A\cap\W$ for each $c\in\C$ by the same proof as for Lemma~\ref{lem:Reedy_objects}(i). 
In particular, $X\in\Phi(\A)\cap\W^{\C}$.

Conversely, let $X$ be a functor in $\Phi(\A)\cap\W^{\C}$. We need to prove that $\Delta^c\otimes_{\C}X\in\A\cap\W$ for each $c\in\C$. We will prove that by transfinite induction on $\alpha:=\deg(c)$. If $\alpha=0$, then $\Delta^c\cong\C(-,c)$ and $\Delta^c\otimes_{\C}X\cong X(c)\in\A\cap\W$, as required.
Suppose now that $\alpha>0$ and $\Delta^d\otimes_{\C}X\in\A\cap\W$ whenever $\deg(d)<\alpha$. Then we have a short exact sequence
\[ 0\to \I^\C_{<\alpha}(-,c)\tens{\C}X \to X(c) \to \Delta^c\tens{\C}X \to 0  \]
and $\I^\C_{<\alpha}(-,c)$ is filtered by standard functors of objects of degree $<\alpha$ by Theorem~\ref{thm:proj_by_standards} and Remark~\ref{rem:proj_by_standards}. If we apply the functor $-\otimes_{\C}X\colon \K^{\C\op}\to\M$ to the latter filtration, we see that $\I^\C_{<\alpha}(-,c)\otimes_{\C}X$ is $(\A\cap\W)$--filtered in $\M$, and so $\I^\C_{<\alpha}(-,c)\otimes_{\C}X\in\A\cap\W$ by assumption. Now $\Delta^c\otimes_{\C}X\in\A$ since we assume that $X\in\Phi(\A)$. On the other hand, $X(c)\in\W$ since $X\in\W^\C$, and as $\W$ is closed under cokernels of monomorphisms, we also get that $\Delta^c\otimes_{\C}X\in\W$.
\end{proof}

\begin{theorem}
\label{thm:Reedy_model}
Let $(\A,\W,\B)$ be a hereditary Hovey triple on $\M$. Then there exists a hereditary Hovey triple
$(\Phi(\A),\W^{\C},\Psi(\B))$ on $\M^{\C}$.
\end{theorem}

\begin{proof}
We are given two hereditary and complete cotorsion pairs $(\A\cap\W,\B)$ and $(\A,\W\cap\B)$ in $\M$. From Theorem \ref{thm:lifting_cot} we obtain two  complete and hereditary cotorsion pairs $(\Phi(\A\cap\W),\Psi(\B))$ and $(\Phi(\A),\Psi(\W\cap\B))$ in $\M^{\C}$. Recall that $\A\cap\W$ is closed under filtrations by Proposition~\ref{prop:eklof}, so applying Proposition~\ref{prop:intersection-with-W} to $\A,\W\subseteq\M$ leads to a complete and hereditary cotorsion pair $(\Phi(\A)\cap\W^{\C},\Psi(\B))$ in $\M^\C$. Applying Proposition~\ref{prop:intersection-with-W} to $\B\op,\W\op\subseteq\M\op$ and the opposite $k$--linear Reedy category $\C\op$ results by the same token in a complete and hereditary cotorsion pair $(\Phi(\B\op)\cap(\W\op)^{\C},\Psi(\A\op))$ in $(\C\op,\M\op)\cong(\C,\M)\op$, which is the same datum as a complete and hereditary cotorsion pair $(\Phi(\A),\W^{\C}\cap\Psi(\B))$ in $\M^{\C}$.
Finally, it is easy to see that $\W^{\C}$ is a thick subcategory of $\M^{\C}$.
\end{proof}

\section*{Acknowledgements}
The first-named author is grateful to Steffen Koenig and Teresa Conde for interesting discussions on the relation between Reedy categories and quasi-hereditary algebras, especially for pointing out the connection with exact Borel subalgebras.

\end{document}